\documentclass[final,1p,times]{elsarticle}
\usepackage{amssymb}
\usepackage{amsthm}
\usepackage{amsmath,amssymb,amsopn,amsfonts,mathrsfs,amsbsy,amscd}
\usepackage{longtable}
\usepackage{lineno}

\newcommand{\prs}{\langle \;,\;\rangle}

\newcommand{\too}{\longrightarrow}
\newcommand{\fl}{\mathfrak{L}}
\newcommand{\om}{\omega}

\newcommand{\esp}{\quad\mbox{and}\quad}

\def\br{[\;,\;]}

\newcommand{\G}{\mathfrak{g}}
\newcommand{\g}{\mathfrak{g}}

\newcommand{\h}{{\mathfrak{h}}}

\newcommand{\ad}{{\mathrm{ad}}}

\newcommand{\tr}{{\mathrm{tr}}}
\newcommand{\ric}{{\mathrm{ric}}}

\newcommand{\eb}{\bar{e}}

\newcommand{\B}{{\cal B}}

\newcommand{\di}{\displaystyle}

\newcommand{\na}{\nabla}

\newcommand{\Ric}{\mathrm{Ric}}

\newcommand{\al}{\alpha}

\newcommand{\e}{\epsilon}

\newcommand{\la}{\lambda}

\newtheorem{theo}{Theorem}[section]
\newtheorem{pr}{Proposition}[section]

\newtheorem{exem}{Example}

\newtheorem{conjecture}{Conjecture}
\font\bb=msbm10

\def\B{\hbox{\bb B}}
\def\R{\hbox{\bb R}}

\begin{document}

\begin{frontmatter}
	
	%% Title, authors and addresses
	
	%% use the tnoteref command within \title for footnotes;
	%% use the tnotetext command for the associated footnote;
	%% use the fnref command within \author or \address for footnotes;
	%% use the fntext command for the associated footnote;
	%% use the corref command within \author for corresponding author footnotes;
	%% use the cortext command for the associated footnote;
	%% use the ead command for the email address,
	%% and the form \ead[url] for the home page:
	%%
	%\title{Left invariant para-K\"ahler and hyper-para-K\"ahler structures on Lie groups\tnoteref{label1}}
	%% \tnotetext[label1]{}
	%\author{\corref{cor1}\fnref{label2}}
	
	%% \ead{email address}
	%% \ead[url]{home page}
	%% \fntext[label2]{}
	%% \cortext[cor1]{}
	%% \address{Address\fnref{label3}}
	%\fntext[label3]{This research was conducted within the framework of Action concert\'ee CNRST-CNRS Project SPM04/13.}
	%\linenumbers
	\title{ Left-invariant Codazzi tensors and harmonic curvature on  Lie groups endowed with a left invariant Lorentzian metric }
	\author[label1,label2]{Ilyes Aberaouze, Mohamed Boucetta}
	\address[label1]{Universit\'e Cadi-Ayyad\\
		Facult\'e des sciences et techniques\\
		BP 549 Marrakech Maroc\\e-mail: aberaouzeilyes@gmail.com 
		
	}
	\address[label2]{Universit\'e Cadi-Ayyad\\
		Facult\'e des sciences et techniques\\
		BP 549 Marrakech Maroc\\e-mail: m.boucetta@uca.ac.ma}
	
	% \address[label2]{}
	
	% \address[label3]{}
	
	%\author{}
	
	%\address{}
	
	\begin{abstract} 
		A Lorentzian Lie group is a Lie group endowed with a left invariant Lorentzian metric.
		We study left-invariant Codazzi tensors on Lorentzian Lie groups. We  obtain new results on left-invariant Lorentzian metrics with harmonic curvature and non-parallel Ricci operator.
		
		In contrast to the Riemannian case, the Ricci operator of a let-invariant Lorentzian metric can be of four types: diagonal, of type $\{n-2,z\bar{z}\}$,  of type $\{n,a2\}$ and of type $\{n,a3\}$.
		We first describe Lorentzian Lie algebras with a non-diagonal Codazzi operator and with these descriptions in mind, we study three classes of Lorentzian Lie groups with harmonic curvature.
		Namely, we give a complete description of the Lie algebra of Lorentzian Lie groups having harmonic curvature and where the Ricci operator is non-diagonal and its diagonal part consists of one real eigenvalue $\al$.
		
	\end{abstract}

\end{frontmatter}
%\linenumbers

\section{Introduction}\label{section1}

On a Riemannian manifold, a symmetric $(0, 2)$ tensor field $A$ is called Codazzi if $\na_X(A)(Y,Z)=\na_Y(A)(X,Z)$ for any vector fields $X,Y,Z$ and where $\na$ is the Levi-Civita connection.
The name comes from the fact that the Codazzi equations for a hypersurface in $\R^n$ simply say that the second fundamental form is a Codazzi tensor. Both topological and geometric consequences of the existence of a nontrivial
Codazzi tensor field on a Riemannian manifold have been studied in \cite{biven, der3}.

A pseudo-Riemannian manifold $(M,\prs)$ is said to have {\it harmonic curvature} if its curvature tensor  has a vanishing divergence and
it is known \cite{bour, der1, der2, gray} that a pseudo-Riemannian manifold $(M,\prs)$ has a harmonic curvature if and only if its Ricci tensor $\Ric$ is  Codazzi.  Obviously, any Einstein pseudo-Riemannian metric $(\Ric=\la\prs)$ and more generally any Ricci-parallel metric $(\na(\Ric)=0)$ has harmonic curvature. Thus the  interesting case will be where the curvature is harmonic and the metric is not Ricci-parallel. 

The study of Riemannian manifolds (when the metric is definite positive)  with harmonic curvature is old and goes to \cite{Lich}. 
 One reason for the interest in such manifolds is the observation that a Riemannian manifold has harmonic curvature if and only if the Riemannian connection is a solution for the Yang-Mills equations on the tangent bundle \cite{bour2}.
 
 In \cite{atri}, d'Atri initiated a study of left-invariant Codazzi tensors on Lie groups endowed with a left-invariant Riemannian metric, with the goal of better understanding the harmonic curvature condition in this setting. Along the way, he showed that any left-invariant Riemannian metric with harmonic curvature on a nilpotent Lie group is Ricci-parallel.
 
 Recently, in \cite{Aber}, the authors generalized this result to solvable Lie groups and any Lie group of dimension $\leq6$. They also stated the following conjecture:

\begin{conjecture}\label{conj} Any homogeneous Riemannian manifold $M$ with harmonic curvature is Ricci-parallel.
	\end{conjecture} This conjecture remains open in general. However, it has been verified for a number of special cases:  in dimension four (see \cite{spiro, zaim}),  when $M$ is a sphere or a projective space (see \cite{peng}),  when $M$ is naturally reductive homogeneous (\cite{marshal})  and in some compact Lie groups (\cite{wu}).

On the other hand, there is a paucity of results on pseudo-Riemannian manifolds that are not definite positive and have a Codazzi tensor or harmonic curvature. To our knowledge, there are only a few results in the homogeneous case in dimensions 3 and 4 \cite{Calvaruso, zaim}.

In this paper, we investigate left-invariant Codazzi tensors on Lie groups with a Lorentzian left-invariant metric (of signature $(1,n-1)$). Our goal is to obtain new results on left-invariant Lorentzian metrics with harmonic curvature and non-parallel Ricci operator, keeping in mind that Conjecture \ref{conj} is not true in this context (see \cite{Calvaruso}).

Let us state briefly our main results. Let $(G,g)$ be a Lie group endowed with a left invariant Lorentzian metric. It is evident that all the geometric information of $(G,g)$ is encoded just in the inner product $\prs=g(e)$ and the Lie bracket $\br$ of the Lie algebra $\G$ of $G$. We call $(\G,\br,\prs)$ a Lorentzian Lie algebra. The \emph{ Levi-Civita product} of ${\G}$ which is the restriction of the Levi-Civita connection to left invariant vector fields  is the bilinear map $\mathrm{L}:{\G}\times{\G}\too{\G}$ given by  Koszul's
formula
\begin{eqnarray}\label{levicivitabis}2\langle 
	\mathrm{L}_uv,w\rangle&=&\langle [u,v],w\rangle+\langle [w,u],v\rangle+
	\langle [w,v],u\rangle \quad\mbox{for any}\; u,v,w\in\G.\end{eqnarray}
For any $u,v\in{\G}$, $\mathrm{L}_{u}:{\G}\too{\G}$ is skew-symmetric and $[u,v]=\mathrm{L}_{u}v-\mathrm{L}_{v}u$.
The curvature of $(G,g)$ at the origin is given by
\begin{equation}
\label{curvature}\mathrm{K}(u,v)=\mathrm{L}_{[u,v]}-[\mathrm{L}_{u},\mathrm{L}_{v}] \quad\mbox{for any}\; u,v,\in\G.
\end{equation}
The Ricci curvature $\mathrm{ric}:{\G}\times{\G}\too\R$ and its Ricci operator $\Ric:{\G}\too{\G}$ are given by \begin{equation}\label{ric}\langle  \Ric (u),v\rangle=\mathrm{ric}(u,v)=\mathrm{tr}\left(w\too
\mathrm{K}(u,w)v\right) \quad\mbox{for any}\; u,v,w\in\G.\end{equation}
A left invariant Codazzi operator on $(G,g)$ is entirely determined by  a symmetric operator $A:\G\too\G$ such that 
\begin{equation}\label{codazzibis}
	\langle\mathrm{L}_uv,A(w)\rangle+
	\langle\mathrm{L}_uw,A(v)\rangle=\langle\mathrm{L}_vu,A(w)\rangle+
	\langle\mathrm{L}_vw,A(u)\rangle.
\end{equation} We call $A$ a Codazzi tensor on $\G$. Thus $(G,g)$ has a harmonic curvature if and only if the Ricci operator $\Ric$ is a Codazzi tensor on $\G$.

In the definite positive case, any symmetric operator is diagonal in an orthonormal basis. However, in the Lorentzian case, there are four types of symmetric operators: diagonal, 
$\{n-2,z\bar{z}\}$, $\{n,a2\}$, and $\{n,a3\}$ (see Theorem \ref{4types}). The description of Lorentzian Lie algebras with a diagonal Codazzi operator is similar to the definite positive case.

Our first main result is the description of Lorentzian Lie algebras with a non-diagonal Codazzi operator (see Theorems \ref{mainzz}-\ref{maina3}). This involved a significant amount of computation, as all the cases of Equation \eqref{codazzibis} were explored. With these descriptions in mind, we study three classes of Lorentzian Lie algebras with harmonic curvature. Namely, we give a complete description of Lorentzian Lie algebras having harmonic curvature and where the Ricci operator is non-diagonal and its diagonal part consists of one real eigenvalue $\al$. We also assume that $\al\not=a$ in the cases $\{n,a2\}$ and $\{n,a3\}$. We obtain three theorems (see Theorems \ref{mainbiszz}, \ref{mainbisa2} and \ref{mainbisa3}) which show that these Lorentzian Lie algebras are either the product or semi-direct product of an Einstein Riemannian Lie algebra with a three-dimensional Lorentzian Lie algebra.

The paper is organized as follows.

$\bullet$ Section \ref{section2} expands Equation \eqref{codazzibis} to obtain necessary and sufficient conditions for a Lorentzian Lie algebra to admit a Codazzi operator.

$\bullet$ Section \ref{section3} solves the equations obtained in Section \ref{section2} to characterize Lorentzian Lie algebras having a Codazzi operator of type non-diagonal.

$\bullet$ Section \ref{section4} gives a complete description of Lorentzian Lie algebras whose curvature is harmonic and the Ricci operator is of type $\{n-2,z\bar{z}\}$ with one real Ricci direction.

$\bullet$ Section \ref{section5} (resp. Section \ref{section6}) gives a complete description of Lorentzian Lie algebras whose curvature is harmonic and the Ricci operator is of type $\{n,a2,\al\}$ (resp. $\{n,a3,\al\}$) with $a\neq \al$.

$\bullet$ Section \ref{section7} is an appendix where all the computations used in the paper are presented.

\section{Lorentzian Lie algebras having a symmetric Codazzi operator: preliminary study }\label{section2}

   We start by recalling a well-known result giving the classification of symmetric operators on Lorentzian vector spaces (see \cite{oneil}).

\begin{theo}\label{4types}
	Let $f$ be  a symmetric operator of a Lorentzian vector space $(V, \prs)$ of dimension $n\geq3$. Then there exists a basis $\B$ of $V$ such that the matrices of $f$ and $\prs$ have one of the following forms: 
	\begin{enumerate}
		\item Type $\{\mathrm{diag}\}$: $M(f,\B)=\mathrm{diag}(\alpha_1,\dots, \alpha_n)$,
		$M( \prs,\B) =\mathrm{diag}(1, \dots,1, -1)$.
		
		\item Type $\{n-2, z\bar{z}\}$: $M(f, \B) =\mathrm{diag}(\alpha_1,\dots, \alpha_{n-2})\oplus \begin{pmatrix}
			a&b\\-b&a\end{pmatrix}$, $b \neq0$, $M( \prs, \B) =I_{n-2}\oplus \begin{pmatrix}
			0&1\\1&0\end{pmatrix}$.
		
		\item Type $\{n, a2\}$: $M(f, \B) =\mathrm{diag}(\alpha_1, \dots, \alpha_{n-2})\oplus \begin{pmatrix}
			a&\pm1\\0&a\end{pmatrix}$, $M( \prs, \B) =I_{n-2}\oplus \begin{pmatrix}
			0&1\\1&0\end{pmatrix}$.
		
		\item Type $\{n, a3\}$: $M(f, \B) =\mathrm{diag}(\alpha_1, \dots, \alpha_{n-3})\oplus \begin{pmatrix}
			a&1&0\\0&a&1\\0&0&a\end{pmatrix}$, $M( \prs, \B) =I_{n-3}\oplus \begin{pmatrix}
			0&0&1\\0&1&0\\1&0&0\end{pmatrix}$.
	\end{enumerate}
\end{theo}

A { pseudo-Euclidean Lie algebra} is a real Lie algebra $(\G,\br)$ endowed with a symmetric nondegenerate bilinear form $\prs$. We call it Euclidean (resp. Lorentzian) if $\prs$ is definite positive (resp. of signature $(1,n-1)$). The Levi-Civita product of $(\G,\br,\prs)$ is the bilinear map $\mathrm{L}:{\G}\times{\G}\too{\G}$   given by \eqref{levicivitabis}. The curvature $\mathrm{K}$, the Ricci curvature $\ric$ and the Ricci operator $\Ric$ are given by \eqref{curvature} and \eqref{ric}. We call $(\G,\br,\prs)$ $\al$-Einstein if $\Ric=\al \mathrm{Id}_\G$.

A Codazzi operator on $\G$ is a symmetric operator $A:\G\too\G$ satisfying  \eqref{codazzibis}. 
This is equivalent to
\[ \langle[u,v],A(w)\rangle=\langle\mathrm{L}_vw,A(u)\rangle-\langle\mathrm{L}_uw,A(v)\rangle. \]
Now, by virtue of \eqref{levicivitabis},
\begin{align*}
	\langle\mathrm{L}_vw,A(u)\rangle&=\frac12\left(\langle[v,w],A(u)\rangle+\langle [A(u),v],w\rangle
	+ \langle [A(u),w],v\rangle    \right),\\
	\langle\mathrm{L}_uw,A(v)\rangle&=\frac12\left(\langle[u,w],A(v)\rangle+\langle [A(v),u],w\rangle
	+ \langle [A(v),w],u\rangle    \right).
\end{align*}
Hence, we get the following characterization of Codazzi operators which we will use through this paper. 
\begin{pr}\label{list} $A$ is a Codazzi tensor if and only if
	\begin{equation}\label{co} 2\langle[u,v],A(w)\rangle
		-\langle [A(u),v],w\rangle-\langle [u,A(v)],w\rangle=
		\langle[v,w],A(u)\rangle- \langle [A(v),w],u\rangle
		+ \langle [A(u),w],v\rangle
		-\langle[u,w],A(v)\rangle.
	\end{equation}
	
\end{pr}

We say that $(\G,\br,\prs)$ has a harmonic curvature if $\Ric$ is a Codazzi operator of $\G$.

Let $A$ be a Codazzi operator on a Lorentzian Lie algebra $(\G,\br,\prs)$.  Then, according to Theorem \ref{4types}, $\G=\h\oplus\mathfrak{L}$ where $\h=\h_1\oplus\ldots\oplus\h_r$, $\h_i=\ker(A-\al_i\mathrm{Id}_\G)$ and one of the following situations holds:
\begin{enumerate}
	\item If $A$ is of type $\{\mathrm{diag}\}$ then $\fl=\{0\}$.
	\item If $A$ is of type $\{n-2,z\bar{z}\}$ then $\fl=\mathrm{span}\{e,\eb \}$ with $A(e)=ae-b\eb$,  $A(\eb)=be+a\eb$, $\langle e,e\rangle=\langle\eb,\eb\rangle=0$ and $\langle e,\eb\rangle=1$.
	\item If $A$ is of type $\{n,a2\}$ then $\fl=\mathrm{span}\{e,\eb \}$ with $A(e)=ae$,  $A(\eb)=e+a\eb$, $\langle e,e\rangle=\langle\eb,\eb\rangle=0$ and $\langle e,\eb\rangle=1$.
	\item	If $A$ is of type $\{n,a3\}$ then $\fl=\mathrm{span}\{e,f,\eb \}$ with $A(e)=ae$, $A(f)=e+af$, $A(\eb)=f+a\eb$, $\langle e,e\rangle=\langle\eb,\eb\rangle=\langle e,f\rangle=\langle\eb,f\rangle=0$ and $\langle e,\eb\rangle=\langle f,f\rangle=1$.
\end{enumerate}

So $A$ is a Codazzi operator if and only if all the following conditions hold:
\begin{enumerate}
	\item[$(h)$]  the equation \eqref{co} is satisfied for any $u,v,w\in\h$,
	\item[$(l)$] the equation \eqref{co} is satisfied for  any $u,v,w\in\fl$,
	\item[$(hl0)$] the equation \eqref{co} is satisfied for  any $u,v\in\h$, $w\in\fl$, 
	\item[$(hl1)$] the equation \eqref{co} is satisfied for  any $u,w\in\h$, $v\in\fl$, 
	\item[$(lh0)$] the equation \eqref{co} is satisfied for  any $u,v\in\fl$, $w\in\h$, 
	\item[$(lh1)$] the equation \eqref{co} is satisfied for  any $u,w\in\fl$, $v\in\h$.
\end{enumerate} 

When $A$ is of type diagonal the situation is the same as the case when the metric is definite positive and we get the following proposition in a similar way as  \cite[Theorem 2.1 ]{Aber}.

\begin{pr}\label{diag}  Let $(\G,\br,\prs)$ be a Lorentzian Lie algebra and $A$ a symmetric operator on $\G$ of type $\{\mathrm{Diag}\}$. Put $\G=\h_1\oplus\ldots\oplus\h_r$  with $A_{|\h_i}=\al_i\mathrm{Id}_{\h_i}$. Then $A$ is a Codazzi operator if and only if the two following conditions are satisfied.
	\begin{enumerate}
		\item For any $i\not=j$, $u,v\in\h_i$, $w\in\h_j$, 
		\[ \langle[u,v],w\rangle=0\esp \langle [w,u],v\rangle+\langle [w,v],u\rangle=0. \]
		\item For any $i,j,k$ distinct and for any $u_i\in\h_i$, $u_j\in\h_j$ and $u_k\in\h_k$,
		\[ (\al_i-\al_j)^2\langle [u_j,u_k],u_i\rangle+\left( \al_{{j}}-\al_{{k}} \right) ^{2}
		\langle [u_j,u_i],u_k\rangle=0. \]
	\end{enumerate} 
	
\end{pr}

To find the conditions $(l)$ to $(lh1)$ is a tedious computation which we give in the appendix. The results of this computation and those of Proposition \ref{diag} are summarized in the following propositions.

\begin{pr}\label{zz} Let $(\G,\br,\prs)$ be a Lorentzian Lie algebra and $A$ a symmetric operator on $\G$ of type $\{n-2,z\bar{z}\}$. Put $\G=\h\oplus\fl$ where $\h=\h_1\oplus\ldots\oplus\h_r$, $\fl=\mathrm{span}(e,\eb)$ with $A_{|\h_i}=\al_i\mathrm{Id}_{\h_i}$, $A(e)=ae-b\eb$ and $A(\eb)=be+a\eb$, $b\not=0$. Then $A$ is a Codazzi operator if and only if the four following conditions are satisfied.
	\begin{enumerate}
		\item For any $i\not=j$, $u,v\in\h_i$, $w\in\h_j$, 
		\[ \langle[u,v],w\rangle=0\esp \langle [w,u],v\rangle+\langle [w,v],u\rangle=0. \]
		
		\item For any $i,j,k$ distinct and for any $u_i\in\h_i$, $u_j\in\h_j$ and $u_k\in\h_k$,
		\[ (\al_i-\al_j)^2\langle [u_j,u_k],u_i\rangle+\left( \al_{{j}}-\al_{{k}} \right) ^{2}
		\langle [u_j,u_i],u_k\rangle=0. \]
		\item $[\fl,\fl]\subset\h$,
		\item for any $u_i\in\h_i$, $v_j\in\h_j$:
		\[ \begin{cases}
			(2a-\al_i-\al_j)\langle[u_i,v_j],e\rangle
			-2b\langle[u_i,v_j],\eb\rangle+
			(\al_j-\al_i)\langle[u_i,e],v_j\rangle
			+(\al_j-\al_i)\langle[v_j,e],u_i\rangle=0,\\
			2b\langle[u_i,v_j],e\rangle
			+(2a-\al_i-\al_j)\langle[u_i,v_j],\eb\rangle+
			(\al_j-\al_i)\langle[u_i,\eb],v_j\rangle
			+(\al_j-\al_i)\langle[v_j,\eb],u_i\rangle=0,\\
			(2\al_j-\al_i-a)\langle[u_i,e],v_j\rangle
			+b\langle[u_i,\eb],v_j\rangle	
			+(a-\al_i)\langle[e,v_j],u_i\rangle-b\langle[\eb,v_j],u_i\rangle
			+(a-\al_i)\langle[u_i,v_j],e\rangle-b
			\langle[u_i,v_j],\eb\rangle
			=0,\\
			(2\al_j-\al_i-a)\langle[u_i,\eb],v_j\rangle
			-b\langle[u_i,e],v_j\rangle	
			+(a-\al_i)\langle[\eb,v_j],u_i\rangle+b\langle[e,v_j],u_i\rangle
			+(a-\al_i)\langle[u_i,v_j],\eb\rangle+b
			\langle[u_i,v_j],e\rangle
			=0,\\
			(\al_i-a)\langle[e,\eb],u_i\rangle
			+b\langle[\eb,u_i],\eb\rangle
			+b\langle[e,u_i],e\rangle
			=0,\\
			2(a-\al_i)\langle[e,u_i],e\rangle
			-3b\langle[e,u_i],\eb\rangle
			+b\langle[\eb,u_i],e\rangle-b\langle[e,\eb],u_i\rangle=0,\\
			2(a-\al_i)\langle[\eb,u_i],\eb\rangle
			+3b\langle[\eb,u_i],e\rangle
			-b\langle[e,u_i],\eb\rangle-b\langle[e,\eb],u_i\rangle=0,\\
			2b\langle[e,u_i],e\rangle
			+(a-\al_i)\langle[e,u_i],\eb\rangle+
			(a-\al_i)\langle[\eb,u_i],e\rangle
			+(\al_i-a)\langle[e,\eb],u_i\rangle=0,\\
			(a-\al_i)\langle[\eb,u_i],e\rangle
			-2b\langle[\eb,u_i],\eb\rangle
			+(a-\al_i)\langle[e,u_i],\eb\rangle
			+(a-\al_i)\langle[e,\eb],u_i\rangle=0.
		\end{cases} \]
	\end{enumerate}
	
\end{pr}

\begin{pr}\label{a2} Let $(\G,\br,\prs)$ be a Lorentzian Lie algebra and $A$ a symmetric operator on $\G$ of type $\{n,a2\}$. Put $\G=\h\oplus\fl$ where $\h=\h_1\oplus\ldots\oplus\h_r$, $\fl=\mathrm{span}(e,\eb)$ with $A_{|\h_i}=\al_i\mathrm{Id}_{\h_i}$, $A(e)=ae$ and $A(\eb)=e+a\eb$. Then $A$ is a Codazzi operator if and only if the four following conditions are satisfied.
	\begin{enumerate}
		\item For any $i\not=j$, $u,v\in\h_i$, $w\in\h_j$, 
		\[ \langle[u,v],w\rangle=0\esp \langle [w,u],v\rangle+\langle [w,v],u\rangle=0. \]
		\item For any $i,j,k$ distinct and for any $u_i\in\h_i$, $u_j\in\h_j$ and $u_k\in\h_k$,
		\[ (\al_i-\al_j)^2\langle [u_j,u_k],u_i\rangle+\left( \al_{{j}}-\al_{{k}} \right) ^{2}
		\langle [u_j,u_i],u_k\rangle=0. \]
		\item $\langle[e,\eb],e\rangle=0$,
		\item for any $u_i\in\h_i$, $v_j\in\h_j$:
		\[ \begin{cases}
			(2a-\al_i-\al_j)\langle[u_i,v_j],e\rangle
			+
			(\al_j-\al_i)\langle[u_i,e],v_j\rangle
			+(\al_j-\al_i)\langle[v_j,e],u_i\rangle=0,\\
			2\langle[u_i,v_j],e\rangle
			+(2a-\al_i-\al_j)\langle[u_i,v_j],\eb\rangle+
			(\al_j-\al_i)\langle[u_i,\eb],v_j\rangle
			+(\al_j-\al_i)\langle[v_j,\eb],u_i\rangle=0,\\
			(2\al_j-\al_i-a)\langle[u_i,e],v_j\rangle+(a-\al_i)\langle[e,v_j],u_i\rangle+
			(a-\al_i)\langle[u_i,v_j],e\rangle=0,\\
			-\langle[u_i,e],v_j\rangle-
			\langle[v_j,e],u_i\rangle+ (2\al_j-\al_i-a)\langle[u_i,\eb],v_j\rangle+(a-\al_i)\langle[\eb,v_j],u_i\rangle+
			(a-\al_i)\langle[u_i,v_j],\eb\rangle
			+\langle[u_i,v_j],e\rangle=0,\\
			(2\al_i-2a)\langle[e,\eb],u_i\rangle
			+2\langle[e,u_i],e\rangle
			=0,\\
			2(a-\al_i)\langle[e,u_i],e\rangle=0,\\
			3\langle[\eb,u_i],e\rangle
			+2(a-\al_i)\langle[\eb,u_i],\eb\rangle-\langle[e,u_i],\eb\rangle
			-\langle [e,\eb],u_i\rangle=0,\\
			
			2\langle[e,u_i],e\rangle +(a-\al_i)\langle[e,u_i],\eb\rangle 
			+(a-\al_i)\langle[\eb,u_i],e\rangle 
			+(\al_i-a)\langle[e,\eb],u_i\rangle=0,\\
			(a-\al_i)\langle[\eb,u_i],e\rangle
			+(a-\al_i)\langle[e,u_i],\eb\rangle
			+(\al_i-a)\langle[\eb,e],u_i\rangle=0.
		\end{cases} \]
	\end{enumerate}
	
\end{pr}

\begin{pr}\label{a3} Let $(\G,\br,\prs)$ be a Lorentzian Lie algebra and $A$ a symmetric operator on $\G$ of type $\{n,a3\}$. Put $\G=\h\oplus\fl$ where $\h=\h_1\oplus\ldots\oplus\h_r$, $\fl=\mathrm{span}(e,f,\eb)$ with $A_{|\h_i}=\al_i\mathrm{Id}_{\h_i}$, $A(e)=ae$,
	$A(f)=e+af$ and $A(\eb)=f+a\eb$. Then $A$ is a Codazzi operator if and only if the four following conditions are satisfied.
	\begin{enumerate}
		\item For any $i\not=j$, $u,v\in\h_i$, $w\in\h_j$, 
		\[ \langle[u,v],w\rangle=0\esp \langle [w,u],v\rangle+\langle [w,v],u\rangle=0. \]
		\item For any $i,j,k$ distinct and for any $u_i\in\h_i$, $u_j\in\h_k$ and $u_k\in\h_k$,
		\[ (\al_i-\al_j)^2\langle [u_j,u_k],u_i\rangle+\left( \al_{{j}}-\al_{{k}} \right) ^{2}
		\langle [u_j,u_i],u_k\rangle=0. \]
		\item \[ 
		\begin{cases}
			\langle[e,f],e\rangle=\langle[f,e],f\rangle=\langle[e,\eb],e\rangle=0,\\
			3\langle[e,\eb],f\rangle
			-\langle[e,f],\eb\rangle
			+\langle[f,\eb],e\rangle=0,\\
			2\langle[\eb,f],f\rangle 
			-\langle[\eb,e],\eb\rangle=0,\\
			3\langle[\eb,f],e\rangle
			-\langle[\eb,e],f\rangle+\langle[e,f],\eb\rangle=0.
		\end{cases} \]
		\item for any $u_i\in\h_i$, $v_j\in\h_j$:
		\[ \begin{cases}
			(2a-\al_i-\al_j)\langle[u_i,v_j],e\rangle
			+
			(\al_j-\al_i)\langle[u_i,e],v_j\rangle
			+(\al_j-\al_i)\langle[v_j,e],u_i\rangle=0,\\
			2\langle[u_i,v_j],e\rangle+ (2a-\al_i-\al_j)\langle[u_i,v_j],f\rangle
			+
			(\al_j-\al_i)\langle[u_i,f],v_j\rangle
			+(\al_j-\al_i)\langle[v_j,f],u_i\rangle=0,\\
			2\langle[u_i,v_j],f\rangle+ (2a-\al_i-\al_j)\langle[u_i,v_j],\eb\rangle
			+
			(\al_j-\al_i)\langle[u_i,\eb],v_j\rangle
			+(\al_j-\al_i)\langle[v_j,\eb],u_i\rangle=0,\\
			(2\al_j-\al_i-a)\langle[u_i,e],v_j\rangle+(a-\al_i)\langle[e,v_j],u_i\rangle+
			(a-\al_i)\langle[u_i,v_j],e\rangle=0,\\
			-\langle[u_i,e],v_j\rangle-\langle[v_j,e],u_i\rangle+ (2\al_j-\al_i-a)\langle[u_i,f],v_j\rangle+(a-\al_i)\langle[f,v_j],u_i\rangle+
			(a-\al_i)\langle[u_i,v_j],f\rangle
			+\langle[u_i,v_j],e\rangle=0,\\
			-\langle[u_i,f],v_j\rangle-\langle[v_j,f],u_i\rangle+ (2\al_j-\al_i-a)\langle[u_i,\eb],v_j\rangle+(a-\al_i)\langle[\eb,v_j],u_i\rangle+
			(a-\al_i)\langle[u_i,v_j],\eb\rangle
			+\langle[u_i,v_j],f\rangle=0,\\
			(2\al_i-2a)\langle[e,f],u_i\rangle
			+2\langle[e,u_i],e\rangle
			=0,\\
			(2\al_i-2a)\langle[e,\eb],u_i\rangle
			-\langle[e,f],u_i\rangle+\langle[f,u_i],e\rangle+\langle[e,u_i],f\rangle
			=0,\\
			(2\al_i-2a)\langle[f,\eb],u_i\rangle
			-\langle[e,\eb],u_i\rangle-
			\langle[\eb,u_i],e\rangle
			-\langle[e,u_i],\eb\rangle+2\langle[f,u_i],f\rangle
			=0,\\
			2(a-\al_i)\langle[e,u_i],e\rangle=0,\\
			3\langle[f,u_i],e\rangle +
			2(a-\al_i)\langle[f,u_i],f\rangle -\langle[e,u_i],f\rangle
			-\langle[e,f],u_i\rangle=0,\\
			3\langle[\eb,u_i],f\rangle
			+2(a-\al_i)\langle[\eb,u_i],\eb\rangle
			-\langle [f,u_i],\eb\rangle
			-\langle [f,\eb],u_i\rangle=0,\\
			2\langle[e,u_i],e\rangle
			+(a-\al_i)\langle[e,u_i],f\rangle
			+(a-\al_i)\langle[f,u_i],e\rangle
			+(\al_i-a)\langle[e,f],u_i\rangle=0,\\
			(a-\al_i)\langle[f,u_i],e\rangle
			+(a-\al_i)\langle[e,u_i],f\rangle
			+(a-\al_i)\langle[e,f],u_i\rangle=0,\\	
			2\langle[e,u_i],f\rangle
			+(a-\al_i)\langle[e,u_i],\eb\rangle
			+(a-\al_i)\langle[\eb,u_i],e\rangle
			+(a-\al_i)\langle[\eb,e],u_i\rangle=0,\\
			(a-\al_i)\langle[\eb,u_i],e\rangle
			+(a-\al_i)\langle[e,u_i],\eb\rangle
			-\langle [f,u_i],e\rangle
			+\langle[e,u_i],f\rangle+\langle[e,f],u_i\rangle
			+(a-\al_i)\langle[e,\eb],u_i\rangle	=0,\\
			2\langle[f,u_i],f\rangle
			+(a-\al_i)\langle[f,u_i],\eb\rangle
			-\langle [e,u_i],\eb\rangle
			+\langle[\eb,u_i],e\rangle
			+(a-\al_i)\langle[\eb,u_i],f\rangle
			+(\al_i-a)\langle[f,\eb],u_i\rangle
			-\langle[e,\eb],u_i\rangle=0,\\
			2\langle[\eb,u_i],e\rangle
			+(a-\al_i)\langle[\eb,u_i],f\rangle
			+(a-\al_i)\langle[f,u_i],\eb\rangle
			+(a-\al_i)\langle[f,\eb],u_i\rangle=0.
		\end{cases} \]
	\end{enumerate}
	
\end{pr}

\section{ Characterization of Lorentzian Lie algebras having a symmetric Codazzi operator of type non diagonal}\label{section3}

In this section, we investigate more deeply the conditions obtained in Propositions \ref{zz}-\ref{a3} to get a more precise characterization of Lorentzian Lie algebras having a symmetric Codazzi operator of type non diagonal. 

\subsection{Codazzi operator of type $\{n-2,z\bar{z}\}$}

In this subsection, we will prove the following theorem using Proposition \ref{zz}.

\begin{theo}\label{mainzz} Let $(\G,\br,\prs)$ be a Lorentzian Lie algebra and $A$ a symmetric operator on $\G$ of type $\{n-2,z\bar{z}\}$. Put $\G=\h\oplus\fl$ where $\h=\h_1\oplus\ldots\oplus\h_r$, $\fl=\mathrm{span}(e,\eb)$ with $A_{|\h_i}=\al_i\mathrm{Id}_{\h_i}$, $A(e)=ae-b\eb$ and $A(\eb)=be+a\eb$, $b\not=0$. Put $a_i=a-\al_i$. Then $A$ is a Codazzi operator if and only if:
	\begin{enumerate}
		\item For any $i\not=j$, $u,v\in\h_i$, $w\in\h_j$, 
		\[ \langle[u,v],w\rangle=0\esp \langle [w,u],v\rangle+\langle [w,v],u\rangle=0. \]
		
		\item For any $i,j,k$ distinct and for any $u_i\in\h_i$, $u_j\in\h_j$ and $u_k\in\h_k$,
		\[ (\al_i-\al_j)^2\langle [u_j,u_k],u_i\rangle+\left( \al_{{j}}-\al_{{k}} \right) ^{2}
		\langle [u_j,u_i],u_k\rangle=0. \]
		\item $[\fl,\fl]\subset\h$,
		\item For any $i$ and $u_i,v_i\in\h_i$,
		\[ \begin{cases}
			\langle[u_i,v_i],e\rangle=\langle[u_i,v_i],\eb\rangle=0,\\
			\langle[e,u_i],v_i\rangle=-
			\langle[e,v_i],u_i\rangle,\;
			\langle[\eb,u_i],v_i\rangle=-
			\langle[\eb,v_i],u_i\rangle,\\
			\langle[e,u_i],e\rangle =\langle[\eb,u_i],\eb\rangle= \frac{a_i }{2 b}\langle[e,\eb],u_i\rangle,
			\langle[e,u_i],\eb\rangle = -\langle[\eb,u_i],e\rangle=
			\frac{\left(a_i^{2}-b^{2}\right) }{4 b^{2}}\langle[e,\eb],u_i\rangle. \end{cases} \]
		\item For $i\not=j$ and $u_i\in\h_i$ and $v_j\in\h_j$,
		\[ \begin{cases}
			\langle[u_i,e],v_j\rangle=
			\frac1{(\al_i-\al_j)^2}\left( (b^2-a_i^2)\langle[u_i,v_j],e\rangle+2ba_i \langle[u_i,v_j],\eb\rangle   \right),\\	
			\langle[u_i,\eb],v_j\rangle=
			\frac1{(\al_i-\al_j)^2}\left( -2ba_i\langle[u_i,v_j],e\rangle+(b^2-a_i^2) \langle[u_i,v_j],\eb\rangle   \right).	
		\end{cases} \]

	\end{enumerate}
	
\end{theo}

\begin{proof} To get the desired result, we need to solve the system, say (S), obtained in the item 4 in Proposition \ref{zz}.
		Put $X=\langle[e,u_i],e\rangle$, $Y=\langle[e,u_i],\eb\rangle$, $Z=\langle[\eb,u_i],e\rangle$, $T=\langle[\eb,u_i],\eb\rangle$, $U=\langle[e,\eb],u_i\rangle$ and $a_i=a-\al_i$. The last five equations of (S) become
	\[ \begin{cases}
		bX+bT-a_iU=0,\\
		2a_iX-3bY+bZ-bU=0,\\
		-bY+3bZ+2a_iT-bU=0,\\
		2bX+a_iY+a_iZ-a_iU=0,\\
		a_iY+a_iZ-2bT+a_iU=0.
	\end{cases} \]
	The solutions of this linear system are
	\[  X = \frac{a_iU }{2 b}, Y = 
	\frac{\left(a_i^{2}-b^{2}\right) U}{4 b^{2}}, Z = 
	-\frac{\left(a_i^{2}-b^{2}\right) U}{4 b^{2}},T = \frac{a_iU }{2 b}.
	\]
	For $i=j$, we get from the first four equations of (S):
	\[ \begin{cases}
		2a_i\langle[u_i,v_i],e\rangle
		-2b\langle[u_i,v_i],\eb\rangle=0,\\
		2b\langle[u_i,v_i],e\rangle
		+2a_i\langle[u_i,v_i],\eb\rangle=0,\\
		-a_i\langle[u_i,e],v_i\rangle
		+b\langle[u_i,\eb],v_i\rangle	
		+a_i\langle[e,v_i],u_i\rangle-
		b\langle[\eb,v_i],u_i\rangle
		+a_i\langle[u_i,v_i],e\rangle-b
		\langle[u_i,v_i],\eb\rangle
		=0,\\
		-a_i\langle[u_i,\eb],v_i\rangle
		-b\langle[u_i,e],v_i\rangle	
		+a_i\langle[\eb,v_i],u_i\rangle+
		b\langle[e,v_i],u_i\rangle
		+a_i\langle[u_i,v_i],\eb\rangle+b
		\langle[u_i,v_i],e\rangle
		=0.
	\end{cases} \]
	Since $b\not=0$ this is equivalent to
	\[ \begin{cases}
		\langle[u_i,v_i],e\rangle
		=\langle[u_i,v_i],\eb\rangle=0,\\
		b(\langle[u_i,\eb],v_i\rangle-\langle[\eb,v_i],u_i\rangle)	
		+a_i(\langle[e,v_i],u_i\rangle-\langle[u_i,e],v_i\rangle)
		=0,\\	
		a_i(\langle[\eb,v_i],u_i\rangle-\langle[u_i,\eb],v_i\rangle)+
		b(\langle[e,v_i],u_i\rangle-\langle[u_i,e],v_i\rangle)
		=0.
	\end{cases} \]
	Since $b\not=0$, we get
	\[\langle[u_i,v_i],e\rangle
	=\langle[u_i,v_i],\eb\rangle= \langle[e,u_i],v_i\rangle+\langle[e,v_i],u_i\rangle=\langle[\eb,u_i],v_i\rangle+\langle[\eb,v_i],u_i\rangle=0. \]
	
	For $i\not=j$, we get
	\[ \begin{cases}
		(2a-\al_i-\al_j)\langle[u_i,v_j],e\rangle
		-2b\langle[u_i,v_j],\eb\rangle+
		(\al_j-\al_i)\langle[u_i,e],v_j\rangle
		+(\al_j-\al_i)\langle[v_j,e],u_i\rangle=0,\\
		2b\langle[u_i,v_j],e\rangle
		+(2a-\al_i-\al_j)\langle[u_i,v_j],\eb\rangle+
		(\al_j-\al_i)\langle[u_i,\eb],v_j\rangle
		+(\al_j-\al_i)\langle[v_j,\eb],u_i\rangle=0,\\
		(2\al_j-\al_i-a)\langle[u_i,e],v_j\rangle
		+b\langle[u_i,\eb],v_j\rangle	
		+a_i\langle[e,v_j],u_i\rangle-b\langle[\eb,v_j],u_i\rangle
		+(a-\al_i)\langle[u_i,v_j],e\rangle-b
		\langle[u_i,v_j],\eb\rangle
		=0,\\
		(2\al_j-\al_i-a)\langle[u_i,\eb],v_j\rangle
		-b\langle[u_i,e],v_j\rangle	
		+a_i\langle[\eb,v_j],u_i\rangle+b\langle[e,v_j],u_i\rangle
		+a_i\langle[u_i,v_j],\eb\rangle+b
		\langle[u_i,v_j],e\rangle
		=0,\\
		(2\al_i-\al_j-a)\langle[v_j,e],u_i\rangle
		+b\langle[v_j,\eb],u_i\rangle	
		+a_j\langle[e,u_i],v_j\rangle
		-b\langle[\eb,u_i],v_j\rangle
		+a_j\langle[v_j,u_i],e\rangle-b
		\langle[v_j,u_i],\eb\rangle
		=0,\\
		(2\al_i-\al_j-a)\langle[v_j,\eb],u_i\rangle
		-b\langle[v_j,e],u_i\rangle	
		+a_j\langle[\eb,u_i],v_j\rangle+
		b\langle[e,u_i],v_j\rangle
		+a_j\langle[v_j,u_i],\eb\rangle+b
		\langle[v_j,u_i],e\rangle
		=0,\\
	\end{cases} \]
	Put $X=\langle[u_i,v_j],e\rangle$, $Y=\langle[u_i,v_j],\eb\rangle$, $Z=\langle[u_i,e],v_j\rangle$, $T=\langle[u_i,\eb],v_j\rangle$, $U=\langle[v_j,e],u_i\rangle$, $V=\langle[v_j,\eb],u_i\rangle.$
	So
	\[ \begin{cases}
		(2a-\al_i-\al_j)X-2bY+(\al_j-\al_i)Z+(\al_j-\al_i)U=0,\\
		2bX+(2a-\al_i-\al_j)Y+(\al_j-\al_i)T+(\al_j-\al_i)V=0,\\
		(2\al_j-\al_i-a)Z+bT	-a_iU+bV+a_iX-b	Y=0,\\
		(2\al_j-\al_i-a)T
		-bZ-	a_iV-bU
		+a_iY+bX=0,\\
		(2\al_i-\al_j-a)U
		+bV	-a_jZ+bT
		-a_jX+bY
		=0,\\
		(2\al_i-\al_j-a)V
		-bU	-a_jT-
		bZ
		-a_jY-b
		X
		=0,\\
	\end{cases} \]
	The solutions of this system are
	\[ T = 
	\frac{-2a_ib X +(b^2-a_i^2)Y}{\left(\al_{i}-\al_{j}\right)^{2}}
	, U = 
	\frac{ (a_{j}^{2}-b^2)X-2a_jb Y  }{\left(\al_{i}-\al_{j}\right)^{2}}
	, V = 
	\frac{2a_jb X +(a_j^2-b^2)Y }{\left(\al_{i}-\al_{j}\right)^{2}}
	, Z = 
	\frac{(b^2-a_i^2)X +2a_ib Y  }{\left(\al_{i}-\al_{j}\right)^{2}}.
	\]This completes the proof.
	\end{proof}

\subsection{Codazzi operator of type $\{n,a2\}$}

In this subsection, we will prove the following theorem using Proposition \ref{a2}.

\begin{theo}\label{maina2} Let $(\G,\br,\prs)$ be a Lorentzian Lie algebra and $A$ a symmetric operator on $\G$ of type $\{n,n2\}$. Put $\G=\h\oplus\fl$ where $\h=\h_1\oplus\ldots\oplus\h_r$, $\fl=\mathrm{span}(e,\eb)$ with $A_{|\h_i}=\al_i\mathrm{Id}_{\h_i}$, $A(e)=ae$ and $A(\eb)=e+a\eb$.  Put $a_i=a-\al_i$. Then $A$ is a Codazzi operator if and only if:
	\begin{enumerate}
		\item For any $i\not=j$, $u,v\in\h_i$, $w\in\h_j$, 
		\[ \langle[u,v],w\rangle=0\esp \langle [w,u],v\rangle+\langle [w,v],u\rangle=0. \]
		
		\item For any $i,j,k$ distinct and for any $u_i\in\h_i$, $u_j\in\h_j$ and $u_k\in\h_k$,
		\[ (\al_i-\al_j)^2\langle [u_j,u_k],u_i\rangle+\left( \al_{{j}}-\al_{{k}} \right) ^{2}
		\langle [u_j,u_i],u_k\rangle=0. \]
		\item $\langle[e,\eb],e\rangle=0$,
		\item For any $u_i,v_i\in\h_i$,
		\begin{enumerate}
			\item If $a_i=0$, then
			\[ \langle[e,u_i],e\rangle=\langle[u_i,v_i],e\rangle=0,\langle[e,\eb],u_i\rangle=3\langle[\eb,u_i],e\rangle-\langle[e,u_i],\eb\rangle
			\esp \langle[e,u_i],v_i\rangle=-
			\langle[e,v_i],u_i\rangle. \]
			\item If $a_i\not=0$ then
			\[ \begin{cases}
				\langle[e,u_i],e\rangle=\langle[e,\eb],u_i\rangle=
				\langle[u_i,v_i],e\rangle=
				\langle[u_i,v_i],\eb\rangle=0,\\
				\langle[u_i,\eb],e\rangle=-\langle[u_i,e],\eb\rangle,\\
				\langle[\eb,u_i],\eb\rangle=\frac{2}{a_i}
				\langle[e,u_i],\eb\rangle,\\
				\langle[e,u_i],v_i\rangle=-
				\langle[e,v_i],u_i\rangle,\\
				\langle[\eb,u_i],v_i\rangle=-
				\langle[\eb,v_i],u_i\rangle.
			\end{cases} \]

		\end{enumerate}

		\item For $i\not=j$ and $u_i\in\h_i$ and $v_j\in\h_j$,
		\[ 	\langle[e,u_i],v_j\rangle=\frac{a_i^2}{(\al_i-\al_j)^2}\langle[u_i,v_j],e\rangle,\quad
		\langle[\eb,u_i],v_j\rangle=
		\frac{a_i}{(\al_i-\al_j)^2}\left(
		2\langle[u_i,v_j],e\rangle+a_i
		\langle[u_i,v_j],\eb\rangle
		\right).\]

	\end{enumerate}
	
\end{theo}

\begin{proof} To get the desired result, we need to solve the system, say (S), obtained in the item 4 in Proposition \ref{a2}.
	Let us solve the last five equations of (S) first. We distinguish two cases:\\
	$\bullet$ $\al_i\not=a$. Then 
	\[ 
	\langle[e,u_i],e\rangle=\langle[e,\eb],u_i\rangle=0,\;
	\langle[u_i,\eb],e\rangle=-\langle[u_i,e],\eb\rangle\esp
	\langle[\eb,u_i],\eb\rangle=\frac2{a_i}\langle[e,u_i],\eb\rangle.
	\]
	$\bullet$ $a=\al_i$. Then $\langle[e,u_i],e\rangle=0$ and $\langle[e,\eb],u_i\rangle=3\langle[\eb,u_i],e\rangle-\langle[e,u_i],\eb\rangle.$
	
	Now, let us solve the first four equations of $(S)$. We distinguish two cases:\\
	$\bullet$ $i=j$. Then
	\[ \begin{cases}
		2a_i\langle[u_i,v_i],e\rangle
		=0,\\
		2\langle[u_i,v_i],e\rangle
		+2a_i\langle[u_i,v_i],\eb\rangle+
		=0,\\
		-a_i\langle[u_i,e],v_i\rangle+a_i
		\langle[e,v_i],u_i\rangle+
		a_i\langle[u_i,v_i],e\rangle=0,\\
		-\langle[u_i,e],v_i\rangle-
		\langle[v_i,e],u_i\rangle- a_i\langle[u_i,\eb],v_i\rangle+a_i\langle[\eb,v_i],u_i\rangle+
		a_i\langle[u_i,v_i],\eb\rangle
		+\langle[u_i,v_i],e\rangle=0.
		\end{cases} \]
	If $a_i=0$ then $\langle[u_i,v_i],e\rangle=0$ and
	$\langle[u_i,e],v_i\rangle=-\langle[v_i,e],u_i\rangle$.
	\\ If $a_i\not=0$ then $\langle[u_i,v_i],e\rangle=
	\langle[u_i,v_i],\eb\rangle=0$, 
	$\langle[e,u_i],v_i\rangle=-
	\langle[e,v_i],u_i\rangle$ and
	$\langle[\eb,u_i],v_i\rangle=-
	\langle[\eb,v_i],u_i\rangle$.
	
	$\bullet$ $i\not=j$. 
	Put $X=\langle[u_i,e],v_j\rangle$, $Y=\langle[v_j,e],u_i\rangle$,  $Z=
	\langle[u_i,v_j],e\rangle$, $U=\langle[u_i,\eb],v_j\rangle$, $V=\langle[v_j,\eb],u_i\rangle$ and $T=\langle[u_i,v_j],\eb\rangle$. Then the system can written
	\[ \begin{cases}
		(\al_j-\al_i)X+(\al_j-\al_i)Y+(2a-\al_i-\al_j)Z=0,\\
		(\al_j-\al_i)U+(\al_j-\al_i)V+
		2Z
		+(2a-\al_i-\al_j)T=0,\\
		(2\al_j-\al_i-a)X-(a-\al_i)Y+(a-\al_i)Z=0,\\
		-X-Y+(2\al_j-\al_i-a)U+(a-\al_i)V+(a-\al_i)T+Z=0.
	\end{cases} \]
	The solutions of this system are
	\[ U = 
	-\frac{a_i \left(a_i T +2 Z \right)}{\left(\alpha_{i}-\alpha_{j}\right)^{2}}
	, V = 
	\frac{a_j \left(a_j T+2 Z \right)}{\left(\alpha_{i}-\alpha_{j}\right)^{2}}
	, X = 
	-\frac{a_i^2Z }{\left(\alpha_{i}-\alpha_{j}\right)^{2}}
	, Y = 
	\frac{a_j^2 Z}{\left(\alpha_{i}-\alpha_{j}\right)^{2}}.
	\]This completes the proof.
	\end{proof}

\subsection{Codazzi operator of type $\{n,a3\}$}

In this subsection, we will prove the following theorem using Proposition \ref{a3}.

\begin{theo}\label{maina3} Let $(\G,\br,\prs)$ be a Lorentzian Lie algebra and $A$ a symmetric operator on $\G$ of type $\{n,3a\}$. Put $\G=\h\oplus\fl$ where $\h=\h_1\oplus\ldots\oplus\h_r$, $\fl=\mathrm{span}(e,f,\eb)$ with $A_{|\h_i}=\al_i\mathrm{Id}_{\h_i}$, $A(e)=ae$, $A(f)=f+af$ and $A(\eb)=f+a\eb$. Put $a_i=a-\al_i$. Then $A$ is a Codazzi operator if and only if:
	\begin{enumerate}
		\item For any $i\not=j$, $u,v\in\h_i$, $w\in\h_j$, 
		\[ \langle[u,v],w\rangle=0\esp \langle [w,u],v\rangle+\langle [w,v],u\rangle=0. \]
		\item For any $i,j,k$ distinct and for any $u_i\in\h_i$, $u_j\in\h_j$ and $u_k\in\h_k$,
		\[ (\al_i-\al_j)^2\langle [u_j,u_k],u_i\rangle+\left( \al_{{j}}-\al_{{k}} \right) ^{2}
		\langle [u_j,u_i],u_k\rangle=0. \]
		\item \[ 
		\begin{cases}
			\langle[e,f],e\rangle=\langle[f,e],f\rangle=\langle[e,\eb],e\rangle=0,\\
			\langle[\eb,f],e\rangle=-2\langle[e,\eb],f\rangle,\\
			\langle[\eb,f],f\rangle =-\frac12
			\langle[e,\eb],\eb\rangle,\\
			\langle[e,f],\eb\rangle=5\langle[e,\eb],f\rangle
			.
		\end{cases} \]
		\item For any $u_i,v_i\in\h_i$,
		\begin{enumerate}
			\item If $a_i=0$, then
			\[ \begin{cases}
				\langle[e,u_i],e\rangle=\langle[e,u_i],f\rangle=\langle[\eb,u_i],e\rangle=
				\langle[e,f],u_i\rangle=\langle[f,u_i],e\rangle
				=0,\\
				-\langle[e,\eb],u_i\rangle
				-\langle[e,u_i],\eb\rangle+2\langle[f,u_i],f\rangle
				=0,\\
				3\langle[\eb,u_i],f\rangle
				-\langle [f,u_i],\eb\rangle
				-\langle [f,\eb],u_i\rangle=0,\\
				\langle[u_i,v_i],e\rangle=\langle[u_i,v_i],f\rangle=0,\langle[u_i,e],v_i\rangle+\langle[v_i,e],u_i\rangle=\langle[u_i,f],v_i\rangle+\langle[v_i,f],u_i\rangle=0.
			\end{cases} \]

			\item If $a_i\not=0$ then
			\[ \begin{cases}
				\langle[u_i,v_i],e\rangle=\langle[u_i,v_i],f\rangle=\langle[u_i,v_i],\eb\rangle
				=0,\\
				\langle[e,u_i],v_i\rangle
				+\langle[e,v_i],u_i\rangle=0,\\
				\langle[f,u_i],v_i\rangle+
				\langle[f,v_i],u_i\rangle
				=0,\\
				\langle[\eb,u_i],v_i\rangle
				+\langle[\eb,v_i],u_i\rangle=0,\\
				\langle[e,f],u_i\rangle=
				\langle[e,u_i],e\rangle=\langle[e,\eb],u_i\rangle
				=0,\\		
				\langle[e,u_i],f\rangle=-
				\langle[f,u_i],e\rangle=
				\frac{a_i^2}3
				\langle[f,\eb],u_i\rangle,\quad
				\langle[f,u_i],f\rangle=\frac{2a_i}3
				\langle[f,\eb],u_i\rangle,\\
				\langle[e,u_i],\eb\rangle=
				a_i(a_i\langle[\eb,u_i],\eb\rangle
				+2
				\langle[\eb,u_i],f\rangle-\frac23\langle[f,\eb],u_i\rangle),\\
				\langle[\eb,u_i],e\rangle=
				-a_i(a_i\langle[\eb,u_i],\eb\rangle
				+2
				\langle[\eb,u_i],f\rangle),\\
				\langle[f,u_i],\eb\rangle
				=2a_i\langle[\eb,u_i],\eb\rangle
				+3
				\langle[\eb,u_i],f\rangle-\langle[f,\eb],u_i\rangle.
			\end{cases} \]
			
		\end{enumerate}

		\item For $i\not=j$ and $u_i\in\h_i$ and $v_j\in\h_j$,
		\[ \begin{cases}\di
			\langle[u_i,e],v_j\rangle=-\frac{a_i^2}{(\al_i-\al_j)^2}\langle[u_i,v_j],e\rangle,\\\di
			\langle[u_i,f],v_j\rangle=
			-\frac{a_i}{(\al_i-\al_j)^2}\left(
			a_i\langle[u_i,v_j],f\rangle+2
			\langle[u_i,v_j],\eb\rangle\right),\quad\\\di
			\langle[u_i,\eb],v_j\rangle=
			-\frac{1}{(\al_i-\al_j)^2}\left(
			a_i^2\langle[u_i,v_j],\eb\rangle+2
			a_i\langle[u_i,v_j],f\rangle
			+\langle[u_i,v_j],e\rangle\right).
		\end{cases}		\]

	\end{enumerate}
	
\end{theo}

\begin{proof} The system in the item 3 in Proposition \ref{a3} is obviously equivalent to the system in the item 3 of the theorem. Let us solve  the system, say (S), obtained in the item 4 in Proposition \ref{a3}.
	We start by solving the  the last twelve equations of (S). We distinguish two cases:

	$\bullet$ $a_i=0$. Then
	
	\[ \begin{cases}
		\langle[e,u_i],e\rangle=\langle[e,u_i],f\rangle=\langle[\eb,u_i],e\rangle
		=0,\\
		-\langle[e,f],u_i\rangle+\langle[f,u_i],e\rangle
		=0,\\
		-\langle[e,\eb],u_i\rangle
		-\langle[e,u_i],\eb\rangle+2\langle[f,u_i],f\rangle
		=0,\\
		3\langle[f,u_i],e\rangle  
		-\langle[e,f],u_i\rangle=0,\\
		3\langle[\eb,u_i],f\rangle
		-\langle [f,u_i],\eb\rangle
		-\langle [f,\eb],u_i\rangle=0.
	\end{cases} \]
	Thus
	\[ \begin{cases}
		\langle[e,u_i],e\rangle=\langle[e,u_i],f\rangle=\langle[\eb,u_i],e\rangle=
		\langle[e,f],u_i\rangle=\langle[f,u_i],e\rangle
		=0,\\
		-\langle[e,\eb],u_i\rangle
		-\langle[e,u_i],\eb\rangle+2\langle[f,u_i],f\rangle
		=0,\\
		3\langle[\eb,u_i],f\rangle
		-\langle [f,u_i],\eb\rangle
		-\langle [f,\eb],u_i\rangle=0.
	\end{cases} \]
	
	$\bullet$ $a_i\not=0$. Then
	
	\[ \begin{cases}
		\langle[e,f],u_i\rangle=
		\langle[e,u_i],e\rangle
		=0,\\
		(2\al_i-2a)\langle[e,\eb],u_i\rangle
		+\langle[f,u_i],e\rangle+\langle[e,u_i],f\rangle
		=0,\\
		(2\al_i-2a)\langle[f,\eb],u_i\rangle
		-\langle[e,\eb],u_i\rangle-
		\langle[\eb,u_i],e\rangle
		-\langle[e,u_i],\eb\rangle+2\langle[f,u_i],f\rangle
		=0,\\
		3\langle[f,u_i],e\rangle +
		2(a-\al_i)\langle[f,u_i],f\rangle -\langle[e,u_i],f\rangle
		=0,\\
		3\langle[\eb,u_i],f\rangle
		+2(a-\al_i)\langle[\eb,u_i],\eb\rangle
		-\langle [f,u_i],\eb\rangle
		-\langle [f,\eb],u_i\rangle=0,\\
		\langle[e,u_i],f\rangle
		+\langle[f,u_i],e\rangle
		=0,\\
		\langle[f,u_i],e\rangle
		+\langle[e,u_i],f\rangle
		=0,\\	
		2\langle[e,u_i],f\rangle
		+(a-\al_i)\langle[e,u_i],\eb\rangle
		+(a-\al_i)\langle[\eb,u_i],e\rangle
		+(a-\al_i)\langle[\eb,e],u_i\rangle=0,\\
		(a-\al_i)\langle[\eb,u_i],e\rangle
		+(a-\al_i)\langle[e,u_i],\eb\rangle
		-\langle [f,u_i],e\rangle
		+\langle[e,u_i],f\rangle
		+(a-\al_i)\langle[e,\eb],u_i\rangle	=0,\\
		2\langle[f,u_i],f\rangle
		+(a-\al_i)\langle[f,u_i],\eb\rangle
		-\langle [e,u_i],\eb\rangle
		+\langle[\eb,u_i],e\rangle
		+(a-\al_i)\langle[\eb,u_i],f\rangle
		+(\al_i-a)\langle[f,\eb],u_i\rangle
		-\langle[e,\eb],u_i\rangle=0,\\
		2\langle[\eb,u_i],e\rangle
		+(a-\al_i)\langle[\eb,u_i],f\rangle
		+(a-\al_i)\langle[f,u_i],\eb\rangle
		+(a-\al_i)\langle[f,\eb],u_i\rangle=0.
		
	\end{cases} \]
	Thus
	\[ \begin{cases}
		\langle[e,f],u_i\rangle=
		\langle[e,u_i],e\rangle=\langle[e,\eb],u_i\rangle
		=0,\\
		\langle[f,u_i],e\rangle+\langle[e,u_i],f\rangle
		=0,\\
		(2\al_i-2a)\langle[f,\eb],u_i\rangle
		-
		\langle[\eb,u_i],e\rangle
		-\langle[e,u_i],\eb\rangle+2\langle[f,u_i],f\rangle
		=0,\\
		3\langle[f,u_i],e\rangle +
		2(a-\al_i)\langle[f,u_i],f\rangle -\langle[e,u_i],f\rangle
		=0,\\
		3\langle[\eb,u_i],f\rangle
		+2(a-\al_i)\langle[\eb,u_i],\eb\rangle
		-\langle [f,u_i],\eb\rangle
		-\langle [f,\eb],u_i\rangle=0,\\	
		2\langle[e,u_i],f\rangle
		+(a-\al_i)\langle[e,u_i],\eb\rangle
		+(a-\al_i)\langle[\eb,u_i],e\rangle
		=0,\\
		(a-\al_i)\langle[\eb,u_i],e\rangle
		+(a-\al_i)\langle[e,u_i],\eb\rangle
		-\langle [f,u_i],e\rangle
		+\langle[e,u_i],f\rangle
		=0,\\
		2\langle[f,u_i],f\rangle
		+(a-\al_i)\langle[f,u_i],\eb\rangle
		-\langle [e,u_i],\eb\rangle
		+\langle[\eb,u_i],e\rangle
		+(a-\al_i)\langle[\eb,u_i],f\rangle
		+(\al_i-a)\langle[f,\eb],u_i\rangle
		=0,\\
		2\langle[\eb,u_i],e\rangle
		+(a-\al_i)\langle[\eb,u_i],f\rangle
		+(a-\al_i)\langle[f,u_i],\eb\rangle
		+(a-\al_i)\langle[f,\eb],u_i\rangle=0.
		
	\end{cases} \]
	Put $X=\langle [e,u_i],f\rangle$, $Y=\langle [e,u_i],\eb\rangle$, $U=\langle [\eb,u_i],e\rangle$, $V=\langle [\eb,u_i],f\rangle$, $W=\langle [\eb,u_i],\eb\rangle$, $P=\langle [f,u_i],e\rangle$, $Q=\langle [f,u_i],f\rangle$, $R=\langle [f,u_i],\eb\rangle$, $T=\langle [f,\eb],u_i\rangle$. So
	\[ \begin{cases}
		\langle[e,f],u_i\rangle=
		\langle[e,u_i],e\rangle=\langle[e,\eb],u_i\rangle
		=0,\\
		P+X	=0,\\
		(2\al_i-2a)T-U-Y+2Q=0,\\
		3P +2(a-\al_i)Q -X=0,\\
		3V+2(a-\al_i)W-R-T=0,\\	
		2X+(a-\al_i)Y+(a-\al_i)U=0,\\
		(a-\al_i)U+(a-\al_i)Y-P+X
		=0,\\
		2Q+(a-\al_i)R-Y+U+(a-\al_i)V+(\al_i-a)T=0,\\
		2U+(a-\al_i)V+(a-\al_i)R+(a-\al_i)T=0.
	\end{cases} \]
	The solutions of this system are
	\[ P = -\frac{a_i^{2} T}{3}, Q = \frac{2 a_i T}{3}, R = 2 W a_i -T +3 V
	, U = -a_i^{2} W -2 a V, V = V, W = W, X = \frac{a_i^{2} T}{3}, Y = 
	a_i^{2} W -\frac{2}{3} a_i T +2 a_i V.
	\]

	Let us solve the first six equations of (S).
	\[ \begin{cases}
		(2a-\al_i-\al_j)\langle[u_i,v_j],e\rangle
		+
		(\al_j-\al_i)\langle[u_i,e],v_j\rangle
		+(\al_j-\al_i)\langle[v_j,e],u_i\rangle=0,\\
		2\langle[u_i,v_j],e\rangle+ (2a-\al_i-\al_j)\langle[u_i,v_j],f\rangle
		+
		(\al_j-\al_i)\langle[u_i,f],v_j\rangle
		+(\al_j-\al_i)\langle[v_j,f],u_i\rangle=0,\\
		2\langle[u_i,v_j],f\rangle+ (2a-\al_i-\al_j)\langle[u_i,v_j],\eb\rangle
		+
		(\al_j-\al_i)\langle[u_i,\eb],v_j\rangle
		+(\al_j-\al_i)\langle[v_j,\eb],u_i\rangle=0,\\
		(2\al_j-\al_i-a)\langle[u_i,e],v_j\rangle+(a-\al_i)\langle[e,v_j],u_i\rangle+
		(a-\al_i)\langle[u_i,v_j],e\rangle=0,\\
		-\langle[u_i,e],v_j\rangle-\langle[v_j,e],u_i\rangle+ (2\al_j-\al_i-a)\langle[u_i,f],v_j\rangle+(a-\al_i)\langle[f,v_j],u_i\rangle+
		(a-\al_i)\langle[u_i,v_j],f\rangle
		+\langle[u_i,v_j],e\rangle=0,\\
		-\langle[u_i,f],v_j\rangle-\langle[v_j,f],u_i\rangle+ (2\al_j-\al_i-a)\langle[u_i,\eb],v_j\rangle+(a-\al_i)\langle[\eb,v_j],u_i\rangle+
		(a-\al_i)\langle[u_i,v_j],\eb\rangle
		+\langle[u_i,v_j],f\rangle=0.
	\end{cases} \]
	
	For $i=j$, we get
	
	\[ \begin{cases}
		2a_i\langle[u_i,v_i],e\rangle
		=0,\\
		2\langle[u_i,v_i],e\rangle+ 2a_i\langle[u_i,v_i],f\rangle
		=0,\\
		2\langle[u_i,v_i],f\rangle+ 2a_i\langle[u_i,v_i],\eb\rangle
		=0,\\
		-a_i\langle[u_i,e],v_i\rangle
		+a_i\langle[e,v_i],u_i\rangle+
		a_i\langle[u_i,v_i],e\rangle=0,\\
		-\langle[u_i,e],v_i\rangle-\langle[v_i,e],u_i\rangle- a_i\langle[u_i,f],v_i\rangle+
		a_i\langle[f,v_i],u_i\rangle+
		a_i\langle[u_i,v_i],f\rangle
		+\langle[u_i,v_i],e\rangle=0,\\
		-\langle[u_i,f],v_i\rangle-\langle[v_i,f],u_i\rangle -a_i\langle[u_i,\eb],v_i\rangle
		+a_i\langle[\eb,v_i],u_i\rangle+
		a_i\langle[u_i,v_i],\eb\rangle
		+\langle[u_i,v_i],f\rangle=0.
	\end{cases} \]
	
	$\bullet$ If $a_i=0$ then
	
	\[ \langle[u_i,v_i],e\rangle=\langle[u_i,v_i],f\rangle=0,\langle[u_i,e],v_i\rangle+\langle[v_i,e],u_i\rangle=\langle[u_i,f],v_i\rangle+\langle[v_i,f],u_i\rangle=0. \]
	
	$\bullet$ If $a_i\not=0$ then the system becomes
	
	\[ \begin{cases}
		\langle[u_i,v_i],e\rangle=\langle[u_i,v_i],f\rangle=\langle[u_i,v_i],\eb\rangle
		=0,\\
		\langle[e,u_i],v_i\rangle
		+\langle[e,v_i],u_i\rangle=0,\\
		\langle[f,u_i],v_i\rangle+
		\langle[f,v_i],u_i\rangle
		=0,\\
		\langle[\eb,u_i],v_i\rangle
		+\langle[\eb,v_i],u_i\rangle=0.
	\end{cases} \]
	
	For $i\not=j$. Put $X=\langle[u_i,e],v_j\rangle$,
	$Y=\langle[v_j,e],u_i\rangle$, $U=\langle[u_i,f],v_j\rangle$, $V=\langle[v_j,f],u_i\rangle$,
	$P=\langle[u_i,\eb],v_j\rangle$, $Q=\langle[v_j,\eb],u_i\rangle$,
	$R=\langle[u_i,v_j],e\rangle$, $S=\langle[u_i,v_j],f\rangle$,
	$T=\langle[u_i,v_j],\eb\rangle$. Then
	\[ \begin{cases}
		(2a-\al_i-\al_j)R	+	(\al_j-\al_i)X+(\al_j-\al_i)Y=0,\\
		2R+ (2a-\al_i-\al_j)S+(\al_j-\al_i)U
		+(\al_j-\al_i)V=0,\\
		2S+ (2a-\al_i-\al_j)T
		+
		(\al_j-\al_i)P
		+(\al_j-\al_i)Q=0,\\
		(2\al_j-\al_i-a)X-
		(a-\al_i)Y+
		(a-\al_i)R=0,\\
		-X-Y+ (2\al_j-\al_i-a)U-(a-\al_i)V+
		(a-\al_i)S
		+R=0,\\
		-U-V+ (2\al_j-\al_i-a)P
		-(a-\al_i)Q+
		(a-\al_i)T
		+S=0.
	\end{cases} \]
	The solution of this system are\[\begin{cases}\di
		P = 
		-\frac{a_i^{2} T +2a_i S  +R}{\left(\alpha_{j}-\alpha_{i}\right)^{2}},\quad
		 Q = 
		\frac{a_j^{2} T +2a_j S +R}{\left(\alpha_{j}-\alpha_{i}\right)^{2}},\\\di
		U = 
		-\frac{a_i \left(a_iS +2 R \right)}{\left(\alpha_{j}-\alpha_{i}\right)^{2}}
		,\quad V = 
		\frac{a_j \left(a_jS +2 R \right)}{\left(\alpha_{j}-\alpha_{i}\right)^{2}}
		,\\\di X = 
		-\frac{a_i^2R }{\left(\alpha_{j}-\alpha_{i}\right)^{2}}
		,\quad Y = 
		\frac{a_j^2 R}{\left(\alpha_{j}-\alpha_{i}\right)^{2}}.
	\end{cases}
	\]This completes the proof.
	\end{proof}

\section{A complete description of Lorentzian Lie algebras whose curvature is harmonic and the Ricci operator is of type $\{n-2,z\bar{z}\}$ with one real Ricci direction}\label{section4}

In this section, we characterize completely Lorentzian Lie algebras having a harmonic curvature and the Ricci operator of type $\{n-2,z\bar{z}\}$ with one real Ricci direction. This means that $(\G,\br,\prs)$ is a Lorentzian Lie algebra which has a harmonic curvature and $\G=\fl\oplus\h$ where $\h$ is a Euclidean vector subspace, $\fl=\mathrm{span}(e,\eb)$, $\langle e,e\rangle=\langle \eb,\eb\rangle=0$, $\langle e,\eb\rangle=1$ and the Ricci operator $\Ric$ satisfies $\Ric_{|\h}=\al\mathrm{id}_\h$,  $\Ric(e)=ae-b\eb$,  $\Ric(\eb)=be+a\eb$ and $b\not=0$. Namely, we will prove the following theorem.

\begin{theo}\label{mainbiszz}Let $(\G,\br,\prs)$ be a Lorentzian Lie algebra of dimension $\geq3$ such that its Ricci operator is of type $\{n-2,z\bar{z}\}$ with a unique real eigenvalue $\al$. Then $(\G,\br,\prs)$  has a harmonic curvature if and only if one of of the following situation occurs:
	\begin{enumerate}\item $\dim\G=3$ and $(\G,\br,\prs)$ is isomorphic to $\mathrm{sl}(2,\R)$ endowed with  the metric defined in Proposition \ref{sl2}. 
		\item $\dim\G>3$ and $(\G,\br,\prs)$ is isomorphic to the product of $\mathrm{sl}(2,\R)$ endowed with  the metric defined in Proposition \ref{sl2} with a Euclidean Einstein Lie algebra with the Einstein constant $\al<0$.
		
	\end{enumerate}Moreover, in both cases $(\G,\br,\prs)$ is not Ricci-parallel.

\end{theo}

First, let us define the  Lorentzian metric on $\mathrm{sl}(2,\R)$ mentioned in the theorem above.

\begin{pr}\label{sl2} Consider $\mathrm{sl}(2,\R)$ endowed with the basis $\B=(X_1,X_2,X_3)$ 
	\[ X_1=\left(\begin{matrix}
		0&1\\-1&0
	\end{matrix}
	\right) ,\; X_2=\left(\begin{matrix}
		0&1\\1&0
	\end{matrix}
	\right)\esp X_3=\left(\begin{matrix}
		1&0\\0&-1
	\end{matrix}
	\right)\]and the  metric $\prs_0$ given by its matrix in $\B$
	\[ M(\prs_1,\B)=-\frac8\alpha\left(\begin{matrix}
		1&0&-1\\0&1&0\\-1&0&0
	\end{matrix}
	\right).\quad (\al<0). \]
	Then $(\mathrm{sl}(2,\R),\prs_0)$ is a Lorentzian Lie algebra with a harmonic curvature not Ricci-parallel.
\end{pr}
\begin{proof} The Ricci operators of $\prs_0$ is given by
	\[ \Ric_0=\left(\begin{array}{ccc}
		0 & 0 & -\alpha  
		\\
		0 & \alpha  & 0 
		\\
		\alpha  & 0 & -\alpha  
	\end{array}\right).
	\]
	One can check straightforwardly that $\Ric_0$ is a Codazzi operator which is not parallel. 
	\end{proof}

\subsection{Proof of Theorem \ref{mainbiszz}} Let $(\G,\br,\prs)$ be a Lorentzian Lie algebra which has a harmonic curvature and $\G=\fl\oplus\h$ where $\h$ is a Euclidean vector subspace, $\fl=\mathrm{span}(e,\eb)$, $\langle e,e\rangle=\langle \eb,\eb\rangle=0$, $\langle e,\eb\rangle=1$ and the Ricci operator  satisfies
\begin{equation}\label{riczz}
\Ric_{|\h}=\al\mathrm{id}_\h,\;  \Ric(e)=ae-b\eb, \; \Ric(\eb)=be+a\eb \esp b\not=0.	
\end{equation}
According to Theorem \ref{mainzz}, $\h$ is an Euclidean  	 Lie subalgebra, 
\begin{equation}\label{brzz}[e,\eb]=u_0\in\h,\; [e,u]=\om\langle u,u_0\rangle e+\rho\langle u,u_0\rangle\eb+Au \esp [\eb,u]=\rho\langle u,u_0\rangle e-\om\langle u,u_0\rangle \eb+Bu, 
\end{equation}for any $u\in\h$, 
 $A,B:\h\too\h$ are skew-symmetric, $\om=\frac{(a-\al)^2-b^2}{4b^2}$ and $\rho=\frac{(a-\al)}{2b}$. Put $\mu=\langle u_0,u_0\rangle$. For any $u\in\G$,  $\ad_u:\G\too\G$ is given by $\ad_u(v)=[u,v]$ and if $u\in\h$ we denote by $\ad_u^\h$ the restriction of $\ad_u$ to $\h$.
 
Recall that the Ricci curvature of $(\G,\br,\prs)$ is given by
\begin{equation}\label{ricci} \ric(u,v)=-\frac12\tr(\ad_u\circ\ad_v)
-\frac12\tr(\ad_u^*\circ\ad_v)
-\frac14\tr(J_u\circ J_v)-\frac12\left(\langle[H,u],v\rangle+\langle[H,v],u\rangle\right), \end{equation}where $H$ is defined by the relation $\langle H,u\rangle=\tr(\ad_u)$ and $J_uv=\ad_v^*u$ (see \cite[Proposition 2.1]{tibssirte}).

\begin{pr} We have
	\[ \ric(e,e)=-\mu\rho
	(1+2\om)+\rho\tr(\ad_{u_0}^\h), \]and hence $\mu\not=0$.
	
\end{pr}
\begin{proof} Choose an orthonormal basis $(u_1,\ldots,u_r)$ of $\h$. We have
	\begin{align*}
		\tr(\ad_e)&=\langle [e,\eb],e\rangle+\sum_{i=1}^r\langle [e,u_i],u_i\rangle=0,\\
		\tr(\ad_{\eb})&=\langle [\eb,e],\eb\rangle+\sum_{i=1}^r\langle [\eb,u_i],u_i\rangle=0\\
	\tr(\ad_u)&=\langle [u,e],\eb\rangle+\langle [u,\eb],e\rangle+\sum_{i=1}^r\langle [u,u_i],u_i\rangle
	=-\om u_0+\om u_0+\tr(\ad_u^\h)=\tr(\ad_u^\h).	
	\end{align*}
So $H\in\h$.
	On the other hand, 
	\begin{align*}
		\tr(\ad_e^2)&=\langle [e,[e,\eb]],e\rangle
		+\sum_{i=1}^r\langle [e,[e,u_i]],u_i\rangle
		=2\rho\mu+\tr(A^2)\\
		\tr(\ad_e^*\circ\ad_e)&=\sum_{i=1}^r\langle [e,u_i],[e,u_i]\rangle
		=-\tr(A^2)+2\om\rho\mu,\\
		\tr(J_e\circ J_e)
		&=-2\langle \ad_{e}^*e,\ad_{\eb}^*e\rangle
		-\sum_{i=1}^r\langle \ad_{u_i}^*e,\ad_{u_i}^*e\rangle,\\
		\ad_{e}^*e&=\langle \ad_{e}^*e,\eb\rangle e+\langle \ad_{e}^*e,e\rangle\eb+\sum_{i=1}^r\langle \ad_{e}^*e,u_i\rangle u_i=\rho u_0 ,\\
		\ad_{\eb}^*e&=\langle \ad_{\eb}^*e,\eb\rangle e+
		\langle \ad_{\eb}^*e,e\rangle \eb+\sum_{i=1}^r
		\langle \ad_{\eb}^*e,u_i\rangle u_i
		=-\om u_0,\\
		\ad_{u}^*e&=\langle \ad_{u}^*e,\eb\rangle e+
		\langle \ad_{u}^*e,e\rangle \eb+\sum_{i=1}^r
		\langle \ad_{u}^*e,u_i\rangle u_i=\om\langle u,u_0\rangle e-\rho\langle u,u_0\rangle \eb.
	\end{align*}So $\tr(J_e\circ J_e)=4\rho\om.$
	Moreover, $\langle [H,e],e\rangle=-\rho\tr(\ad_{u_0}^\h)$ and we get the desired formula. Now, according to \eqref{riczz}, $\ric(e,e)=b\not=0$ and hence $\mu\not=0$.
\end{proof}

\begin{pr}\label{jacobizz}We have $u_0\in[\h,\h]^\perp\cap\h$ and $A=B=\ad_{u_0}^\h=0$.

\end{pr}
\begin{proof}
	Note first that since $b\not=0$, $(\om,\rho)\not=(0,0)$. Let us write the  necessary and sufficient conditions for the bracket \eqref{brzz} to satisfies the Jacobi identity.  For any $u,v\in\h$,
	\begin{align*}
		[u,[e,\eb]]&=[[u,e],\eb]+[e,[u,\eb]]\\
		&=-\om\langle u,u_0\rangle[e,\eb]+[\eb,A(u)]+\om\langle u,u_0\rangle[e,\eb]-[e,Bu]\\
		&=\rho\langle Au,u_0\rangle e-\om\langle Au,u_0\rangle \eb+BAu-\om\langle Bu,u_0\rangle e-\rho\langle Bu,u_0\rangle\eb-ABu\\
		&=\left(\rho\langle Au,u_0\rangle-\om\langle Bu,u_0\rangle\right)e-
		\left(\om\langle Au,u_0\rangle+\rho\langle Bu,u_0\rangle\right)\eb+[B,A]u.
	\end{align*}
	Since $(\rho,\om)\not=0$, we get $Au_0=Bu_0=0$ and $\ad_{u_0}^\h=[A,B]$. \begin{align*}
		[e,[u,v]]&=\om\langle [u,v],u_0\rangle e+\rho\langle [u,v],u_0\rangle\eb+A[u,v]\\
		&=[[e,u],v]+[u,[e,v]]\\
		&=\om\langle u,u_0\rangle [e,v]+\rho\langle u,u_0\rangle[\eb,v]+[Au,v]+\om\langle v,u_0\rangle [u,e]+\rho\langle v,u_0\rangle[u,\eb]+[u,Av]\\
		&=\om\langle u,u_0\rangle\left(\om\langle v,u_0\rangle e+\rho\langle v,u_0\rangle\eb+Av\right)+\rho\langle u,u_0\rangle\left(\rho\langle v,u_0\rangle e-\om\langle v,u_0\rangle \eb+Bv\right)\\
		&-\om\langle v,u_0\rangle\left(\om\langle u,u_0\rangle e+\rho\langle u,u_0\rangle\eb+Au\right)-\rho\langle v,u_0\rangle\left(\rho\langle u,u_0\rangle e-\om\langle u,u_0\rangle \eb+Bu \right)+[Au,v]+[u,Av]\\
		&=\om\langle u,u_0\rangle Av+\rho\langle u,u_0\rangle Bv-\om\langle v,u_0\rangle Au
		-\rho\langle v,u_0\rangle Bu+[Au,v]+[u,Av].
	\end{align*}
	\begin{align*}
		[\eb,[u,v]]&=\rho\langle [u,v],u_0\rangle e-\om\langle [u,v],u_0\rangle\eb+B[u,v]\\
		&=[[\eb,u],v]+[u,[\eb,v]]\\
		&=\rho\langle u,u_0\rangle [e,v]-\om\langle u,u_0\rangle[\eb,v]+[Bu,v]+\rho\langle v,u_0\rangle [u,e]-\om\langle v,u_0\rangle[u,\eb]+[u,Bv]\\
		&=\rho\langle u,u_0\rangle\left(\om\langle v,u_0\rangle e+\rho\langle v,u_0\rangle\eb+Av\right)-\om\langle u,u_0\rangle\left(\rho\langle v,u_0\rangle e-\om\langle v,u_0\rangle \eb+Bv\right)\\
		&-\rho\langle v,u_0\rangle\left(\om\langle u,u_0\rangle e+\rho\langle u,u_0\rangle\eb+Au\right)+\om\langle v,u_0\rangle\left(\rho\langle u,u_0\rangle e-\om\langle u,u_0\rangle \eb+Bu \right)+[Bu,v]+[u,Bv]\\
		&=\rho\langle u,u_0\rangle Av-\om\langle u,u_0\rangle Bv-\rho\langle v,u_0\rangle Au
		+\om\langle v,u_0\rangle Bu
		+[Bu,v]+[u,Bv],
	\end{align*} Hence 
the bracket \eqref{brzz} satisfies the Jacobi identity if and only if, for any $u,v\in\h$,
\[ \begin{cases}\ad_{u_0}^\h=[A,B],\; Au_0=Bu_0=0,\; u_0\in[\h,\h]^\perp\cap\h,\\
	A[u,v]=\om\langle u,u_0\rangle Av+\rho\langle u,u_0\rangle Bv-\om\langle v,u_0\rangle Au
	-\rho\langle v,u_0\rangle Bu+[Au,v]+[u,Av],\\
	B[u,v]=\rho\langle u,u_0\rangle Av-\om\langle u,u_0\rangle Bv-\rho\langle v,u_0\rangle Au
	+\om\langle v,u_0\rangle Bu
	+[Bu,v]+[u,Bv].
\end{cases} \]
In particular $\h_0=u_0^\perp\cap\h$ is an ideal containing $[\h,\h]$, invariant by $A,B$ and $\ad_{u_0}^\h$ and in restriction to $\h_0$, $A$ and $B$ are two derivations satisfying
\[ [A,\ad_{u_0}]=\mu(\om A+\rho B),\; 
[B,\ad_{u_0}]=\mu(\rho A-\om B)\esp [A,B]=\ad_{u_0}^\h. \]
Consider the 3-dimensional Lie algebra $\G_0=\mathrm{span}(u,v,w)$ and
\[ [u,v]=w, [u,w]=\mu(\om u+\rho v)\esp [v,w]=\mu(\rho A-\om B). \]
Since $\mu\not=0$, $\G_0$ is isomorphic to $\mathrm{sl}(2,\R)$ and the map $\phi:\G_0\too\mathrm{so}(\h_0)$ given by $\phi(u)=A$, $\phi(v)=B$ and $\phi(w)=\ad_{u_0}^\h$ is a representation. So $\ker\phi$ is an ideal of $\G_0$ and hence $\dim\ker\phi=3$ or $0$. The last case is not possible since $\mathrm{so}(\h_0)$ has no subalgebra isomorphic to $\mathrm{sl}(2,\R)$. This  completes the proof.
\end{proof}

We have shown so far that if $(\G,\br,\prs)$ has a harmonic curvature then 
\begin{equation}\label{brzzbis}[e,\eb]=u_0,\; [e,u]=\om\langle u,u_0\rangle e+\rho\langle u,u_0\rangle\eb \esp [\eb,u]=\rho\langle u,u_0\rangle e-\om\langle u,u_0\rangle \eb 
\end{equation}for any $u\in\h$, $u_0\in[\h,\h]^\perp\cap\h$, $\ad_{u_0}=0$, $\om=\frac{(a-\al)^2-b^2}{4b^2}$ and $\rho=\frac{(a-\al)}{2b}$. Thus  $(\G,\br,\prs)$ is the product of the 3-dimensional  Lie algebra $(\G_0,\br,\prs_0)$ generated by $\B_0=(e,\eb,u_0)$ and the  Euclidean Lie algebra $\h_0=u_0^\perp\cap\h$. So we can compute the Ricci curvature easily as a product of the Ricci curvature of $\G_0$ and $\h_0$ and we get

 For any $u\in\h_0$,
	\[ \ric(e,e)=-\ric(\eb,\eb)=-\mu\rho
		(1+2\om),
		\ric(e,\eb)=(\om+\frac12)\mu,
		\ric(u_0,u_0)=-\mu^2\left(\frac12+2\rho^2\right),
		\ric(u,u)=
		\ric^\h(u,u).
	 \]So $\h_0$ is a $\al$-Einstein Lie algebra and
\[
b=(1+2\om)\rho\mu,\;a=
		(\om+\frac12)\mu, \al=-\mu\left(2\rho^2+\frac12\right).
  \] In particular, $\al<0$.
If we put $\mu=-\frac{\al}{2\rho^2+\frac12}$, we get
\[b-(1+2\om)\rho\mu= \frac{a \alpha -\alpha^{2}+2 b^{2}}{2 b}\esp a-
(\om+\frac12)\mu= a +\frac{\alpha}{2}. 
\]
So $a=-\frac\al2$ and $-\frac32\al^2+2b^2=0$ thus $b^2=\frac34\al^2$.
Finally,
\[ \om=\frac12,\; \rho=\e\frac{\sqrt{3}}2\esp \mu=-\frac{\al}{2},\;\e^2=1. \]
To summarize, we have shown that $(\G,\br,\prs)$ is the product of the 3-dimensional  Lie algebra $(\G_0,\br,\prs_0)$ generated by $\B_0=(e,\eb,u_0)$ and the $\al$-Einstein Euclidean Lie algebra $\h_0$. Put $f=u_0$. The Lie bracket and the metric on $\G_0$ are given by
\[ [e,f]=-\frac{\al}{4}\left(e+\e\sqrt{3}\eb  \right),[\eb,f]=-\frac{\al}{4}\left(\e\sqrt{3}e-\eb  \right)\esp [e,\eb]=f\esp M(\prs_0,\B_0)=\left(\begin{matrix}
	0&1&0\\1&0&0\\0&0&-\frac\al2
\end{matrix}
\right), \]where $\e^2=1$.
To complete the proof of Theorem \ref{mainbiszz}, we need to show the following proposition.
\begin{pr} $(\G_0,\br,\prs_0)$ is isomorphic to $\mathrm{sl}(2,\R)$ endowed with the metric given in Proposition \ref{sl2}.
	
\end{pr}

\begin{proof} The basis $(X_1,X_2,X_3)$ given in Proposition \ref{sl2} satisfies
	\[ [X_1,X_2]=2X_3,\;[X_3,X_1]=2X_2\esp [X_3,X_2]=2X_1. \]
	For $\e=1$ the basis $\B_1=(f_1,f_2,f_3)$ of $\G_0$ given by
\[ f_1=\frac{\sqrt{3}\, b}{3} e+ b\eb   ,\;f_2=\frac4\al f\esp f_3=-\frac{2\sqrt{3}b}{3}e,\; b^2=-\frac{4\sqrt{3}}{\alpha} \]satisfies
\[ [f_1,f_2]=2f_3,\;[f_3,f_1]=2f_2\esp [f_3,f_2]=2f_1, \]  the metric in this basis is given by
\[ M(\prs_0,\B_1)=-\frac8\alpha\left(\begin{matrix}
	1&0&-1\\0&1&0\\-1&0&0
\end{matrix}
\right) \]and we get the desired isomorphism.

For $\e=-1$ the basis $\B_{-1}=(f_1,f_2,f_3)$ given by
\[ f_1=b e   ,\;f_2=\frac4\al f\esp f_3=-\frac{b}{2} e+ \frac{\sqrt{3}\, b}{2}\eb
,\; b^2=-\frac{16\sqrt{3}}{3\alpha} \]satisfies
\[ [f_1,f_2]=2f_3,\;[f_3,f_1]=2f_2\esp [f_3,f_2]=2f_1 \] and the metric in this basis is given by
\[ M(\prs,\B_{-1})=-\frac8\alpha\left(\begin{matrix}
	0&0&1\\0&1&0\\1&0&-1
\end{matrix}
\right). \]
To conclude, remark that
the automorphism $Q$ of $\mathrm{sl}(2,\R)$ whose matrix in $(X_1,X_2,X_3)$ is $$\left(\begin{array}{ccc}
		-\frac{2 \sqrt{3}}{3} & 0 & \frac{\sqrt{3}}{3} 
		\\
		0 & -1 & 0 
		\\
		-\frac{\sqrt{3}}{3} & 0 & \frac{2 \sqrt{3}}{3} 
	\end{array}\right)
	$$   satisfies $\langle Qu,Qv\rangle_1=\langle u,v\rangle_2$ where $\prs_1$ is the metric of  $\mathrm{sl}(2,\R)$ whose matrix in the basis $(X_1,X_2,X_3)$ is $M(\prs_0,\B_1)$ and $\prs_2$ is the metric of  $\mathrm{sl}(2,\R)$ whose matrix in the basis $(X_1,X_2,X_3)$ is $M(\prs_0,\B_{-1})$.
\end{proof}

\section{A complete description of Lorentzian Lie algebras having harmonic   curvature  and the Ricci operator is of type $\{n,a2,\al\}$  with $a\not=\al$}\label{section5}

In this section, we characterize completely  Lorentzian Lie algebras having a harmonic curvature and the Ricci operator of type $\{n,a2\}$ with one real Ricci direction different from $a$. This means that $(\G,\prs)$ is a Lorentzian Lie algebra which is curvature harmonic and $\G=\fl\oplus\h$ where the Ricci operator $\Ric$ satisfies $\Ric_{|\h}=\al\mathrm{id}_\h$, $\fl=\mathrm{span}(e,\eb)$ with $\Ric(e)=ae$,  $\Ric(\eb)=e+a\eb$ and $a\not=\al$. Namely, we will prove the following theorem.

\begin{theo}\label{mainbisa2} Let $(\G,\br,\prs)$ be a Lorentzian Lie algebra such that its Ricci operator is of type $\{n,a2\}$ and has  a unique eigenvalue $\al\not=a$. Then $(\G,\br,\prs)$  has a harmonic curvature if and only if $a=0$,  $\al<0$ and $\G=\mathrm{span}(e,\eb)\oplus \h$ where $\mathrm{span}(e,\eb)$ is a Lorentzian abelian ideal, $\h$ is a Euclidean subalgebra of $\G$ such that the restriction of the metric to $\h$ is $\al$-Einstein
	and there exists a non null vector $X\in[\h,\h]^\perp$ such that, for any $u\in X^\perp$
	\[ [e,u]=[\eb,u]=0,\; [e,X]=\langle X,X\rangle e,\;[\eb,X]=-\langle X,X\rangle\left(\frac2\al e+\eb\right)\esp \al=-4|X|^2-2\tr(\ad_X^\h).  \]
	Moreover, $(\G,\br,\prs)$ is not Ricci-parallel.
	\end{theo}

\subsection{Proof of Theorem \ref{mainbisa2}}	
	
Let $(\G,\br,\prs)$ be a Lorentzian Lie algebra with a harmonic curvature and $\G=\fl\oplus\h$ where $\h$ is a Euclidean vector subspace, $\fl=\mathrm{span}(e,\eb)$, $\langle e,e\rangle=\langle \eb,\eb\rangle=0$, $\langle e,\eb\rangle=1$ and the Ricci operator  satisfies 
\begin{equation}\label{rica2}\Ric_{|\h}=\al\mathrm{id}_\h,  \Ric(e)=ae\esp \Ric(\eb)=e+a\eb.\end{equation}
We suppose $a\not=\al$.
For any $u\in\G$,  $\ad_u:\G\too\G$ is given by $\ad_u(v)=[u,v]$.

According to Theorem \ref{maina2}, $\h$ is a Lie subalgebra and, for any $u,v\in\h$,
\begin{equation}\label{b2a} [e,\eb]=\mu e, [e,u]=\langle X,u\rangle e+Au,[\eb,u]=\frac2{a-\al}\langle X,u\rangle e-\langle X,u\rangle\eb+Bu,  
\end{equation}
where $A,B:\h\too\h$ are skew-symmetric and $X\in\h$. Put $\rho=\frac2{a-\al}$  and if $u\in\h$ we denote by $\ad_u^\h$ the restriction of $\ad_u$ to $\h$.

Let us start by showing that $X$ must be non null.
\begin{pr}\label{para2}
	We have
	\[ \ric(\eb,\eb)=\rho(2|X|^2+\tr(\ad_X^\h)), \]
	and hence $X\not=0$.
\end{pr}

\begin{proof} We will use the formula \eqref{ricci}. Choose an orthonormal basis $(u_1,\ldots,u_r)$ of $\h$. We have
	\begin{align*}
		\tr(\ad_e)&=\langle [e,\eb],e\rangle+\sum_{i=1}^r\langle [e,u_i],u_i\rangle=0,\\
		\tr(\ad_{\eb})&=\langle [\eb,e],\eb\rangle+\sum_{i=1}^r\langle [\eb,u_i],u_i\rangle=-\mu\\
		\tr(\ad_u)&=\langle [u,e],\eb\rangle+\langle [u,\eb],e\rangle+\sum_{i=1}^r\langle [u,u_i],u_i\rangle
		=-\langle X,u\rangle+\langle X,u\rangle+\tr(\ad_u^\h)=\tr(\ad_u^\h).	
	\end{align*}
	So $H=-\mu e+H_0$ where $H_0\in\h$ and $\langle H_0,u\rangle=\tr(\ad_u^\h)$.
	On the other hand, 
	\begin{align*}
		\tr(\ad_{\eb}^2)&=\langle [\eb,[\eb,e]],\eb\rangle
		+\sum_{i=1}^r\langle [\eb,[\eb,u_i]],u_i\rangle
		=\mu^2+\tr(B^2)\\
		\tr(\ad_{\eb}^*\circ\ad_{\eb})&=
		\sum_{i=1}^r\langle [\eb,u_i],[\eb,u_i]\rangle
		=-2\rho|X|^2-\tr(B^2),\\
		\tr(J_{\eb}^2)
		&=-2\langle \ad_{e}^*\eb,\ad_{\eb}^*\eb\rangle
		-\sum_{i=1}^r\langle \ad_{u_i}^*\eb,\ad_{u_i}^*\eb\rangle,\\
		\ad_{e}^*\eb&=\langle \ad_{e}^*\eb,\eb\rangle e+\langle \ad_{e}^*\eb,e\rangle\eb+\sum_{i=1}^r\langle \ad_{e}^*\eb,u_i\rangle u_i
		=\mu e+X ,\\
		\ad_{\eb}^*\eb&=\langle \ad_{\eb}^*\eb,\eb\rangle e+
		\langle \ad_{\eb}^*\eb,e\rangle \eb+\sum_{i=1}^r
		\langle \ad_{\eb}^*\eb,u_i\rangle u_i
		=-\mu\eb+\rho X,\\
		\ad_{u}^*\eb&=\langle \ad_{u}^*\eb,\eb\rangle e+
		\langle \ad_{u}^*\eb,e\rangle \eb+\sum_{i=1}^r
		\langle \ad_{u}^*\eb,u_i\rangle u_i=
		-\rho \langle X,u\rangle e-\langle X,u\rangle \eb.
	\end{align*}So $\tr(J_{\eb}^2)=2\mu^2-4\rho|X|^2.$
	Moreover, $\langle [H,\eb],\eb\rangle=-\mu^2-\rho\tr(\ad_{X}^\h)$ and we get the desired formula. Now, according to \eqref{rica2}, $\ric(\eb,\eb)=1$ and hence $X\not=0$.
	
\end{proof}

\begin{pr} We have $A=B=0$, $\mu=0$,  $X\in[\h,\h]^\perp$ and $X\not=0$.
	
\end{pr}

\begin{proof}

Let us write the Jacobi identity. For any $u,v\in\h$,
\begin{align*}
	[u,[e,\eb]]&=\mu[u,e]\\
	&=-\mu\langle X,u\rangle e-\mu Au\\
	&=[[u,e],\eb]+[e,[u,\eb]]\\
	&=-\langle X,u\rangle [e,\eb]+[\eb,Au]+\langle X,u\rangle[e,\eb]-[e,Bu]\\
	&=\rho\langle X,Au\rangle e-\langle X,Au\rangle\eb+BAu-\langle X,Bu\rangle e-ABu.
\end{align*} 
\begin{align*}
	[e,[u,v]]&=\langle X,[u,v]\rangle e+A[u,v]\\
	&=[[e,u],v]+[u,[e,v]]\\
	&=\langle X,u\rangle[e,v]+[Au,v]+\langle X,v\rangle[u,e]+[u,Av]\\
	&=\langle X,u\rangle\langle X,v\rangle e+\langle X,u\rangle Av-\langle X,u\rangle\langle X,v\rangle e-\langle X,v\rangle Au+[u,Av]+[Au,v].
\end{align*}
\begin{align*}
	[\eb,[u,v]]&=\rho\langle X,[u,v]\rangle e-\langle X,[u,v]\rangle\eb+B[u,v]\\
	&=[[\eb,u],v]+[u,[\eb,v]]\\
	&=\rho\langle X,u\rangle[e,v]
	-\langle X,u\rangle[\eb,v]+[Bu,v]
	-\rho\langle X,v\rangle[e,u]
	+\langle X,v\rangle[\eb,u]+[u,Bv]\\
	&=\rho\langle X,u\rangle\langle X,v\rangle e+\rho\langle X,u\rangle Av
	-\langle X,u\rangle+\rho\langle X,v\rangle e+\langle X,u\rangle\langle X,v\rangle\eb-\langle X,u\rangle Bv+[Bu,v]
	\\&	-\rho\langle X,v\rangle\langle X,u\rangle e-\rho\langle X,v\rangle Au
	+\langle X,v\rangle+\rho\langle X,u\rangle e-\langle X,v\rangle\langle X,u\rangle\eb+\langle X,v\rangle Bu+[u,Bv]\\
	&=\rho\langle X,u\rangle Av
	-\langle X,u\rangle Bv-\rho\langle X,v\rangle Au+\langle X,v\rangle Bu+[u,Bv]+[Bu,v].
\end{align*}
The bracket \eqref{b2a} satisfies the Jacobi identity if and only if, for any $u,v\in\h$,
	\[ \begin{cases}
		AX=0,BX=-\mu X, \mu A=[A,B], X\in[\h,\h]^\perp\\
		A[u,v]=\langle X,u\rangle Av-\langle X,v\rangle Au+[u,Av]+[Au,v],\\
		B[u,v]=\rho\langle X,u\rangle Av
		-\langle X,u\rangle Bv-\rho\langle X,v\rangle Au+\langle X,v\rangle Bu+[u,Bv]+[Bu,v].
	\end{cases} \]
	For $u=X$ and $v\in X^\perp$, we get
\[\begin{cases}[ A,\ad_X^h]=\langle X,X\rangle A,\\
	[B,\ad_X^\h]=\rho\langle X,X\rangle A
	-\langle X,X\rangle B.\end{cases}
\]
Put $\ad_X^\h=S_X+A_X$ where $S_X$ is the symmetric part and $A_X$ its skew-symmetric part. Since $[A,S_X]$ and $[B,S_X]$ are symmetric, we get 
\[\begin{cases}[ A,A_X]=\langle X,X\rangle A,\\
	[B,A_X]=\rho\langle X,X\rangle A
	-\langle X,X\rangle B.\end{cases}
\]Thus $(A,B,A_X)$ generates a solvable Lie subalgebra of $\mathrm{so}(\h)$ and hence must be abelian. So $A=B=0$ and since $X\not=0$ then  $\mu=0$.\end{proof}

The following proposition gives the Ricci curvature of $(\G,\br,\prs)$. The details of the computation are given in the appendix.

\begin{pr}\label{riccia21} For any $u\in\h$,
	\[ \begin{cases}\ric(e,e)=\ric(e,\eb)=
		\ric(e,u)=
		\ric(\eb,u)=0\\
		\ric(\eb,\eb)=2\rho|X|^2+\rho\tr(\ad^\h_X),\\
		\ric(u,u)=\ric^\h(u,u),
	\end{cases} \]where $\ric^\h$ is the Ricci curvature of the restriction of the metric to $\h$.
	
\end{pr}
From this proposition and \eqref{rica2}, we get that
\[ a=\ric(e,\eb)=0,\;1=\ric(\eb,\eb)=2\rho|X|^2+\rho\tr(\ad^\h_X)\esp \ric^\h(u,u)=\al\langle u,u\rangle. \]
This shows that $(\h,\br_\h,\prs)$ is $\al$-Einstein and, since $\rho=-\frac2\al$,
$ \al=-4|X|^2-2\tr(\ad_X^\h)$. If $\al>0$ then $\h$ is unimodular and hence $\tr(\ad_X^\h)=0$ which is a contradiction. To complete the proof, let us show that $(\G,\br,\prs)$ is not Ricci-parallel. Indeed,
\begin{eqnarray*}
	(\nabla_X\ric)(\eb,\eb)&=& \langle[X,\Ric \eb],w\rangle+\langle [\eb,\Ric \eb],X\rangle+ \langle [\eb,X],\Ric \eb\rangle
	-\langle [X,\eb],\Ric \eb\rangle\\&&-\langle [\Ric \eb,\eb],X\rangle-\langle [\Ric \eb,X],\eb\rangle\\
	&=& 2\langle [X,e],\eb\rangle +2\langle[\eb,X],e \rangle\\
	&=& -4|X|^2\not=0. 
\end{eqnarray*}

\section{A complete description of Lorentzian Lie algebras having    harmonic curvature  and the Ricci operator is of type $\{n,a3,\al\}$  with $a\not=\al$}\label{section6}

In this section, we characterize completely  Lorentzian Lie algebras having a harmonic curvature and the Ricci operator of type $\{n,a3\}$ with one real Ricci direction different from $a$. This means that $(\G,\prs)$ is a Lorentzian Lie algebra which is curvature harmonic and $\G=\fl\oplus\h$ where the Ricci operator $\Ric$ satisfies $\Ric_{|\h}=\al\mathrm{id}_\h$, $\fl=\mathrm{span}(e,f,\eb)$ with $\Ric(e)=ae$, $\Ric(f)=e+af$, $\Ric(\eb)=f+a\eb$ and $a\not=\al$. Namely, we will prove the following theorem.

\begin{theo}\label{mainbisa3} Let $(\G,\br,\prs)$ be a Lorentzian Lie algebra such that its Ricci operator is of type $\{n,a3\}$ and has a unique eigenvalue $\al\not=a$. Then $(\G,\br,\prs)$  has a harmonic curvature if and only if $a=0$,  $\al<0$ and $\G$ has an orthogonal splitting
	$\G=\mathrm{span}(e,f,\eb)\oplus\h$ 
	where $\mathrm{span}(e,f,\eb)$ is a 3-dimensional abelian Lorentzian Lie subalgebra, $\h$ is a subalgebra, the restriction of $\prs$ to  $\h$ is Euclidean a $\al$-Einstein   and there exists a non null vector $X\in[\h,\h]^\perp$ and $U\in\h$ such that,  for any $u\in \h$,
$$\begin{cases}[e,\eb]= [e,f]=0=[\eb,f]=0,\\
	[e,u]=\langle X,u\rangle e,\\
	[\eb,u]=\langle Y,u\rangle e-\langle X,u\rangle \eb+\langle T,u\rangle f+\frac{\langle T,X\rangle}{\langle X,X\rangle}Cu,\\
	[f,u]=\langle U,u\rangle e+Cu,\\
	\langle e,e\rangle=\langle \eb,\eb\rangle=\langle e,f\rangle=\langle \eb,f\rangle=0,\langle e,\eb\rangle=\langle f,f\rangle=1, 
\end{cases}	$$
where $C$ is a skew-symmetric derivation of $\h$  satisfying $CX=CU=0$,  $Y=-\frac1{\al}\left(2U+\frac3\al X\right)$, $ T=-\frac2\al X-U$,  
\[ \langle [\h,P^\perp],{U}\rangle=0\esp \langle {U},[U,X]\rangle=\langle {U},{U}\rangle \langle X,X\rangle-\langle U,X\rangle^2,   \]
where $P=\mathrm{span}(X,U)$ and
\[ 
	\tr(\ad_X^\h)=-\al-|X|^2\esp
	\tr(\ad_U^\h)=-\frac5{2\al}|X|^2-3\langle X,U\rangle+\frac32.
 \]
 Moreover, $(\G,\br,\prs)$ is not Ricci-parallel.
	
\end{theo}

\subsection{Proof of Theorem \ref{mainbisa3}}
	
 Let $(\G,\br,\prs)$ be a Lorentzian Lie algebra which has a harmonic curvature and $\G=\fl\oplus\h$ where $\h$ is a Euclidean vector subspace, $\fl=\mathrm{span}(e,f,\eb)$, $\langle e,e\rangle=\langle \eb,\eb\rangle=\langle f,e\rangle=\langle f,\eb\rangle=0$, $\langle e,\eb\rangle=\langle f,f\rangle=1$ and the Ricci operator $\Ric$ satisfies 
 \begin{equation}\label{rica3}
 	\Ric_{|\h}=\al\mathrm{id}_\h,\;  \Ric(e)=ae,\; \Ric(f)=e+af \esp \Ric(\eb)=f+a\eb.
 \end{equation}
For any $u\in\G$,  $\ad_u:\G\too\G$ is given by $\ad_u(v)=[u,v]$.

We suppose that $a\not=\al$. According to Theorem \ref{maina3}, $\h$ is a Lie subalgebra and, for any $u,v\in\h$,
\begin{equation}\label{b3a} 
	\begin{cases}[e,\eb]=x e+y f, [e,f]=5ye,[\eb,f]=ze-2y\eb-\frac12 xf+u_0,\\
		[e,u]=\langle X,u\rangle e-\frac{(a-\al)^2}3\langle u_0,u\rangle f+Au,\\
		[\eb,u]=\langle Y,u\rangle e+\langle Z,u\rangle \eb+\langle T,u\rangle f+Bu,\\
		[f,u]=\langle U,u\rangle e+\frac{(a-\al)^2}3\langle u_0,u\rangle\eb-\frac{2(a-\al)}3\langle u_0,u\rangle f+Cu	
	\end{cases}	  
\end{equation}
where, $x,y,z\in\R$, $X,Y,Z,T,u_0\in\h$, $A,B,C:\h\too\h$ are skew-symmetric and
\[ X=\rho(\rho Y+2T+\frac23u_0),Z=-\rho(\rho Y+2T),U=2\rho Y+3T+u_0\esp \rho=a-\al. \]
 If $u\in\h$ we denote by $\ad_u^\h$ the restriction of $\ad_u$ to $\h$.

\begin{pr} We have
	\[ \ric(e,e)=-\frac{\rho^4}{36}|u_0|^2. \]
	In particular, $u_0=0$.
	\end{pr}
\begin{proof} We will use the formula \eqref{ricci}. Choose an orthonormal basis $(u_1,\ldots,u_r)$ of $\h$.
	Let us compute. 
	\begin{align*}
		\tr(\ad_e^2)&=\langle [e,[e,\eb]],e\rangle+
		\langle [e,[e,f]],f\rangle+\sum_{i=1}^r\langle [e,[e,u_i]],u_i\rangle
		=\tr(A^2),\\
		\tr(\ad_e^*\circ\ad_e)&=\langle [e,f],[e,f]\rangle+\sum_{i=1}^r\langle [e,u_i],[e,u_i]\rangle
		=-\tr(A^2)+\frac{\rho^4}{9}|u_0|^2,\\
		\tr(J_e\circ J_e)
		&=-2\langle \ad_{e}^*e,\ad_{\eb}^*e\rangle
		-\langle \ad_{f}^*e,\ad_{f}^*e\rangle
		-\sum_{i=1}^r\langle \ad_{u_i}^*e,\ad_{u_i}^*e\rangle,\\
		\ad_{e}^*e&=0,\\
		\ad_{f}^*e&=\langle \ad_{f}^*e,\eb\rangle e+
		\langle \ad_{f}^*e,e\rangle \eb+\sum_{i=1}^r
		\langle \ad_{f}^*e,u_i\rangle u_i
		=2y e+\frac{\rho^2}{3}u_0,\\
		\ad_{u}^*e&=\langle \ad_{u}^*e,\eb\rangle e+
		\langle \ad_{u}^*e,e\rangle \eb+\sum_{i=1}^r
		\langle \ad_{u}^*e,u_i\rangle u_i
		=\langle Z,u\rangle e.	
	\end{align*}So
	\[ \tr(J_e\circ J_e)=-\frac{\rho^4}{9}|u_0|^2. \]
	But, for any $u\in\G$, $\langle [e,u],e\rangle=0$ and hence
	\[ \ric(e,e)=-\frac12\tr(\ad_e^2)-\frac12\tr(\ad_e\circ\ad_e^*)-\frac14\tr(J_e^2)
	-\langle[H,e],e\rangle=-\frac{\rho^4}{36}|u_0|^2. \]
	Since $\Ric(e)=ae$ then $\ric(e,e)=0$ and hence $u_0=0$.
\end{proof}
So far we have shown that if $\G$ is curvature harmonic then $u_0=0$. Moreover, a direct computation shows that
\[ [e,[\eb,f]]+[\eb,[f,e]]+[f,[e,\eb]]=-\frac{9yx}2e+3y^2f. \]
Thus the Jacobi identity is satisfied for $(e,f,\eb)$ if and only if $y=0$. So $Z=-X$ and \eqref{b3a} becomes
\begin{equation}\label{b3abis} 
	\begin{cases}[e,\eb]=x e, [e,f]=0,[\eb,f]=ze-\frac12 xf,\\
		[e,u]=\langle X,u\rangle e+Au,\\
		[\eb,u]=\langle Y,u\rangle e-\langle X,u\rangle \eb+\langle T,u\rangle f+Bu,\\
		[f,u]=\langle U,u\rangle e+Cu	
	\end{cases}	,  
\end{equation}where $A,B,C:\h\too\h$ are skew-symmetric and $X=\rho(\rho Y+2T),U=2\rho Y+3T.$ This is equivalent to
\begin{equation}\label{relations}
 Y=\frac1{\rho}\left(2U-\frac3\rho X\right)\esp T=\frac2\rho X-U.
\end{equation}

 \begin{pr} We have
 	\[\ric(\eb,f)=\frac1{\rho}\left( |X|^2+\tr(\ad_X^\h)\right).  \]
 	In particular, $X\not=0$.
 \end{pr}
\begin{proof} We will use the formula \eqref{ricci}. Choose an orthonormal basis $(u_1,\ldots,u_r)$ of $\h$.
	\begin{align*}
		\tr(\ad_f\circ\ad_{\eb})&=\langle [f,[\eb,e]],\eb\rangle+\langle [f,[\eb,f]],f\rangle+\sum_{i=1}^r\langle [f,[\eb,u_i]],u_i\rangle\\
		&=\tr(CB),\\
		\tr(\ad_{f}^*\circ\ad_{\eb})&=	+\sum_{i=1}^r\langle [\eb,u_i],[f,u_i]\rangle\\
		&=-\tr(CB)-\langle X,U\rangle,\\
	 \tr(J_{f}\circ J_{\eb})
		&=-\langle \ad_{e}^*f,\ad_{\eb}^*\eb\rangle
		-\langle \ad_{\eb}^*f,\ad_{e}^*\eb\rangle-\langle \ad_{f}^*f,\ad_{f}^*\eb\rangle
		-\sum_{i=1}^r\langle \ad_{u_i}^*f,\ad_{u_i}^*\eb\rangle.
	\end{align*}
	\begin{align*}
		\ad_{\eb}^*\eb&=\langle \ad_{\eb}^*\eb,\eb\rangle e+
		\langle \ad_{\eb}^*\eb,e\rangle \eb+\langle \ad_{\eb}^*\eb,f\rangle f+\sum_{i=1}^r
		\langle \ad_{\eb}^*\eb,u_i\rangle u_i
		=-x\eb+zf+Y, \\
		\ad_{e}^*f&=\langle \ad_{e}^*f,\eb\rangle e+
		\langle \ad_{e}^*f,e\rangle \eb+\langle \ad_{e}^*f,f\rangle f+\sum_{i=1}^r
		\langle \ad_{e}^*f,u_i\rangle u_i
		=0,\\
		\ad_{\eb}^*f&=\langle \ad_{\eb}^*f,\eb\rangle e+
		\langle \ad_{\eb}^*f,e\rangle \eb+\langle \ad_{\eb}^*f,f\rangle f+\sum_{i=1}^r
		\langle \ad_{\eb}^*f,u_i\rangle u_i
		=-\frac12x f+T,
	\end{align*}
	\begin{align*}
		\ad_{e}^*\eb&=\langle \ad_{e}^*\eb,\eb\rangle e+
		\langle \ad_{e}^*\eb,e\rangle \eb+\langle \ad_{e}^*\eb,f\rangle f+\sum_{i=1}^r
		\langle \ad_{e}^*\eb,u_i\rangle u_i
		=xe+X,\\
		\ad_{f}^*f&=\langle \ad_{f}^*f,\eb\rangle e+
		\langle \ad_{f}^*f,e\rangle \eb+\sum_{i=1}^r
		\langle \ad_{f}^*f,u_i\rangle u_i
		=\frac12x e,\\
		\ad_{f}^*\eb&=\langle \ad_{f}^*\eb,\eb\rangle e+
		\langle \ad_{f}^*\eb,e\rangle \eb+\langle \ad_{f}^*\eb,f\rangle f+\sum_{i=1}^r
		\langle \ad_{f}^*\eb,u_i\rangle u_i
		=-ze+U,\\
		\ad_{u}^*\eb&=\langle \ad_{u}^*\eb,\eb\rangle e+
		\langle \ad_{u}^*\eb,e\rangle \eb+\langle \ad_{u}^*\eb,f\rangle f+\sum_{i=1}^r
		\langle \ad_{u}^*\eb,u_i\rangle u_i
		=-\langle Y,u\rangle e-\langle X,u\rangle \eb	-\langle U,u\rangle f,\\
		\ad_{u}^*f&=\langle \ad_{u}^*f,\eb\rangle e+
		\langle \ad_{u}^*f,e\rangle \eb+\langle \ad_{u}^*f,f\rangle f+\sum_{i=1}^r
		\langle \ad_{u}^*f,u_i\rangle u_i	
		=-\langle T,u\rangle e.
	\end{align*}
	So
	\[ \tr(J_{\eb}\circ J_f)=-2\langle X,T\rangle. \]
	Thus
	\[ \ric(\eb,f)=\frac12\langle X,T\rangle+\frac12\langle X,U\rangle+\frac12\langle[\eb,H],f\rangle+\frac12\langle[f,H],\eb\rangle. \]
	We have
	\[ \tr(\ad_e)=0,\tr(\ad_{\eb})=-x,\;\tr(\ad_f)=0,\tr(\ad_u)=\tr(\ad_u^\h). \]
	So
	\[ H=-xe+H^\h. \]
	\[ \ric(\eb,f)=\frac12\langle X,T\rangle+\frac12\langle X,U\rangle+\frac12\tr(\ad_T^\h)+\frac12\tr(\ad_U^\h). \]From \eqref{rica3}, we have $\ric(\eb,f)=1$ and hence $X\not=0$.
	\end{proof}

The details of the computation for the following proposition are give in the appendix.
\begin{pr}\label{Jacobia3} The bracket \eqref{b3abis} satisfies the Jacobi identity if and only  if for any $u,v\in\h$,
	\begin{equation}\label{jacobi} \begin{cases}[A,B]=xA,\; [A,C]=0\esp [B,C]=zA-\frac12xC,\\
			J_X=0, J_U=U\wedge X,\\
	 AX=CX=AT=AU=0,\\
			2zX+\frac12x U+CY-BU=0,xX+BX-AY=0,CT=\frac12xX,\\
			A[u,v]_\h=[u,Av]_\h+[Au,v]_\h-\langle X,v\rangle Au+\langle X,u\rangle Av,\\
			B[u,v]_\h=[u,Bv]_\h+[Bu,v]_\h
			-\langle Y,v\rangle Au+\langle X,v\rangle Bu
			-\langle T,v\rangle Cu
			+\langle Y,u\rangle Av-\langle X,u\rangle Bv
			+\langle T,u\rangle Cv,\\
			C[u,v]_\h=[u,Cv]_\h+[Cu,v]_\h-\langle U,v\rangle Au+\langle U,u\rangle Av,
			
	\end{cases} \end{equation}where $J_u,u\wedge v:\h\too\h$ are the skew-symmetric endomorphisms given by $J_uv=(\ad_v^\h)^*u$ and $(u\wedge v)(w)=\langle u,w\rangle v-\langle v,w\rangle u$.
	
\end{pr}

Let us reduce \eqref{jacobi} having in mind \eqref{relations}.

\begin{pr} The system  \eqref{jacobi} is equivalent to
	\[ \begin{cases}x=z=0, A=0, B=\frac{\langle T,X\rangle}{\langle X,X\rangle}C, CX=CU=0,\\
		C[u,v]_\h=[u,Cv]_\h+[Cu,v]_\h,\\
		X\in[\h,\h]^\perp, \langle U,[U,X]\rangle =\langle U,U\rangle \langle X,X\rangle-\langle X,U\rangle^2\esp
		\langle U,[\h,P^\perp]\rangle=0.
	\end{cases} \]for any $u,v\in\h$ and where $P=\mathrm{span}(X,U)$.
	\end{pr}
\begin{proof}
From the relations $AX=AU=0$ we deduce that $AY=0$ and hence $BX=-xX$. But $B$ is skew-symmetric and $X\not=0$ so $x=0$,  $BX=CT=CU=CY=0$ and $BU=2zX$. On the other hand, the last three equations in \eqref{jacobi} are equivalent to
\[ \begin{cases}
	A[u,v]_\h=[u,Av]_\h+[Au,v]_\h,\\
	B[u,v]_\h=[u,Bv]_\h+[Bu,v]_\h,
	\\
	C[u,v]_\h=[u,Cv]_\h+[Cu,v]_\h,\end{cases}\esp\begin{cases}
	[A,\ad_X^h]=\langle X,X\rangle A,\\
	[B,\ad_X^\h]=\langle Y,X\rangle A-\langle X,X\rangle B+\langle T,X\rangle C,\\
	[C,\ad_X^\h]=\langle U,X\rangle A,
\end{cases} \]for any $u,v\in X^\perp$.
Put $\ad_X^\h=(\ad_X^\h)^a+(\ad_X^\h)^s$ where $(\ad_X^\h)^a$ and $(\ad_X^\h)^s$ are the  skew-symmetric and the symmetric part, respectively. Since $[A,(\ad_X^\h)^s],[B,(\ad_X^\h)^s],[C,(\ad_X^\h)^s]$ are symmetric, we get
\[ \begin{cases}
	[A,(\ad_X^\h)^a]=\langle X,X\rangle A,\\
	[B,(\ad_X^\h)^a]=\langle Y,X\rangle A-\langle X,X\rangle B+\langle T,X\rangle C,\\
	[C,(\ad_X^\h)^a]=\langle U,X\rangle A.
\end{cases} \]
Thus $\mathrm{span}(A,C,(\ad_X^\h)^a)$ generates a solvable Lie subalgebra of $\mathrm{so}(\h)$ and hence it must be abelian so $A=0$. Also $\mathrm{span}(B,C,(\ad_X^\h)^a)$ does the same and hence 
$\di B=\frac{\langle T,X\rangle}{\langle X,X\rangle}C.$ But $CU=0$ and $BU=zX$ so $z=0$. 

Let us now solve the relations $J_X=0$ and $J_U=U\wedge X$. The relation $J_X=0$ is equivalent to $X\in[\h,\h]^\perp$. 
Denote by $P=\mathrm{span}(X,U)$. The relation $J_U=U\wedge X$ is equivalent to
\[ \begin{cases}
	\ad_U^*U=\langle U,U\rangle X-\langle X,U\rangle U,\\
	\ad_X^*U=\langle U,X\rangle X-\langle X,X\rangle U,\\
	\ad_p^*U=0,\quad p\in P^\perp.
\end{cases} \]This is equivalent to
\begin{equation}\label{bar} 
	\langle U,[U,X]\rangle =\langle U,U\rangle \langle X,X\rangle-\langle X,U\rangle^2\esp
	\langle U,[\h,P^\perp]\rangle=0, 
\end{equation}which completes the proof.
\end{proof}

So far, we have shown that \eqref{b3abis} becomes
\begin{equation} 
	\begin{cases}[e,\eb]= [e,f]=0=[\eb,f]=0,\\
			[e,u]=\langle X,u\rangle e,\\
			[\eb,u]=\langle Y,u\rangle e-\langle X,u\rangle \eb+\langle T,u\rangle f+\frac{\langle T,X\rangle}{\langle X,X\rangle}Cu,\\
			[f,u]=\langle U,u\rangle e+Cu
	\end{cases}	
\end{equation}where $C$ is a skew-symmetric derivation of $\h$  satisfying $CX=0$,  $Y=\frac1{\rho}\left(2U-\frac3\rho X\right)$, $ T=\frac2\rho X-U$ and $U$ satisfies \eqref{bar}.

The proof of the following proposition is given in the appendix.
\begin{pr}\label{riccia3} For any $u\in\h$, we have
	\[ \begin{cases}\ric(e,e)=\ric(e,f)=\ric(e,\eb)=\ric(f,f)=0,\\
		\ric(\eb,\eb)=2\langle X,Y\rangle-\frac12|T|^2+\frac12\langle U,U\rangle+\tr(\ad_Y^\h),\\
		\ric(f,\eb)=\frac1\rho(|X|^2+\tr(\ad_X^\h)).,\\
		\ric(e,u)=0,\\
		\ric(u,\eb)=0,\\
		\ric(u,f)=0,\\
		\ric(u,u)=\ric^\h(u,u),
	\end{cases} \]where $Y=\frac1{\rho}\left(2U-\frac3\rho X\right)$ and $ T=\frac2\rho X-U$.
	
\end{pr}
From \eqref{rica3} and the relation $\ric(e,\eb)=0$, we deduce $a=0$ and $\rho=-\al$. We have
\begin{align*}
0&=\ric(\eb,\eb)=2\langle X,\frac2\rho U-\frac3{\rho^2}X\rangle
-\frac12\left(\frac4{\rho^2}|X|^2+|U|^2-\frac4\rho \langle X,U\rangle\right)+\frac12\langle U,U\rangle+\frac2\rho\tr(\ad_U^\h)-\frac3{\rho^2}\tr(\ad_X^\h)\\
&=-\frac8{\rho^2}|X|^2+\frac6\rho\langle X,U\rangle	+\frac2\rho\tr(\ad_U^\h)-\frac3{\rho^2}\tr(\ad_X^\h),\\
1&=\ric(\eb,f)=\frac1\rho|X|^2+\frac1\rho\tr(\ad_X^\h).
	\end{align*}
So
\[ \begin{cases}
	\tr(\ad_X^\h)=-\al-|X|^2,\\
		\tr(\ad_U^\h)=-\frac5{2\al}|X|^2-3\langle X,U\rangle+\frac32.
\end{cases} \]
The last relation implies that $(\h,\prs)$ is a $\al$-Einstein Euclidean Lie algebra. If $\al>0$ then $\h$ is unimodular and  $\tr(\ad_X^\h)=0$. Thus $\al=-|X|^2$ which is a contradiction. So $\al<0$.	
To complete the proof, let us show that $\G$ is never Ricci-flat. Indeed, we have
\begin{eqnarray*}
	(\nabla_X\ric)(f,\eb)&=& \langle[X,\Ric f],\eb\rangle+\langle [\eb,\Ric f],X\rangle+ \langle [\eb,X],\Ric f\rangle
	-\langle [X,f],\Ric \eb\rangle\\&&-\langle [\Ric \eb,f],X\rangle-\langle [\Ric \eb,X],f\rangle\\
	&=& -2|X|^2\not=0. 
\end{eqnarray*}

\begin{exem} In Theorem \ref{mainbisa3}, there is a situation where the conditions on $X,U$ are easily satisfied. It is the case $\h$ is not unimodular and $X=U=aH$ where $H$ is given by the relation $\langle H,u \rangle=\tr(\ad_u^h)$. Then $\tr(\ad_U)=a|H|^2$ and $\al=-al-a^2l$, where $l=|H|^2$. Let us illustrate this situation when $\dim\h\in\{2,3\}$ to examples of 5-dimensional and 6-dimensional Lorentzian Lie algebras with harmonic curvature and not Ricci-parallel.
	\begin{enumerate}
		\item If $\dim\h=2$, let $e_1\in\h$ such that $\langle e_1,H \rangle=0$ and $\langle e_1,e_1 \rangle=1$ then $[H,e_1]=le_1$ and the Ricci operator of $\h$ is given by
		\[\Ric_\h=\begin{pmatrix}
			-l&0\\0&-l\end{pmatrix}.\]
		Then $\al=-l$, and we have $a^2l+al-l=0$ and so $\: a=\dfrac{-1\pm\sqrt 5}{2}$.  On the other hand $\tr\left(\ad_U^{\mathrm{b}}\right)=-\frac{5}{2 \alpha}|X|^2-3\langle X, U\rangle+\frac{3}{2} $, so 
		$$l=\dfrac{4\sqrt5\mp1}{3\sqrt5\mp1}$$
		
		The Lie bracket on $\g$ is given by  
		$$
		\left\{\begin{array}{l}
			{[e, \bar{e}]=[e, f]=0=[\bar{e}, f]=[e,e_1]=[\eb,e_1]=[f,e_1]=0,} \\
			{[e, H]= \dfrac{-1\pm\sqrt5}{2}le,} \\
			{[\bar{e}, H]=\dfrac{(2l-1)(-1\pm\sqrt 5)}{2l} e-\dfrac{-1\pm\sqrt 5}{2}l \bar{e}+\dfrac{(2-l)(-1\pm\sqrt 5)}{2} f} \\
			{[f, H]=\dfrac{-1\pm\sqrt 5}{2}l e}
		\end{array}\right.
		$$
		And the Ricci operator 
		\[\Ric=\begin{pmatrix}
			-l&0&0&0&0 \\
			0&-l&0&0&0 \\
			0&0&0&1&0  \\
			0&0&0&0&1\\
			0&0&0&0&0
		\end{pmatrix}\]
		Where $l=\dfrac{4\sqrt5\mp1}{3\sqrt5\mp1}$.

		\item If $\dim\h=3$, let $(H,e_1,e_2)$ be an orthogonal basis such 
		\[ [H,e_1]=(l-b)e_1,\quad [H,e_2]=ce_1+be_2,\qquad[e_1,e_2]=de_1+d'e_2 \]
		Since $\tr(\ad_{e_1})=\tr(\ad_{e_2})=0$ then $d=d'=0$. Then
		\[\Ric_\h=\begin{pmatrix}
			\dfrac{-4b^2+4bl - c^2 - 2l^2}{2l}& 0 & 0  \\ 
			0 & \dfrac{2bl + c^2 - 2l^2}{2l} & \dfrac{l-b}{l}c \\
			0 & \dfrac{l-b}{l}c & -\dfrac{2lb+c^2}{2l} 
		\end{pmatrix} \]
		If $c\not= 0$ then $b=l$ so $\Ric_\h(e_2)=\frac{c^2}{2l}e_2$ and $\Ric_\h(e_3)=-\frac{c^2+l^2}{2l}e_3$ which is impossible since $\h$ is Einstein. Then $c=0$ and so $b=\frac{l}{2}$ and $\al=-\frac{l}{2}$. We have now $a^2l+al-\frac{l}{2}=0$ so $a=\frac{-1-\sqrt3}{2}$. On the other hand $\tr\left(\ad_U^{\mathrm{b}}\right)=-\frac{5}{2 \alpha}|X|^2-3\langle X, U\rangle+\frac{3}{2} $, this gives
		$$l=2\left(\dfrac{4\sqrt3\mp1}{3\sqrt3\mp1}\right)$$
		
		The Lie bracket on $\g$ is given by  
		$$
		\left\{\begin{array}{l}
			{[e, \bar{e}]=[e, f]=0=[\bar{e}, f]=[e,e_1]=[e,e_2]=0}\\
			{[\eb,e_1]=\frac{4-l}{l}\lambda e_2 ,\quad [f,e_1]=\lambda e_2,} \\
			{[\eb,e_2]=-\frac{4-l}{l}\lambda e_1 , \quad[f,e_2]=-\lambda e_1,} \\
			{[e, H]= \dfrac{-1\pm\sqrt3}{2}le,} \\
			{[\bar{e}, H]=\dfrac{(2l-6)(-1\pm\sqrt 3)}{l} e-\dfrac{-1\pm\sqrt 3}{2}l \bar{e}+\dfrac{(4-l)(-1\pm\sqrt 3)}{2} f} \\
			{[f, H]=\dfrac{-1\pm\sqrt 3}{2}l e}
		\end{array}\right.
		$$
		And the Ricci operator 
		\[\Ric=\begin{pmatrix}
			-\frac{l}{2}&0&0&0&0&0 \\
			0&-\frac{l}{2}&0&0&0&0 \\
			0&0&-\frac{l}{2}&0&0&0  \\
			0&0&0&0&1& 0 \\
			0&0&0&0&0&1 \\
			0&0 &0 &0 & 0&0
		\end{pmatrix},\]
		where $l=2\left(\dfrac{4\sqrt3\mp1}{3\sqrt3\mp1}\right)$.
	\end{enumerate}
\end{exem}

\section{Appendix}\label{section7}

We start by giving the conditions for an operator to be a Codazzi operator, namely, to satisfy \eqref{co} in the different situations listed after Proposition \ref{list}.

\subsection{The conditions for $(l)$ to hold}
We distinguish three cases:
\begin{enumerate}
	\item $\fl=\mathrm{span}(e,\bar{e})$,  $A(e)=ae-b\bar{e}$ and $A(\bar{e})=be+a\bar{e}$.

	$\bullet$	For $u=e$,  $v=\bar{e}$ and $w=e$, 
	\[ 2\langle[e,\bar{e}],ae-b\bar{e}\rangle-
	\langle [ae,\bar{e}],e\rangle
	-\langle [e,a\bar{e}],e\rangle=
	\langle[\bar{e},e],ae-b\bar{e}\rangle
	-\langle[a\bar{e},e],e\rangle+
	\langle[-b\bar{e},e],\bar{e}\rangle \]
	So
	\[\boxed{\langle[e,\bar{e}],\bar{e}\rangle=0.}\]

	$\bullet$ For $u=e$,  $v=\bar{e}$ and $w=\bar{e}$,
	\[ 2\langle[e,\bar{e}],be+a\bar{e}\rangle
	-\langle [ae,\bar{e}],\bar{e}\rangle-\langle [e,a\bar{e}],\bar{e}\rangle=
	- \langle [be,\bar{e}],e\rangle
	+ \langle [ae,\bar{e}],\bar{e}\rangle
	-\langle[e,\bar{e}],be+a\bar{e}\rangle. \]
	
	So
	\[\boxed{ \langle[e,\bar{e}],e\rangle=0.} \]
	
	\item $\fl=\mathrm{span}(e,\bar{e})$,  $A(e)=ae$ and $A(\bar{e})= e+a\bar{e}$.
	
	$\bullet$ For $u=e$,  $v=\bar{e}$ and $w=e$,
	
	\[ 2\langle[e,\bar{e}],ae\rangle
	-\langle [ae,\bar{e}],e\rangle-\langle [e,a\bar{e}],e\rangle=
	\langle[\bar{e},e],ae\rangle- 
	\langle [a\bar{e},e],e\rangle.
	\]
	
	$\bullet$ For $u=e$,  $v=\bar{e}$ and $w=\bar{e}$, 
	\[ 2\langle[e,\bar{e}], e+a\bar{e}\rangle
	-\langle [ae,\bar{e}],\bar{e}\rangle-\langle [e,a\bar{e}],\bar{e}\rangle=
	- \langle [ e,\bar{e}],e\rangle
	+ \langle [ae,\bar{e}],\bar{e}\rangle
	-\langle[e,\bar{e}], e+a\bar{e}\rangle \]
	So
	\[\boxed{\langle[e,\bar{e}],e\rangle=0.} \]

	\item $\fl=\mathrm{span}(e,f,\bar{e})$,  $A(e)=a e$,
	$A(f)=e+af$ and $A(\bar{e})=f+a\bar{e}$.
	
	$\bullet$ For $u=e$,  $v=f$ and $w=e$,
	
	\[ 2\langle[e,f],ae\rangle
	-\langle [ae,f],e\rangle-\langle [e,af],e\rangle=
	\langle[f,e],ae\rangle- \langle [af,e],e\rangle
	\]
	This is true.

	$\bullet$ For $u=e$,  $v=f$ and $w=f$,
	
	\[ 2\langle[e,f],e+af\rangle
	-\langle [ae,f],f\rangle-\langle [e,af],f\rangle=
	- \langle [e,f],e\rangle
	+ \langle [ae,f],f\rangle
	-\langle[e,f],e+af\rangle \]
	So
	\[\boxed{ \langle[e,f],e\rangle=0.} \]

	$\bullet$ For $u=e$,  $v=f$ and $w=\bar{e}$,

	\[ 2\langle[e,f],f+a\eb\rangle
	-\langle [ae,f],\eb\rangle-\langle [e,af],\eb\rangle=
	\langle[f,\eb],ae\rangle- \langle [e+af,\eb],e\rangle
	+ \langle [ae,\eb],f\rangle
	-\langle[e,\eb],e+af\rangle. \]
	So
	\[ \boxed{\langle[e,f],f\rangle+\langle[e,\eb],e\rangle=0.} \]

	$\bullet$ For $u=e$,  $v=\bar{e}$ and $w=e$,
	
	\[ 2\langle[e,\eb],ae\rangle
	-\langle [ae,\eb],e\rangle-\langle [e,f+a\eb],e\rangle=
	\langle[\eb,e],ae\rangle- \langle [f+a\eb,e],e\rangle.
	\]So
	\[ \boxed{\langle[e,f],e\rangle=0.} \]

	$\bullet$ For $u=e$,  $v=\bar{e}$ and $w=\bar{e}$, 
	
	\[ 2\langle[e,\eb],f+a\eb\rangle
	-\langle [ae,\eb],\eb\rangle-\langle [e,f+a\eb],\eb\rangle=
	- \langle [f,\eb],e\rangle
	+ \langle [ae,\eb],\eb\rangle
	-\langle[e,\eb],f+a\eb\rangle. \]
	So
	\[\boxed{ 3\langle[e,\eb],f\rangle
		-\langle[e,f],\eb\rangle
		+\langle[f,\eb],e\rangle=0.} \]

	$\bullet$ For $u=e$,  $v=\bar{e}$ and $w=f$, 
	
	\[ 2\langle[e,\eb],e+af\rangle
	-\langle [ae,\eb],f\rangle-\langle [e,f+a\eb],f\rangle=
	\langle[\eb,f],ae\rangle- \langle [f+a\eb,f],e\rangle
	+ \langle [ae,f],\eb\rangle
	-\langle[e,f],f+a\eb\rangle. \]So
	\[ \boxed{\langle[e,\eb],e\rangle 
		=0.}\]

	$\bullet$ For $u=\eb$,  $v=f$ and $w=\eb$,
	
	\[ 2\langle[\eb,f],f+a\eb\rangle
	-\langle [a\eb,f],\eb\rangle-\langle [\eb,e+af],\eb\rangle=
	\langle[f,\eb],f+a\eb\rangle- \langle [e+af,\eb],\eb\rangle
	+ \langle [f,\eb],f\rangle.
	\]So
	\[\boxed{ 2\langle[\eb,f],f\rangle 
		-\langle[\eb,e],\eb\rangle=0.}\]

	$\bullet$ For $u=\eb$,  $v=f$ and $w=f$,

	\[ 2\langle[\eb,f],e+af\rangle
	-\langle [a\eb,f],f\rangle
	-\langle [\eb,e+af],f\rangle=
	- \langle [e,f],\eb\rangle
	+ \langle [a\eb,f],f\rangle
	-\langle[\eb,f],e+af\rangle \]
	\[\boxed{ 3\langle[\eb,f],e\rangle
		-\langle[\eb,e],f\rangle+\langle[e,f],\eb\rangle=0.} 
	\]

	$\bullet$ For $u=\eb$,  $v=f$ and $w=e$, 
	
	\[ 2\langle[\eb,f],ae\rangle
	-\langle [a\eb,f],e\rangle
	-\langle [\eb,e+af],e\rangle=
	\langle[f,e],f+a\eb\rangle
	- \langle [af,e],\eb\rangle
	+ \langle [f+a\eb,e],f\rangle
	-\langle[\eb,e],e+af\rangle. \]
	So
	\[ \boxed{\langle[f,e],f\rangle=0. }\]

\end{enumerate}

\subsection{The conditions for $(hl0)$ to hold}
We distinguish three cases:

\begin{enumerate}
	\item $\fl=\mathrm{span}(e,\eb)$, $A(e)=ae-b\eb$ and $A(\eb)=be+a\eb$.
	
	$\bullet$ For $u=u_i\in\h_i$, $v=v_j\in\h_j$ and $w=e$,

	\[ 2\langle[u_i,v_j],ae-b\eb\rangle
	-\langle [\al_iu_i,v_j],e\rangle-\langle [u_i,\al_j v_j],e\rangle=
	\langle[v_j,e],\al_i u_i\rangle- \langle [\al_jv_j,e],u_i\rangle
	+ \langle [\al_iu_i,e],v_j\rangle
	-\langle[u_i,e],\al_jv_j\rangle. \]
	So
	\[\boxed{ (2a-\al_i-\al_j)\langle[u_i,v_j],e\rangle
		-2b\langle[u_i,v_j],\eb\rangle+
		(\al_j-\al_i)\langle[u_i,e],v_j\rangle
		+(\al_j-\al_i)\langle[v_j,e],u_i\rangle=0.} \]
	
	$\bullet$ $u=u_i\in\h_i$, $v=v_j\in\h_j$ and $w=\eb$.
	
	\[ 2\langle[u_i,v_j],be+a\eb\rangle
	-\langle [\al_iu_i,v_j],\eb\rangle-\langle [u_i,\al_jv_j],\eb\rangle=
	\langle[v_i,\eb],\al_iu_i\rangle- \langle [\al_jv_j,\eb],u_i\rangle
	+ \langle [\al_iu_i,\eb],v_j\rangle
	-\langle[u_i,\eb],\al_jv_j\rangle. \]	
	
	\[\boxed{ 2b\langle[u_i,v_j],e\rangle
		+(2a-\al_i-\al_j)\langle[u_i,v_j],\eb\rangle+
		(\al_j-\al_i)\langle[u_i,\eb],v_j\rangle
		+(\al_j-\al_i)\langle[v_j,\eb],u_i\rangle=0.} \]
	
	\item $\fl=\mathrm{span}(e,\eb)$, $A(e)=ae$ and $A(\eb)=e+a\eb$.
	
	$\bullet$ $u=u_i\in\h_i$, $v=v_j\in\h_j$ and $w=e$.
	
	\[ 2\langle[u_i,v_i],ae\rangle
	-\langle [\al_iu_i,v_i],e\rangle-\langle [u_i,\al_jv_j],e\rangle=
	\langle[v_j,e],\al_iu_i\rangle- \langle [\al_jv_j,e],u_i\rangle
	+ \langle [\al_iu_i,e],v_j\rangle
	-\langle[u_i,e],\al_jv_j\rangle. \]
	So
	\[\boxed{ (2a-\al_i-\al_j)\langle[u_i,v_j],e\rangle
		+
		(\al_j-\al_i)\langle[u_i,e],v_j\rangle
		+(\al_j-\al_i)\langle[v_j,e],u_i\rangle=0.} \]

	$\bullet$ $u=u_i\in\h_i$, $v=v_j\in\h_j$ and $w=\eb$.
	
	\[ 2\langle[u_i,v_i],e+a\eb\rangle
	-\langle [\al_iu_i,v_j],\eb\rangle-\langle [u_i,\al_jv_j],\eb\rangle=
	\langle[v_i,\eb],\al_iu_i\rangle- \langle [\al_jv_j,\eb],u_i\rangle
	+ \langle [\al_iu_i,\eb],v_j\rangle
	-\langle[u_i,\eb],\al_jv_j\rangle. \]
	So
	\[\boxed{ 2\langle[u_i,v_j],e\rangle
		+(2a-\al_i-\al_j)\langle[u_i,v_j],\eb\rangle+
		(\al_j-\al_i)\langle[u_i,\eb],v_j\rangle
		+(\al_j-\al_i)\langle[v_j,\eb],u_i\rangle=0.} \]
	
	\item $\fl=\mathrm{span}(e,f,\bar{e})$,  $A(e)=a e$,
	$A(f)=e+af$ and $A(\bar{e})=f+a\bar{e}$.
	
	$\bullet$ For $u=u_i\in\h_i$,  $v=v_j\in\h_j$ and $w=e$.
	
	$$	2\langle[u_i,v_j],ae\rangle
	-\langle [\al_iu_i,v_j],e\rangle-\langle [u_i,\al_jv_j],e\rangle=
	\langle[v_j,e],\al_iu_i\rangle- \langle [\al_jv_j,e],u_i\rangle
	+ \langle [\al_iu_i,e],v_j\rangle
	-\langle[u_i,e],\al_jv_j\rangle.$$
	
	So
	\[\boxed{ (2a-\al_i-\al_j)\langle[u_i,v_j],e\rangle
		+
		(\al_j-\al_i)\langle[u_i,e],v_j\rangle
		+(\al_j-\al_i)\langle[v_j,e],u_i\rangle=0.} \]

	$\bullet$ For $u=u_i\in\h_i$,  $v=v_j\in\h_j$ and $w=f$.
	
	$$	2\langle[u_i,v_j],e+af\rangle
	-\langle [\al_iu_i,v_j],f\rangle-\langle [u_i,\al_jv_j],f\rangle=
	\langle[v_i,f],\al_iu_i\rangle- \langle [\al_jv_j,f],u_i\rangle
	+ \langle [\al_iu_i,f],v_j\rangle
	-\langle[u_i,f],\al_jv_j\rangle.$$
	\[\boxed{2\langle[u_i,v_j],e\rangle+ (2a-\al_i-\al_j)\langle[u_i,v_j],f\rangle
		+
		(\al_j-\al_i)\langle[u_i,f],v_j\rangle
		+(\al_j-\al_i)\langle[v_j,f],u_i\rangle=0.} \]

	$\bullet$ For $u=u_i\in\h_i$,  $v=v_j\in\h_j$ and $w=\eb$.
	
	$$	2\langle[u_i,v_j],f+a\eb\rangle
	-\langle [\al_iu_i,v_j],\eb\rangle-\langle [u_i,\al_jv_j],\eb\rangle=
	\langle[v_j,\eb],\al_iu_i\rangle- \langle [\al_jv_j,\eb],u_i\rangle
	+ \langle [\al_iu_i,\eb],v_j\rangle
	-\langle[u_i,\eb],\al_jv_j\rangle.$$

	\[\boxed{2\langle[u_i,v_j],f\rangle+ (2a-\al_i-\al_j)\langle[u_i,v_j],\eb\rangle
		+
		(\al_j-\al_i)\langle[u_i,\eb],v_j\rangle
		+(\al_j-\al_i)\langle[v_j,\eb],u_i\rangle=0.} \]

\end{enumerate}

\subsection{The conditions for $(hl1)$ to hold}
We distinguish three cases:

\begin{enumerate}
	\item $\fl=\mathrm{span}(e,\eb)$, $A(e)=ae-b\eb$ and $A(\eb)=be+a\eb$.
	
	$\bullet$ For $u=u_i\in\h_i$, $w=w_j\in\h_j$, $v=e$.
	
	\[ 2\langle[u_i,e],\al_jw_j\rangle
	-\langle [\al_iu_i,e],w_j\rangle-\langle [u_i,ae-b\eb],w_j\rangle=
	\langle[e,w_j],\al_iu_i\rangle- \langle [ae-b\eb,w_j],u_i\rangle
	+ \langle [\al_iu_i,w_j],e\rangle
	-\langle[u_i,w_j],ae-b\eb\rangle. \]
	So
	\[ \boxed{(2\al_j-\al_i-a)\langle[u_i,e],w_j\rangle
		+b\langle[u_i,\eb],w_j\rangle	
		+(a-\al_i)\langle[e,w_j],u_i\rangle-b\langle[\eb,w_j],u_i\rangle
		+(a-\al_i)\langle[u_i,w_j],e\rangle-b
		\langle[u_i,w_j],\eb\rangle
		=0.
	} \]

	$\bullet$ For $u=u_i\in\h_i$, $w=w_j\in\h_j$, $v=\eb$.
	
	\[ 2\langle[u_i,\eb],\al_jw_j\rangle
	-\langle [\al_iu_i,\eb],w_j\rangle-\langle [u_i,be+a\eb],w_j\rangle=
	\langle[\eb,w_j],\al_iu_i\rangle- \langle [be+a\eb,w_j],u_j\rangle
	+ \langle [\al_iu_i,w_j],\eb\rangle
	-\langle[u_i,w_j],be+a\eb\rangle. \]
	So
	\[ \boxed{(2\al_j-\al_i-a)\langle[u_i,\eb],w_j\rangle
		-b\langle[u_i,e],w_j\rangle	
		+(a-\al_i)\langle[\eb,w_j],u_i\rangle+b\langle[e,w_j],u_i\rangle
		+(a-\al_i)\langle[u_i,w_j],\eb\rangle+b
		\langle[u_i,w_j],e\rangle
		=0.
	} \]

	\item $\fl=\mathrm{span}(e,\eb)$, $A(e)=ae$ and $A(\eb)=e+a\eb$.
	
	$\bullet$ For $u=u_i\in\h_i$, $w=w_j\in\h_j$, $v=e$.
	
	\[ 2\langle[u_i,e],\al_jw_j\rangle
	-\langle [\al_iu_i,e],w_j\rangle-\langle [u_i,ae],w_j\rangle=
	\langle[e,w_j],\al_iu_i\rangle- \langle [ae,w_j],u_i\rangle
	+ \langle [\al_iu_i,w_j],e\rangle
	-\langle[u_i,w_j],ae\rangle. \]
	So
	\[\boxed{ (2\al_j-\al_i-a)\langle[u_i,e],w_j\rangle+(a-\al_i)\langle[e,w_j],u_i\rangle+
		(a-\al_i)\langle[u_i,w_j],e\rangle=0.} \]

	$\bullet$ For $u=u_i\in\h_i$, $w=w_j\in\h_j$, $v=\eb$.
	
	\[ 2\langle[u_i,\eb],\al_jw_j\rangle
	-\langle [\al_iu_i,\eb],w_j\rangle-\langle [u_i,e+a\eb],w_j\rangle=
	\langle[\eb,w_j],\al_iu_i\rangle- \langle [e+a\eb,w_j],u_i\rangle
	+ \langle [\al_iu_i,w_j],\eb\rangle
	-\langle[u_i,w_i],e+a\eb\rangle. \]
	So
	\[\boxed{-\langle[u_i,e],w_j\rangle-
		\langle[w_j,e],u_i\rangle+ (2\al_j-\al_i-a)\langle[u_i,\eb],w_j\rangle+(a-\al_i)\langle[\eb,w_j],u_i\rangle+
		(a-\al_i)\langle[u_i,w_j],\eb\rangle
		+\langle[u_i,w_j],e\rangle=0.} \]

	\item $\fl=\mathrm{span}(e,f,\eb)$, $A(e)=ae$, $A(f)=e+af$ and $A(\eb)=f+a\eb$.

	$\bullet$ For $u=u_i\in\h_i$, $w=w_j\in\h_j$, $v=e$.
	
	\[ 2\langle[u_i,e],\al_jw_j\rangle
	-\langle [\al_iu_i,e],w_j\rangle-\langle [u_i,ae],w_j\rangle=
	\langle[e,w_j],\al_iu_i\rangle- \langle [ae,w_j],u_i\rangle
	+ \langle [\al_iu_i,w_j],e\rangle
	-\langle[u_i,w_j],ae\rangle. \]
	So
	\[\boxed{ (2\al_j-\al_i-a)\langle[u_i,e],w_j\rangle+(a-\al_i)\langle[e,w_j],u_i\rangle+
		(a-\al_i)\langle[u_i,w_j],e\rangle=0.} \]

	$\bullet$ For $u=u_i\in\h_i$, $w=w_j\in\h_j$, $v=f$.
	
	\[ 2\langle[u_i,f],\al_jw_j\rangle
	-\langle [\al_iu_i,f],w_j\rangle-\langle [u_i,e+af],w_j\rangle=
	\langle[f,w_j],\al_iu_i\rangle- \langle [e+af,w_j],u_i\rangle
	+ \langle [\al_iu_i,w_j],f\rangle
	-\langle[u_i,w_j],e+af\rangle. \]
	
	So
	\[\boxed{-\langle[u_i,e],w_j\rangle-\langle[w_j,e],u_i\rangle+ (2\al_j-\al_i-a)\langle[u_i,f],w_j\rangle+(a-\al_i)\langle[f,w_j],u_i\rangle+
		(a-\al_i)\langle[u_i,w_j],f\rangle
		+\langle[u_i,w_j],e\rangle=0.} \]

	$\bullet$ For $u=u_i\in\h_i$, $w=w_j\in\h_j$, $v=\eb$.
	
	\[ 2\langle[u_i,\eb],\al_jw_j\rangle
	-\langle [\al_iu_i,\eb],w_j\rangle-\langle [u_i,f+a\eb],w_j\rangle=
	\langle[\eb,w_j],\al_iu_i\rangle- \langle [f+a\eb,w_j],u_i\rangle
	+ \langle [\al_iu_i,w_j],\eb\rangle
	-\langle[u_i,w_j],f+a\eb\rangle. \]
	So
	\[\boxed{-\langle[u_i,f],w_j\rangle-\langle[w_j,f],u_i\rangle+ (2\al_j-\al_i-a)\langle[u_i,\eb],w_j\rangle+(a-\al_i)\langle[\eb,w_j],u_i\rangle+
		(a-\al_i)\langle[u_i,w_j],\eb\rangle
		+\langle[u_i,w_j],f\rangle=0.} \]

\end{enumerate}

\subsection{The conditions for $(lh0)$ to hold}
We distinguish three cases:

\begin{enumerate}
	\item $\fl=\mathrm{span}(e,\eb)$, $A(e)=ae-b\eb$ and $A(\eb)=be+a\eb$.
	
	$\bullet$ For $u=e$, $v=\eb$, $w=w_i\in\h_i$,
	
	\[ 2\langle[e,\eb],\al_iw_i\rangle
	-\langle [ae,\eb],w_i\rangle-\langle [e,a\eb],w_i\rangle=
	\langle[\eb,w_i],ae-b\eb\rangle- \langle [be+a\eb,w_i],e\rangle
	+ \langle [ae-b\eb,w_i],\eb\rangle
	-\langle[e,w_i],be+a\eb\rangle. \]
	So
	\[ \boxed{ (2\al_i-2a)\langle[e,\eb],w_i\rangle
		+2b\langle[\eb,w_i],\eb\rangle
		+2b\langle[e,w_i],e\rangle
		=0.
	} \]

	\item $\fl=\mathrm{span}(e,\eb)$, $A(e)=ae$ and $A(\eb)=e+a\eb$.
	
	$\bullet$ For $u=e$, $v=\eb$ and $w=w_i\in\h_i$.
	
	\[ 2\langle[e,\eb],\al_iw_i\rangle
	-\langle [ae,\eb],w_i\rangle-\langle [e,a\eb],w_i\rangle=
	\langle[\eb,w_i],ae\rangle- \langle [e+a\eb,w_i],e\rangle
	+ \langle [ae,w_i],\eb\rangle
	-\langle[e,w_i],e+a\eb\rangle. \]
	So
	\[ \boxed{ (2\al_i-2a)\langle[e,\eb],w_i\rangle
		+2\langle[e,w_i],e\rangle
		=0.
	} \]

	\item $\fl=\mathrm{span}(e,f,\eb)$, $A(e)=ae$, $A(f)=e+af$ and $A(\eb)=f+a\eb$.
	
	$\bullet$ For $u=e$, $v=f$ and $w=w_i\in\h_i$,

	\[ 2\langle[e,f],\al_iw_i\rangle
	-\langle [ae,f],w_i\rangle-\langle [e,af],w_i\rangle=
	\langle[f,w_i],ae\rangle- \langle [e+af,w_i],e\rangle
	+ \langle [ae,w_i],f\rangle
	-\langle[e,w_i],e+af\rangle. \]
	So
	\[ \boxed{(2\al_i-2a)\langle[e,f],w_i\rangle
		+2\langle[e,w_i],e\rangle
		=0.} \]
	
	$\bullet$ For $u=e$, $v=\eb$ and $w=w_i\in\h_i$,
	
	\[ 2\langle[e,\eb],\al_iw_i\rangle
	-\langle [ae,\eb],w_i\rangle-\langle [e,f+a\eb],w_i\rangle=
	\langle[\eb,w_i],ae\rangle- \langle [f+a\eb,w_i],e\rangle
	+ \langle [ae,w_i],\eb\rangle
	-\langle[e,w_i],f+a\eb\rangle. \]
	
	So
	\[ \boxed{(2\al_i-2a)\langle[e,\eb],w_i\rangle
		-\langle[e,f],w_i\rangle+\langle[f,w_i],e\rangle+\langle[e,w_i],f\rangle
		=0.} \]

	$\bullet$ For $u=f$, $v=\eb$ and $w=w_i\in\h_i$,

	\[ 2\langle[f,\eb],\al_iw_i\rangle
	-\langle [e+af,\eb],w_i\rangle-\langle [f,f+a\eb],w_i\rangle=
	\langle[\eb,w_i],e+af\rangle- \langle [f+a\eb,w_i],f\rangle
	+ \langle [e+af,w_i],\eb\rangle
	-\langle[f,w_i],f+a\eb\rangle. \]
	
	\[ \boxed{(2\al_i-2a)\langle[f,\eb],w_i\rangle
		-\langle[e,\eb],w_i\rangle-
		\langle[\eb,w_i],e\rangle
		-\langle[e,w_i],\eb\rangle+2\langle[f,w_i],f\rangle
		=0.} \]

\end{enumerate}

\subsection{The conditions for $(lh1)$ to hold}
We distinguish three cases:

\[ 2\langle[u,v],A(w)\rangle
-\langle [A(u),v],w\rangle-\langle [u,A(v)],w\rangle=
\langle[v,w],A(u)\rangle- \langle [A(v),w],u\rangle
+ \langle [A(u),w],v\rangle
-\langle[u,w],A(v)\rangle.
\]

\begin{enumerate}
	\item $\fl=\mathrm{span}(e,\eb)$, $A(e)=ae-b\eb$ and $A(\eb)=be+a\eb$.
	
	$\bullet$ For $u=e$, $w=e$ and $v=v_i\in\h_i$
	
	\[ 2\langle[e,v_i],ae-b\eb\rangle
	-\langle [ae-b\eb,v_i],e\rangle-\langle [e,\al_iv_i],e\rangle=
	\langle[v_i,e],ae-b\eb\rangle- \langle [\al_iv_i,e],e\rangle
	+ \langle [ae-b\eb,e],v_i\rangle. \]
	So
	\[ \boxed{2(a-\al_i)\langle[e,v_i],e\rangle
		-3b\langle[e,v_i],\eb\rangle
		+b\langle[\eb,v_i],e\rangle+b\langle[\eb,e],v_i\rangle=0.} \]

	$\bullet$ For $u=\eb$, $w=\eb$ and $v=v_i\in\h_i$
	
	\[ 2\langle[\eb,v_i],be+a\eb\rangle
	-\langle [be+a\eb,v_i],\eb\rangle-\langle [\eb,\al_iv_i],\eb\rangle=
	\langle[v_i,\eb],be+a\eb\rangle- \langle [\al_iv_i,\eb],\eb\rangle
	+ \langle [be+a\eb,\eb],v_i\rangle
	. \]
	So
	\[ \boxed{2(a-\al_i)\langle[\eb,v_i],\eb\rangle
		+3b\langle[\eb,v_i],e\rangle
		-b\langle[e,v_i],\eb\rangle+b\langle[\eb,e],v_i\rangle=0.} \]
	
	$\bullet$ For $u=e$, $w=\eb$ and $v=v_i\in\h_i$
	
	\[ 2\langle[e,v_i],be+a\eb\rangle
	-\langle [ae-b\eb,v_i],\eb\rangle-\langle [e,\al_iv_i],\eb\rangle=
	\langle[v_i,\eb],ae-b\eb\rangle- \langle [\al_iv_i,\eb],e\rangle
	+ \langle [ae-b\eb,\eb],v_i\rangle
	-\langle[e,\eb],\al_iv_i\rangle. \]
	So
	\[ \boxed{2b\langle[e,v_i],e\rangle
		+(a-\al_i)\langle[e,v_i],\eb\rangle+
		(a-\al_i)\langle[\eb,v_i],e\rangle
		+(\al_i-a)\langle[e,\eb],v_i\rangle=0.} \]

	$\bullet$ For $u=\eb$, $w=e$ and $v=v_i\in\h_i$
	
	\[ 2\langle[\eb,v_i],ae-b\eb\rangle
	-\langle [be+a\eb,v_i],e\rangle-\langle [\eb,\al_iv_i],e\rangle=
	\langle[v_i,e],be+a\eb\rangle- \langle [\al_iv_i,e],\eb\rangle
	+ \langle [be+a\eb,e],v_i\rangle
	-\langle[\eb,e],\al_iv_i\rangle. \]
	So
	\[ \boxed{ (a-\al_i)\langle[\eb,v_i],e\rangle
		-2b\langle[\eb,v_i],\eb\rangle
		+(a-\al_i)\langle[e,v_i],\eb\rangle
		+(\al_i-a)\langle[\eb,e],v_i\rangle=0.    } \]

	\item $\fl=\mathrm{span}(e,\eb)$, $A(e)=ae$ and $A(\eb)=e+a\eb$.
	
	$\bullet$ For $u=e$, $w=e$ and $v=v_i\in\h_i$.
	
	\[ 2\langle[e,v_i],ae\rangle
	-\langle [ae,v_i],e\rangle-\langle [e,\al_iv_i],e\rangle=
	\langle[v_i,e],ae\rangle- \langle [\al_iv_i,e],e\rangle
	. \]
	So
	\[ \boxed{2(a-\al_i)\langle[e,v_i],e\rangle=0. } \]

	$\bullet$ For $u=\eb$, $w=\eb$ and $v=v_i\in\h_i$.
	
	\[ 2\langle[\eb,v_i],e+a\eb\rangle
	-\langle [e+a\eb,v_i],\eb\rangle-\langle [\eb,\al_iv_i],\eb\rangle=
	\langle[v_i,\eb],e+a\eb\rangle- \langle [\al_iv_i,\eb],\eb\rangle
	+ \langle [e,\eb],v_i\rangle
	. \]
	So
	\[ \boxed{3\langle[\eb,v_i],e\rangle
		+2(a-\al_i)\langle[\eb,v_i],\eb\rangle-\langle[e,v_i],\eb\rangle
		-\langle [e,\eb],v_i\rangle=0.	
	} \]

	$\bullet$ For $u=e$, $w=\eb$ and $v=v_i\in\h_i$.
	
	\[ 2\langle[e,v_i],e+a\eb\rangle
	-\langle [ae,v_i],\eb\rangle-\langle [e,\al_iv_i],\eb\rangle=
	\langle[v_i,\eb],ae\rangle- \langle [\al_iv_i,\eb],e\rangle
	+ \langle [ae,\eb],v_i\rangle
	-\langle[e,\eb],\al_iv_i\rangle. \]
	So
	\[ \boxed{ 2\langle[e,v_i],e\rangle +(a-\al_i)\langle[e,v_i],\eb\rangle 
		+(a-\al_i)\langle[\eb,v_i],e\rangle 
		+(\al_i-a)\langle[e,\eb],v_i\rangle=0.
	} \]

	$\bullet$ For $u=\eb$, $w=e$ and $v=v_i\in\h_i$.
	
	\[ 2\langle[\eb,v_i],ae\rangle
	-\langle [e+a\eb,v_i],e\rangle-\langle [\eb,\al_iv_i],e\rangle=
	\langle[v_i,e],e+a\eb\rangle- \langle [\al_iv_i,e],\eb\rangle
	+ \langle [e+a\eb,e],v_i\rangle
	-\langle[\eb,e],\al_iv_i\rangle. \]
	So
	\[ \boxed{(a-\al_i)\langle[\eb,v_i],e\rangle
		+(a-\al_i)\langle[e,v_i],\eb\rangle
		+(\al_i-a)\langle[\eb,e],v_i\rangle=0.} \]

	\item $\fl=\mathrm{span}(e,f,\eb)$, $A(e)=ae$, $A(f)=e+af$ and $A(\eb)=f+a\eb$.
	
	$\bullet$ For $u=e$, $w=e$ and $v=v_i\in\h_i$.
	
	\[ 2\langle[e,v_i],ae\rangle
	-\langle [ae,v_i],e\rangle-\langle [e,\al_iv_i],e\rangle=
	\langle[v_i,e],ae\rangle
	- \langle [\al_iv_i,e],e\rangle
	. \]
	So
	\[ \boxed{ 2(a-\al_i)\langle[e,v_i],e\rangle=0.    } \]

	$\bullet$ For $u=f$, $w=f$ and $v=v_i\in\h_i$.
	
	\[ 2\langle[f,v_i],e+af\rangle
	-\langle [e+af,v_i],f\rangle-\langle [f,\al_iv_i],f\rangle=
	\langle[v_i,f],e+af\rangle- \langle [\al_iv_i,f],f\rangle
	+ \langle [e+af,f],v_i\rangle
	. \]
	So
	\[ \boxed{ 3\langle[f,v_i],e\rangle +
		2(a-\al_i)\langle[f,v_i],f\rangle -\langle[e,v_i],f\rangle
		-\langle[e,f],v_i\rangle=0.
	} \]

	$\bullet$ For $u=\eb$, $w=\eb$ and $v=v_i\in\h_i$.
	
	\[ 2\langle[\eb,v_i],f+a\eb\rangle
	-\langle [f+a\eb,v_i],\eb\rangle-\langle [\eb,\al_iv_i],\eb\rangle=
	\langle[v_i,\eb],f+a\eb\rangle- \langle [\al_iv_i,\eb],\eb\rangle
	+ \langle [f+a\eb,\eb],v_i\rangle
	. \]
	So
	\[ \boxed{3\langle[\eb,v_i],f\rangle
		+2(a-\al_i)\langle[\eb,v_i],\eb\rangle
		-\langle [f,v_i],\eb\rangle
		-\langle [f,\eb],v_i\rangle=0.
	} \]

	$\bullet$ For $u=e$, $w=f$ and $v=v_i\in\h_i$.
	
	\[ 2\langle[e,v_i],e+af\rangle
	-\langle [ae,v_i],f\rangle-\langle [e,\al_iv_i],f\rangle=
	\langle[v_i,f],ae\rangle- 
	\langle [\al_iv_i,f],e\rangle
	+ \langle [ae,f],v_i\rangle
	-\langle[e,f],\al_iv_i\rangle. \]
	So
	\[ \boxed{2\langle[e,v_i],e\rangle
		+(a-\al_i)\langle[e,v_i],f\rangle
		+(a-\al_i)\langle[f,v_i],e\rangle
		+(\al_i-a)\langle[e,f],v_i\rangle=0.
	} \]

	$\bullet$ For $u=f$, $w=e$ and $v=v_i\in\h_i$.
	
	\[ 2\langle[f,v_i],ae\rangle
	-\langle [e+af,v_i],e\rangle-\langle [f,\al_iv_i],e\rangle=
	\langle[v_i,e],e+af\rangle- \langle [\al_iv_i,e],f\rangle
	+ \langle [e+af,e],v_i\rangle
	-\langle[f,e],\al_iv_i\rangle. \]
	So
	\[ \boxed{(a-\al_i)\langle[f,v_i],e\rangle
		+(a-\al_i)\langle[e,v_i],f\rangle
		+(a-\al_i)\langle[e,f],v_i\rangle=0.} \]

	$\bullet$ For $u=e$, $w=\eb$ and $v=v_i\in\h_i$.
	
	\[ 2\langle[e,v_i],f+a\eb\rangle
	-\langle [ae,v_i],\eb\rangle-\langle [e,\al_iv_i],\eb\rangle=
	\langle[v_i,\eb],ae\rangle- \langle [\al_iv_i,\eb],e\rangle
	+ \langle [ae,\eb],v_i\rangle
	-\langle[e,\eb],\al_iv_i\rangle. \]
	So
	\[ \boxed{2\langle[e,v_i],f\rangle
		+(a-\al_i)\langle[e,v_i],\eb\rangle
		+(a-\al_i)\langle[\eb,v_i],e\rangle
		+(a-\al_i)\langle[\eb,e],v_i\rangle=0.
	} \]

	$\bullet$ For $u=\eb$, $w=e$ and $v=v_i\in\h_i$.
	
	\[ 2\langle[\eb,v_i],ae\rangle
	-\langle [f+a\eb,v_i],e\rangle-\langle [\eb,\al_iv_i],e\rangle=
	\langle[v_i,e],f+a\eb\rangle- \langle [\al_iv_i,e],\eb\rangle
	+ \langle [f+a\eb,e],v_i\rangle
	-\langle[\eb,e],\al_iv_i\rangle. \]
	So
	\[ \boxed{ (a-\al_i)\langle[\eb,v_i],e\rangle
		+(a-\al_i)\langle[e,v_i],\eb\rangle
		-\langle [f,v_i],e\rangle
		+\langle[e,v_i],f\rangle+\langle[e,f],v_i\rangle
		+(a-\al_i)\langle[e,\eb],v_i\rangle	=0.
	} \]

	$\bullet$ For $u=f$, $w=\eb$ and $v=v_i\in\h_i$.
	
	\[ 2\langle[f,v_i],f+a\eb\rangle
	-\langle [e+af,v_i],\eb\rangle-\langle [f,\al_iv_i],\eb\rangle=
	\langle[v_i,\eb],e+af\rangle- \langle [\al_iv_i,\eb],f\rangle
	+ \langle [e+af,\eb],v_i\rangle
	-\langle[f,\eb],\al_iv_i\rangle. \]
	So
	\[ \boxed{2\langle[f,v_i],f\rangle
		+(a-\al_i)\langle[f,v_i],\eb\rangle
		-\langle [e,v_i],\eb\rangle
		+\langle[\eb,v_i],e\rangle
		+(a-\al_i)\langle[\eb,v_i],f\rangle
		+(\al_i-a)\langle[f,\eb],v_i\rangle
		-\langle[e,\eb],v_i\rangle=0.
	} \]

	$\bullet$ For $u=\eb$, $w=f$ and $v=v_i\in\h_i$.
	
	\[ 2\langle[\eb,v_i],e+af\rangle
	-\langle [f+a\eb,v_i],f\rangle-\langle [\eb,\al_iv_i],f\rangle=
	\langle[v_i,f],f+a\eb\rangle- \langle [\al_iv_i,f],\eb\rangle
	+ \langle [f+a\eb,f],v_i\rangle
	-\langle[\eb,f],\al_iv_i\rangle. \]
	So
	\[ \boxed{ 2\langle[\eb,v_i],e\rangle
		+(a-\al_i)\langle[\eb,v_i],f\rangle
		+(a-\al_i)\langle[f,v_i],\eb\rangle
		+(a-\al_i)\langle[f,\eb],v_i\rangle=0.
	} \]

\subsection{Proof of Proposition \ref{riccia21}}	
Recall that
\begin{equation*} [e,\eb]=0, [e,u]=\langle X,u\rangle e,[\eb,u]=\rho\langle X,u\rangle e-\langle X,u\rangle\eb,  
\end{equation*}for any $u\in\h$.

	We will use the formula \eqref{ricci}. Choose an orthonormal basis $(u_1,\ldots,u_r)$ of $\h$. We have shown in Proposition \ref{para2} that $H=H_0\in\h$.
	
	Let us compute $\ric(e,e)$.
	 \begin{align*}
		\tr(\ad_{e}^2)&=\langle [e,[e,\eb]],e\rangle
		+\sum_{i=1}^r\langle [e,[e,u_i]],u_i\rangle
		=0.\\
		\tr(\ad_{e}^*\circ\ad_{e})&=
		\sum_{i=1}^r\langle [e,u_i],[e,u_i]\rangle
		=0,\\
		\tr(J_{e}^2)
		&=-2\langle \ad_{e}^*e,\ad_{\eb}^*e\rangle
		-\sum_{i=1}^r\langle \ad_{u_i}^*e,\ad_{u_i}^*e\rangle,\\
		\ad_{e}^*e&=\langle \ad_{e}^*e,\eb\rangle e+\langle \ad_{e}^*e,e\rangle\eb+\sum_{i=1}^r\langle \ad_{e}^*e,u_i\rangle u_i
		=0 ,\\
		\ad_{u}^*e&=\langle \ad_{u}^*e,\eb\rangle e+
		\langle \ad_{u}^*e,e\rangle \eb+\sum_{i=1}^r
		\langle \ad_{u}^*e,u_i\rangle u_i=
		 \langle X,u\rangle e.
	\end{align*}So $\tr(J_{e}^2)=0.$
	Moreover, $\langle [H,e],e\rangle=0$ and hence $\ric(e,e)=0$. 
	
	Let us compute $\ric(e,\eb)$.
	\begin{align*}
		\tr(\ad_{e}\circ\ad_{\eb})&=
		\sum_{i=1}^r\langle [e,[\eb,u_i]],u_i\rangle
		=0\\
		\tr(\ad_{\eb}^*\circ\ad_{e})&=
		\sum_{i=1}^r\langle [e,u_i],[\eb,u_i]\rangle
		=-|X|^2,\\
		\tr(J_{e}\circ J_{\eb})
		&=-\langle \ad_{e}^*\eb,\ad_{\eb}^*e\rangle-\langle \ad_{\eb}^*\eb,\ad_{e}^*e\rangle
		-\sum_{i=1}^r\langle \ad_{u_i}^*\eb,\ad_{u_i}^*e\rangle,\\
		\ad_{e}^*e&
		=0 ,\\
		\ad_{\eb}^*e&=\sum_{i=1}^r\langle \ad_{\eb}^*e,u_i\rangle u_i=-X\\
		\ad_{e}^*\eb&=\sum_{i=1}^r\langle \ad_{e}^*\eb,u_i\rangle u_i=X\\
		\ad_{u}^*e&=
		\langle X,u\rangle e,\\
		\ad_{u}^*\eb&=\langle \ad_{u}^*\eb,\eb\rangle e+
		\langle \ad_{u}^*\eb,e\rangle \eb+\sum_{i=1}^r
		\langle \ad_{u}^*\eb,u_i\rangle u_i
		=-\rho\langle X,u\rangle e-\langle X,u\rangle\eb.
		\end{align*}So $\tr(J_{e}\circ J_{\eb})=2|X|^2.$
	Moreover, $\langle [H,e],\eb\rangle=-\tr(\ad_X^\h)$, $\langle [H,\eb],e\rangle=\tr(\ad_X^\h)$  and hence $\ric(e,\eb)=0$.

	Let us compute $\ric(e,u)$ for $u\in\h$. 
	\begin{align*}
		\tr(\ad_{e}\circ\ad_{u})&=
		\sum_{i=1}^r\langle [e,[u,u_i]],u_i\rangle
		=0\\
		\tr(\ad_{u}^*\circ\ad_{e})&=
		\sum_{i=1}^r\langle [e,u_i],[u,u_i]\rangle
		=0,\\
		\tr(J_{e}\circ J_{u})
		&=-\langle \ad_{e}^*e,\ad_{\eb}^*u\rangle-\langle \ad_{\eb}^*e,\ad_{e}^*u\rangle
		-\sum_{i=1}^r\langle \ad_{u_i}^*e,\ad_{u_i}^*u\rangle=0
		\end{align*}
	since $\ad_{e}^*e=0$, $\ad_{u_i}^*e=\langle X,u_i\rangle e$ and $\langle \ad_{\eb}^*e,\ad_{e}^*u\rangle=\langle u,[e,\ad_{\eb}^*e]\rangle =0$.
	Moreover, $\langle [H,u],e\rangle=0$, $\langle [H,e],u\rangle=0$  and hence $\ric(e,u)=0$. In the same way we have $\ric(\eb,u)=0$.

	Let us compute $\ric(u,u)$ for $u\in\h$. 
	\begin{align*}
		\tr(\ad_{u}\circ\ad_{u})&=
		\langle [u,[u,e]],\eb\rangle+\langle [u,[u,\eb]],e\rangle
		\sum_{i=1}^r\langle [u,[u,u_i]],u_i\rangle
		=2\langle X,u\rangle^2+\tr(\ad_u^\h\circ\ad_u^\h) \\
		\tr(\ad_{u}^*\circ\ad_{u})&=
		2\langle [u,e],[u,\eb]\rangle+
		\sum_{i=1}^r\langle [u,u_i],[u,u_i]\rangle
		=-2\langle X,u\rangle^2+	\tr((\ad_{u}^\h)^*\circ\ad_{u}),\\
		\tr(J_{u}^2)
		&=-2\langle \ad_{e}^*u,\ad_{\eb}^*u\rangle
		-\sum_{i=1}^r\langle \ad_{u_i}^*u,\ad_{u_i}^*u\rangle=\tr((J_u^\h)^2).
	\end{align*}But $\langle \ad_{e}^*u,\ad_{\eb}^*u\rangle=\langle u,[e,\ad_{\eb}^*u]\rangle=0$ and $\ad_{u_i}^*u=(\ad_{u_i}^\h)^*u$.
Moreover, $H=H_0\in\h$ and hence $\ric(u,u)=\ric^\h(u,u)$.
	
\subsection{Proof of Proposition \ref{Jacobia3}} Recall that the bracket is given by 
\begin{equation*}
	\begin{cases}[e,\eb]=x e, [e,f]=0,[\eb,f]=ze-\frac12 xf,\\
		[e,u]=\langle X,u\rangle e+Au,\\
		[\eb,u]=\langle Y,u\rangle e-\langle X,u\rangle \eb+\langle T,u\rangle f+Bu,\\
		[f,u]=\langle U,u\rangle e+Cu,	
	\end{cases}	,  
\end{equation*}where $A,B,C:\h\too\h$ are skew-symmetric.
Let us write the Jacobi identity. For any $u,v\in\h$,

$\bullet$ $[[e,\eb],u]+[[\eb,u],e]+[[u,e],\eb]=0$ is equivalent to 
\begin{eqnarray*}
	0&=&x[e,u]-\langle X,u \rangle[\bar e,e]+\langle T,u \rangle[f,e]+[Bu,e]-\langle X,u \rangle[e,\bar e]-[Au,\bar e]\\
	&=& x\langle X,u \rangle e+xAu
	-\langle X,Bu \rangle e-ABu+\langle Y,Au \rangle e+BAu-\langle X,Au \rangle \bar{e} +\langle T,Au \rangle f.
\end{eqnarray*}
So
\[
xX +BX - AY=0,\:  AX = 0,\:  AT = 0 \esp  [A,B]=xA.
\]

$\bullet$ $[[e,f],u]+[[f,u],e]+[[u,e],f]=0$ is equivalent to 
\begin{eqnarray*}
	0&=& [Cu,e]-\langle X,u \rangle[e,f]-[Au,f]
	= -\langle X,Cu \rangle e-ACu+\langle U,Au \rangle e+CAu.
\end{eqnarray*}
Then \[ CX=AU \esp [A,C]=0.\] \eject

$\bullet$ $[[f,\eb],u]+[[\eb,u],f]+[[u,f],\eb]=0$ is equivalent to
\begin{eqnarray*}
	0&=& -z[e,u]+\frac12 x[f,u]+\langle Y,u \rangle[e,f]-\langle X,u \rangle[\eb,f]+[Bu,f]-\langle U,u \rangle[e,\eb]-[Cu,\eb]  \\
	&=& -z\langle X,u \rangle e-zAu+\frac12x \langle U,u \rangle e+\frac12x Cu-z\langle X,u \rangle e+\frac12x\langle X,u\rangle f -\langle U,Bu \rangle e-CBu\\
	& &-x\langle U,u\rangle e+\langle Y,Cu \rangle e-\langle X,Cu \rangle\eb+\langle T,Cu \rangle f+BCu\\
	&=& [B,C]u-zAu+\frac12 xCu+(-2z\langle X,u \rangle-\frac12x \langle U,u \rangle +\langle Y,Cu \rangle-\langle U,Bu \rangle)e \\
	& &+\left(\frac12x\langle X,u\rangle-\langle CT,u \rangle \right)f +\langle CX,u\rangle \eb.
\end{eqnarray*}
Then 
\[ [B,C]=zA-\frac12xC,\quad 2zX+\frac12xU+CY-BU=0, CT=\frac12xX\esp CX=0.\]

$\bullet$ $[[e,u],v]+[[u,v],e]+[[v,e],u]=0$ is equivalent to 	
\begin{eqnarray*}
	0&=& \langle X,u\rangle [e,v]+[Au,v]-\langle X,[u,v]\rangle e-A[u,v]-\langle X,v\rangle[e,u]-[Av,u]\\
	&=& \langle X,u\rangle Av-\langle X,[u,v]\rangle e-\langle X,v\rangle Au-A[u,v]+[Au,v]+[u,Av]\\
	&=& -A[u,v]+[Au,v]+[u,Av]-\langle X,v\rangle Au+\langle X,u\rangle Av-\langle J_Xu,v\rangle e.
\end{eqnarray*}
So
$$
A[u,v]=[Au,v]+[u,Av]-\langle X,v\rangle Au +\langle X,u\rangle Av\esp
J_X=0.
$$

$\bullet$ $[[\eb,u],v]+[[u,v],\eb]+[[v,\eb],u]=0$ is equivalent to	
\begin{eqnarray*}
	0&=& \langle Y,u\rangle[e,v]-\langle X,u\rangle[\bar e,v]+\langle T,u\rangle[f,v]+[Bu,v]-\langle Y,[u,v]\rangle e+\langle X,[u,v]\rangle\bar e-\langle T,[u,v]\rangle f\\
	& & -B[u,v]-\langle Y,v\rangle[e,u]+\langle X,v\rangle[\bar e,u]-\langle T,v\rangle[f,u]-[Bv,u]\\
	&=& \langle Y,u\rangle\langle X,v\rangle e +\langle Y,u\rangle Av-\langle X,u\rangle\langle Y,v\rangle e+\langle X,u\rangle\langle X,v\rangle\bar e-\langle X,u\rangle\langle T,v\rangle f-\langle X,u\rangle Bv\\
	& & +\langle T,u\rangle\langle U,v\rangle e+\langle T,u\rangle Cv-\langle Y,[u,v]\rangle e+\langle X,[u,v]\rangle \bar{e}-\langle T,[u,v]\rangle f-B[u,v]+[Bu,v]\\
	& & -\langle Y,v\rangle\langle X,u\rangle e-\langle Y,v\rangle Au +\langle X,v\rangle\langle Y,u\rangle e-\langle X,u\rangle\langle X,v\rangle\bar e+\langle X,v\rangle\langle T,u\rangle f+\langle X,v\rangle Bu\\
	& & -\langle T,v\rangle\langle U,u\rangle e -\langle T,v\rangle Cu+[u,Bv].
\end{eqnarray*}
\[ \langle Y,u\rangle\langle X,v\rangle
-\langle X,u\rangle\langle Y,v\rangle +\langle T,u\rangle\langle U,v\rangle-\langle Y,[u,v]\rangle-\langle Y,v\rangle\langle X,u\rangle+\langle X,v\rangle\langle Y,u\rangle 
-\langle T,v\rangle\langle U,u\rangle=0.\]

This gives 
$$\begin{cases}
	\langle X,[u,v]\rangle =0\\
	2\langle Y,u\rangle\langle X,v\rangle +\langle T,u\rangle\langle U,v\rangle -\langle Y,[u,v]\rangle -2\langle Y,v\rangle\langle X,u\rangle -\langle T,v\rangle\langle U,u\rangle  =0\\
	\langle X,u\rangle\langle T,v\rangle +\langle T,[u,v]\rangle -\langle X,v\rangle\langle T,u\rangle =0\\
	-B[u,v]+[Bu,v]+[u,Bv]+\langle Y,u\rangle Av-\langle X,u\rangle Bv+\langle T,u\rangle Cv-\langle Y,v\rangle Au+\langle X,v\rangle Bu-\langle T,v\rangle Cu=0.
\end{cases} $$
So
$$\begin{cases}
	J_X=0,\;
	J_Y=2Y\wedge X+T\wedge U,\;
	J_T=T\wedge X,\\
	B[u,v]=[Bu,v]+[u,Bv]+\langle Y,u\rangle Av-\langle X,u\rangle Bv+\langle T,u\rangle Cv-\langle Y,v\rangle Au+\langle X,v\rangle Bu-\langle T,v\rangle Cu.
\end{cases} $$
But, according to \eqref{relations}, $Y=\frac1{\rho}\left(2U-\frac3\rho X\right)$ and $T=\frac2\rho X-U$ so
\begin{align*}
	J_Y&=\frac2\rho J_U=\frac4{\rho}U\wedge X+
	\frac2\rho X\wedge U=\frac2{\rho}U\wedge X,\\
	J_T&=-J_U=-U\wedge X.
\end{align*}Finally, $[[\eb,u],v]+[[u,v],\eb]+[[v,\eb],u]=0$ is equivalent to
$$\begin{cases}
	J_X=0,\;
	J_U=U\wedge X,\;
	\\
	B[u,v]=[Bu,v]+[u,Bv]+\langle Y,u\rangle Av-\langle X,u\rangle Bv+\langle T,u\rangle Cv-\langle Y,v\rangle Au+\langle X,v\rangle Bu-\langle T,v\rangle Cu.
\end{cases} $$

$\bullet$ $[[f,u],v]+[[u,v],f]+[[v,f],u]=0$ is equivalent to 	
\begin{eqnarray*}
	0&=& \langle U,u\rangle[e,v]+[Cu,v]-\langle U,[u,v]\rangle e-C[u,v]-\langle U,v\rangle [e,u]-[Cv,u]    \\
	&=& \langle U,u\rangle \langle X,v\rangle e+\langle U,u\rangle Av -\langle U,[u,v]\rangle e-\langle U,v\rangle\langle X,u\rangle e-\langle U,v\rangle Au-C[u,v]+[Cu,v]+[u,Cv].
\end{eqnarray*}
Then 
$$\begin{cases}
	\langle U,u\rangle \langle X,v\rangle -\langle U,[u,v]\rangle -\langle U,v\rangle\langle X,u\rangle =0\\
	\langle U,u\rangle Av-\langle U,v\rangle Au-C[u,v]+[Cu,v]+[u,Cv]=0.
\end{cases} $$
So
\[J_U=U\wedge X \esp C[u,v]=[Cu,v]+[u,Cv]+\langle U,u\rangle Av-\langle U,v\rangle Au.\]
We get finally \eqref{jacobi}.

	\subsection{Proof of Proposition \ref{riccia3}}
	recall that
	\[ \begin{cases}[e,\eb]= [e,f]=0=[\eb,f]=0,\\
		[e,u]=\langle X,u\rangle e,\\
		[\eb,u]=\langle Y,u\rangle e-\langle X,u\rangle \eb+\langle T,u\rangle f+Bu,\\
		[f,u]=\langle U,u\rangle e+Cu,
	\end{cases} \] $B=\frac{\langle T,X\rangle}{\langle X,X\rangle}C$, $CX=CY=CT=CU=0$ and $C$ is a derivation.\\
	
	$\bullet$ $\ric(\eb,\eb)$.
	\begin{align*}
		\tr(\ad_{\eb}^2)&=\langle [\eb,[\eb,e]],\eb\rangle+\langle [\eb,[\eb,f]],f\rangle+\sum_{i=1}^r\langle [\eb,[\eb,u_i]],u_i\rangle=
		\tr(B^2),\\
		\tr(\ad_{\eb}^*\circ\ad_{\eb})&=\langle [\eb,f],[\eb,f]\rangle	+\sum_{i=1}^r\langle [\eb,u_i],[\eb,u_i]\rangle
		=-\tr(B^2)-2\langle X,Y\rangle+|T|^2,
	\end{align*}
	\begin{align*} \tr(J_{\eb}\circ J_{\eb})
		&=-2\langle \ad_{e}^*\eb,\ad_{\eb}^*\eb\rangle
		-\langle \ad_{f}^*\eb,\ad_{f}^*\eb\rangle
		-\sum_{i=1}^r\langle \ad_{u_i}^*\eb,\ad_{u_i}^*\eb\rangle.
	\end{align*}
	\begin{align*}
		\ad_{e}^*\eb&=\langle \ad_{e}^*\eb,\eb\rangle e+
		\langle \ad_{e}^*\eb,e\rangle \eb+\langle \ad_{e}^*\eb,f\rangle f+\sum_{i=1}^r
		\langle \ad_{e}^*\eb,u_i\rangle u_i
		= X,\\
		\ad_{\eb}^*\eb&=\langle \ad_{\eb}^*\eb,\eb\rangle e+
		\langle \ad_{\eb}^*\eb,e\rangle \eb+\langle \ad_{\eb}^*\eb,f\rangle f+\sum_{i=1}^r
		\langle \ad_{\eb}^*\eb,u_i\rangle u_i
		= Y,\\
		\ad_{f}^*\eb&=\langle \ad_{f}^*\eb,\eb\rangle e+
		\langle \ad_{f}^*\eb,e\rangle \eb+\sum_{i=1}^r
		\langle \ad_{f}^*\eb,u_i\rangle u_i
		=U,\\
		\ad_{u}^*\eb&=\langle \ad_{u}^*\eb,\eb\rangle e+
		\langle \ad_{u}^*\eb,e\rangle \eb+\langle \ad_{u}^*\eb,f\rangle f+\sum_{i=1}^r
		\langle \ad_{u}^*\eb,u_i\rangle u_i\\
		&=-\langle Y,u\rangle e-\langle X,u\rangle \eb	-\langle U,u\rangle f.
	\end{align*}
	So
	\[ \tr(J_{\eb}^2)=-2\langle X,Y\rangle-2\langle U,U\rangle-2\langle X,Y\rangle=-4\langle X,Y\rangle-2\langle U,U\rangle. \]
	So
	\[ \ric(\eb,\eb)=2\langle X,Y\rangle-\frac12|T|^2+\frac12\langle U,U\rangle+\langle [\eb,H^\h],\eb\rangle \]Thus
	\[ \ric(\eb,\eb)=2\langle X,Y\rangle-\frac12|T|^2+\frac12\langle U,U\rangle+\tr(\ad_Y^\h). \]
	
	$\bullet$ $\ric(f,\eb)$.
	
	\begin{align*}
		\tr(\ad_f\circ\ad_{\eb})&=\langle [f,[\eb,e]],\eb\rangle+\langle [f,[\eb,f]],f\rangle+\sum_{i=1}^r\langle [f,[\eb,u_i]],u_i\rangle
		=\tr(CB),\\
		\tr(\ad_{f}^*\circ\ad_{\eb})&=	\sum_{i=1}^r\langle [\eb,u_i],[f,u_i]\rangle
		=-\tr(CB)-\langle X,U\rangle,\\
	 \tr(J_{f}\circ J_{\eb})
		&=-\langle \ad_{e}^*f,\ad_{\eb}^*\eb\rangle
		-\langle \ad_{\eb}^*f,\ad_{e}^*\eb\rangle-\langle \ad_{f}^*f,\ad_{f}^*\eb\rangle
		-\sum_{i=1}^r\langle \ad_{u_i}^*f,\ad_{u_i}^*\eb\rangle.
	\end{align*}
	\begin{align*}
		\ad_{\eb}^*\eb&=\langle \ad_{\eb}^*\eb,\eb\rangle e+
		\langle \ad_{\eb}^*\eb,e\rangle \eb+\langle \ad_{\eb}^*\eb,f\rangle f+\sum_{i=1}^r
		\langle \ad_{\eb}^*\eb,u_i\rangle u_i
		=Y, \\
		\ad_{e}^*f&=\langle \ad_{e}^*f,\eb\rangle e+
		\langle \ad_{e}^*f,e\rangle \eb+\langle \ad_{e}^*f,f\rangle f+\sum_{i=1}^r
		\langle \ad_{e}^*f,u_i\rangle u_i
		=0,\\
		\ad_{\eb}^*f&=\langle \ad_{\eb}^*f,\eb\rangle e+
		\langle \ad_{\eb}^*f,e\rangle \eb+\langle \ad_{\eb}^*f,f\rangle f+\sum_{i=1}^r
		\langle \ad_{\eb}^*f,u_i\rangle u_i
		=T,\\
	\end{align*}
	\begin{align*}
		\ad_{e}^*\eb&=\langle \ad_{e}^*\eb,\eb\rangle e+
		\langle \ad_{e}^*\eb,e\rangle \eb+\langle \ad_{e}^*\eb,f\rangle f+\sum_{i=1}^r
		\langle \ad_{e}^*\eb,u_i\rangle u_i
		=X,\\
		\ad_{f}^*f&=\langle \ad_{f}^*f,\eb\rangle e+
		\langle \ad_{f}^*f,e\rangle \eb+\sum_{i=1}^r
		\langle \ad_{f}^*f,u_i\rangle u_i
		=0,\\
		\ad_{f}^*\eb&=\langle \ad_{f}^*\eb,\eb\rangle e+
		\langle \ad_{f}^*\eb,e\rangle \eb+\langle \ad_{f}^*\eb,f\rangle f+\sum_{i=1}^r
		\langle \ad_{f}^*\eb,u_i\rangle u_i
		=U,\\
		\ad_{u}^*\eb&=\langle \ad_{u}^*\eb,\eb\rangle e+
		\langle \ad_{u}^*\eb,e\rangle \eb+\langle \ad_{u}^*\eb,f\rangle f+\sum_{i=1}^r
		\langle \ad_{u}^*\eb,u_i\rangle u_i\\
		&=-\langle Y,u\rangle e-\langle X,u\rangle \eb	-\langle U,u\rangle f,\\
		\ad_{u}^*f&=\langle \ad_{u}^*f,\eb\rangle e+
		\langle \ad_{u}^*f,e\rangle \eb+\langle \ad_{u}^*f,f\rangle f+\sum_{i=1}^r
		\langle \ad_{u}^*f,u_i\rangle u_i
		=-\langle T,u\rangle e.
	\end{align*}
	So
	\[ \tr(J_{\eb}\circ J_f)=-2\langle X,T\rangle. \]
	Thus
	\[ \ric(f,\eb)=\frac12\langle X,T\rangle+\frac12\langle X,U\rangle+\frac12\langle[\eb,H],f\rangle+\frac12\langle[f,H],\eb\rangle. \]
	We have
	\[ \tr(\ad_e)=0,\tr(\ad_{\eb})=-x,\;\tr(\ad_f)=0,\tr(\ad_u)=\tr(\ad_u^\h). \]
	So
	\[ H=H^\h. \]
	On the other hand,  $T+U=\dfrac2\rho X$ then 
	\[ \ric(f,\eb)=\frac1\rho(|X|^2+\tr(\ad_X^\h)). \]

	$\bullet$ $\ric(e,f)$. 
		\begin{align*}
		\tr(\ad_f\circ\ad_{e})&=\langle [f,[e,\eb]],e\rangle+\langle [f,[e,f]],f\rangle+\sum_{i=1}^r\langle [f,[e,u_i]],u_i\rangle
		=\tr(CA),\\
		\tr(\ad_{f}^*\circ\ad_{e})&=	+\sum_{i=1}^r\langle [e,u_i],[f,u_i]\rangle
		=-\tr(AC),\\
	 \tr(J_{f}\circ J_{e})
		&=-\langle \ad_{e}^*f,\ad_{\eb}^*e\rangle
		-\langle \ad_{\eb}^*f,\ad_{e}^*e\rangle-\langle \ad_{f}^*f,\ad_{f}^*e\rangle
		-\sum_{i=1}^r\langle \ad_{u_i}^*f,\ad_{u_i}^*e\rangle.
	\end{align*}
	\begin{align*}
		\ad_{\eb}^*e&=\langle \ad_{\eb}^*e,\eb\rangle e+
		\langle \ad_{\eb}^*e,e\rangle \eb+\langle \ad_{\eb}^*e,f\rangle f+\sum_{i=1}^r
		\langle \ad_{\eb}^*e,u_i\rangle u_i
		=-X, \\
		\ad_{e}^*f&=\langle \ad_{e}^*f,\eb\rangle e+
		\langle \ad_{e}^*f,e\rangle \eb+\langle \ad_{e}^*f,f\rangle f+\sum_{i=1}^r
		\langle \ad_{e}^*f,u_i\rangle u_i
		=0,\\
		\ad_{\eb}^*f&=\langle \ad_{\eb}^*f,\eb\rangle e+
		\langle \ad_{\eb}^*f,e\rangle \eb+\langle \ad_{\eb}^*f,f\rangle f+\sum_{i=1}^r
		\langle \ad_{\eb}^*f,u_i\rangle u_i
		=T,\\
	\end{align*}
	\begin{align*}
		\ad_{e}^*e&=\langle \ad_{e}^*e,\eb\rangle e+
		\langle \ad_{e}^*e,e\rangle \eb+\langle \ad_{e}^*e,f\rangle f+\sum_{i=1}^r
		\langle \ad_{e}^*e,u_i\rangle u_i
		=0,\\
		\ad_{f}^*f&=\langle \ad_{f}^*f,\eb\rangle e+
		\langle \ad_{f}^*f,e\rangle \eb+\sum_{i=1}^r
		\langle \ad_{f}^*f,u_i\rangle u_i
		=0,\\
		\ad_{f}^*e&=\langle \ad_{f}^*e,\eb\rangle e+
		\langle \ad_{f}^*e,e\rangle \eb+\langle \ad_{f}^*e,f\rangle f+\sum_{i=1}^r
		\langle \ad_{f}^*e,u_i\rangle u_i
		=0,\\
		\ad_{u}^*e&=\langle \ad_{u}^*e,\eb\rangle e+
		\langle \ad_{u}^*e,e\rangle \eb+\langle \ad_{u}^*e,f\rangle f+\sum_{i=1}^r
		\langle \ad_{u}^*e,u_i\rangle u_i
		=\langle X,u\rangle e,\\
		\ad_{u}^*f&=\langle \ad_{u}^*f,\eb\rangle e+
		\langle \ad_{u}^*f,e\rangle \eb+\langle \ad_{u}^*f,f\rangle f+\sum_{i=1}^r
		\langle \ad_{u}^*f,u_i\rangle u_i
		=-\langle T,u\rangle e.
	\end{align*}
	So
	\[ \tr(J_{e}\circ J_f)=0\esp \ric(e,f)=0. \]

	$\bullet$ $\ric(e,\eb)$. 
	\begin{align*}
		\tr(\ad_{\eb}\circ\ad_{e})&=\langle [\eb,[e,\eb]],e\rangle+\langle [\eb,[e,f]],f\rangle+\sum_{i=1}^r\langle [\eb,[e,u_i]],u_i\rangle
		=\tr(BA),\\
		\tr(\ad_{\eb}^*\circ\ad_{e})&=\langle [e,\eb],[\eb,e]\rangle	+\sum_{i=1}^r\langle [e,u_i],[\eb,u_i]\rangle
		=-|X|^2-\tr(AB)
	\end{align*}
	\begin{align*} \tr(J_{\eb}\circ J_{e})
		&=-\langle \ad_{e}^*\eb,\ad_{\eb}^*e\rangle
		-\langle \ad_{\eb}^*\eb,\ad_{e}^*e\rangle-\langle \ad_{f}^*\eb,\ad_{f}^*e\rangle
		-\sum_{i=1}^r\langle \ad_{u_i}^*\eb,\ad_{u_i}^*e\rangle.
	\end{align*}
	\begin{align*}
		\ad_{\eb}^*e&=\langle \ad_{\eb}^*e,\eb\rangle e+
		\langle \ad_{\eb}^*e,e\rangle \eb+\langle \ad_{\eb}^*e,f\rangle f+\sum_{i=1}^r
		\langle \ad_{\eb}^*e,u_i\rangle u_i
		=-X, \\
		\ad_{e}^*\eb&=\langle \ad_{e}^*\eb,\eb\rangle e+
		\langle \ad_{e}^*\eb,e\rangle \eb+\langle \ad_{e}^*\eb,f\rangle f+\sum_{i=1}^r
		\langle \ad_{e}^*\eb,u_i\rangle u_i
		=X,\\
		\ad_{\eb}^*\eb&=\langle \ad_{\eb}^*\eb,\eb\rangle e+
		\langle \ad_{\eb}^*\eb,e\rangle \eb+\langle \ad_{\eb}^*\eb,f\rangle f+\sum_{i=1}^r
		\langle \ad_{\eb}^*\eb,u_i\rangle u_i
		=Y,\\
		\ad_{e}^*e&=\langle \ad_{e}^*e,\eb\rangle e+
		\langle \ad_{e}^*e,e\rangle \eb+\langle \ad_{e}^*e,f\rangle f+\sum_{i=1}^r
		\langle \ad_{e}^*e,u_i\rangle u_i
		=0,\\
		\ad_{f}^*\eb&=\langle \ad_{f}^*\eb,\eb\rangle e+
		\langle \ad_{f}^*\eb,e\rangle \eb+\sum_{i=1}^r
		\langle \ad_{f}^*\eb,u_i\rangle u_i
		=U,\end{align*}
	\begin{align*}
		\ad_{f}^*e&=\langle \ad_{f}^*e,\eb\rangle e+
		\langle \ad_{f}^*e,e\rangle \eb+\langle \ad_{f}^*e,f\rangle f+\sum_{i=1}^r
		\langle \ad_{f}^*e,u_i\rangle u_i
		=0,\\
		\ad_{u}^*e&=\langle \ad_{u}^*e,\eb\rangle e+
		\langle \ad_{u}^*e,e\rangle \eb+\langle \ad_{u}^*e,f\rangle f+\sum_{i=1}^r
		\langle \ad_{u}^*e,u_i\rangle u_i
		=\langle X,u\rangle e,\\
		\ad_{u}^*\eb&=\langle \ad_{u}^*\eb,\eb\rangle e+
		\langle \ad_{u}^*\eb,e\rangle \eb+\langle \ad_{u}^*\eb,f\rangle f+\sum_{i=1}^r
		\langle \ad_{u}^*\eb,u_i\rangle u_i\\	
		&=-\langle Y,u\rangle e-\langle X,u\rangle\eb-\langle U,u\rangle f.
	\end{align*}
	So
	\[ \tr(J_{\eb}\circ J_e)=2|X|^2\esp \ric(e,\eb)=0.\]

	$\bullet$ $\ric(f,f)$. 
	\begin{align*}
		\tr(\ad_{f}^2)&=\langle [f,[f,\eb]],e\rangle+\langle [f,[f,f]],f\rangle+\sum_{i=1}^r\langle [f,[f,u_i]],u_i\rangle
		=\tr(C^2),\\
		\tr(\ad_{f}^*\circ\ad_{f})&=2\langle [f,\eb],[f,e]\rangle	+\sum_{i=1}^r\langle [f,u_i],[f,u_i]\rangle
		=-\tr(C^2).
	\end{align*}
	\begin{align*} \tr(J_{f}\circ J_{f})
		&=-2\langle \ad_{e}^*f,\ad_{\eb}^*f\rangle
		-\langle \ad_{f}^*f,\ad_{f}^*f\rangle
		-\sum_{i=1}^r\langle \ad_{u_i}^*f,\ad_{u_i}^*f\rangle.
	\end{align*}
	\begin{align*}
		\ad_{\eb}^*f&=\langle \ad_{\eb}^*f,\eb\rangle e+
		\langle \ad_{\eb}^*f,e\rangle \eb+\langle \ad_{\eb}^*f,f\rangle f+\sum_{i=1}^r
		\langle \ad_{\eb}^*f,u_i\rangle u_i
		=T, \\
		\ad_{e}^*f&=\langle \ad_{e}^*f,\eb\rangle e+
		\langle \ad_{e}^*f,e\rangle \eb+\langle \ad_{e}^*f,f\rangle f+\sum_{i=1}^r
		\langle \ad_{e}^*f,u_i\rangle u_i
		=0,\\
		\ad_{f}^*f&=\langle \ad_{f}^*f,\eb\rangle e+
		\langle \ad_{f}^*f,e\rangle \eb+\langle \ad_{f}^*f,f\rangle f+\sum_{i=1}^r
		\langle \ad_{f}^*f,u_i\rangle u_i
		=0,\\
		\ad_{u}^*f&=\langle \ad_{u}^*f,\eb\rangle e+
		\langle \ad_{u}^*f,e\rangle \eb+\langle \ad_{u}^*f,f\rangle f+\sum_{i=1}^r
		\langle \ad_{u}^*f,u_i\rangle u_i
		=-\langle T,u\rangle e.
	\end{align*}
	So
	\[ \tr(J_{f}^2)=0\esp \ric(f,f)=0.\]

	$\bullet$ $\ric(e,u)$. 
	\begin{align*}
		\tr(\ad_u\circ\ad_{e})&=\langle [u,[e,\eb]],e\rangle+\langle [u,[e,f]],f\rangle+\sum_{i=1}^r\langle [u,[e,u_i]],u_i\rangle
		=\tr(\ad_u^\h\circ A)-\langle X,Au\rangle,\\
		\tr(\ad_{u}^*\circ\ad_{e})&=
		\langle [e,\eb],[u,e]\rangle	+\sum_{i=1}^r\langle [e,u_i],[u,u_i]\rangle
		=-\tr(A\circ\ad_u^\h),\\
	 \tr(J_{u}\circ J_{e})
		&=-\langle \ad_{e}^*u,\ad_{\eb}^*e\rangle
		-\langle \ad_{\eb}^*u,\ad_{e}^*e\rangle-\langle \ad_{f}^*u,\ad_{f}^*e\rangle
		-\sum_{i=1}^r\langle \ad_{u_i}^*u,\ad_{u_i}^*e\rangle.
	\end{align*}
	\begin{align*}
		\ad_{\eb}^*e&=\langle \ad_{\eb}^*e,\eb\rangle e+
		\langle \ad_{\eb}^*e,e\rangle \eb+\langle \ad_{\eb}^*e,f\rangle f+\sum_{i=1}^r
		\langle \ad_{\eb}^*e,u_i\rangle u_i
		=-X \\
		\ad_{e}^*u&=\langle \ad_{e}^*u,\eb\rangle e+
		\langle \ad_{e}^*u,e\rangle \eb+\langle \ad_{e}^*u,f\rangle f+\sum_{i=1}^r
		\langle \ad_{e}^*u,u_i\rangle u_i
		=0,\\
		\ad_{\eb}^*u&=\langle \ad_{\eb}^*u,\eb\rangle e+
		\langle \ad_{\eb}^*u,e\rangle \eb+\langle \ad_{\eb}^*u,f\rangle f+\sum_{i=1}^r
		\langle \ad_{\eb}^*u,u_i\rangle u_i
		=-Bu,\\
		\ad_{e}^*e&=\langle \ad_{e}^*e,\eb\rangle e+
		\langle \ad_{e}^*e,e\rangle \eb+\langle \ad_{e}^*e,f\rangle f+\sum_{i=1}^r
		\langle \ad_{e}^*e,u_i\rangle u_i
		=0,\\
		\ad_{f}^*u&=\langle \ad_{f}^*u,\eb\rangle e+
		\langle \ad_{f}^*u,e\rangle \eb+\sum_{i=1}^r
		\langle \ad_{f}^*u,u_i\rangle u_i
		=-Cu,\\
	\end{align*}
\begin{align*}
		\ad_{f}^*e&=\langle \ad_{f}^*e,\eb\rangle e+
		\langle \ad_{f}^*e,e\rangle \eb+\sum_{i=1}^r
		\langle \ad_{f}^*e,u_i\rangle u_i
		=0,\\
		\ad_{u_i}^*e&=\langle \ad_{u_i}^*e,\eb\rangle e+
		\langle \ad_{u_i}^*e,e\rangle \eb+\langle \ad_{u_i}^*e,f\rangle f+\sum_{j=1}^r
		\langle \ad_{u_i}^*e,u_j\rangle u_j
		=\langle X,u_i\rangle e,\\
		\ad_{u_i}^*u&=\langle \ad_{u_i}^*u,\eb\rangle e+
		\langle \ad_{u_i}^*u,e\rangle \eb+\langle \ad_{u_i}^*u,f\rangle f+\sum_{j=1}^r
		\langle \ad_{u_i}^*u,u_j\rangle u_j\\	
		&=-\langle u,Bu_i\rangle e-\langle u,Cu_i\rangle f+J_u^\h(u_i).
	\end{align*}
	So
	\[ \tr(J_u\circ J_e)=0\esp \ric(e,u)=0. \]

	$\bullet$ $\ric(\eb,u)$. 
		\begin{align*}
		\tr(\ad_u\circ\ad_{\eb})&=\langle [u,[\eb,e]],\eb\rangle+\langle [u,[\eb,f]],f\rangle+\sum_{i=1}^r\langle [u,[\eb,u_i]],u_i\rangle\\
		&=\langle X,Bu\rangle-\langle T,Cu\rangle+\tr(B\circ\ad_u^\h),\\
		\tr(\ad_{u}^*\circ\ad_{\eb})&=\langle [\eb,e],[u,\eb]\rangle+\langle [\eb,f],[u,f]\rangle	+\sum_{i=1}^r\langle [\eb,u_i],[u,u_i]\rangle\\
		&=-\tr(B\circ\ad_u^\h),\\
	 \tr(J_{u}\circ J_{\eb})
		&=-\langle \ad_{e}^*u,\ad_{\eb}^*\eb\rangle
		-\langle \ad_{\eb}^*u,\ad_{e}^*\eb\rangle-\langle \ad_{f}^*u,\ad_{f}^*\eb\rangle
		-\sum_{i=1}^r\langle \ad_{u_i}^*u,\ad_{u_i}^*\eb\rangle.
	\end{align*}
	\begin{align*}
		\ad_{\eb}^*\eb&=Y, \;
		\ad_{u}^*e=\langle X,u\rangle e,\\
		\ad_{\eb}^*u&=\langle \ad_{\eb}^*u,\eb\rangle e+
		\langle \ad_{\eb}^*u,e\rangle \eb+\langle \ad_{\eb}^*u,f\rangle f+\sum_{i=1}^r
		\langle \ad_{\eb}^*u,u_i\rangle u_i
		=-Bu,\
	\end{align*}
	\begin{align*}
		\ad_{e}^*{\eb}&=\langle \ad_{e}^*{\eb},\eb\rangle e+
		\langle \ad_{e}^*{\eb},e\rangle \eb+\langle \ad_{e}^*{\eb},f\rangle f+\sum_{i=1}^r
		\langle \ad_{e}^*{\eb},u_i\rangle u_i
		=X,\\
		\ad_{f}^*\eb&=\langle \ad_{f}^*\eb,\eb\rangle e+
		\langle \ad_{f}^*\eb,e\rangle \eb+\sum_{i=1}^r
		\langle \ad_{f}^*\eb,u_i\rangle u_i
		=U,\\
		\ad_{f}^*u&=\langle \ad_{f}^*u,\eb\rangle e+
		\langle \ad_{f}^*u,e\rangle \eb+\langle \ad_{f}^*u,f\rangle f+\sum_{i=1}^r
		\langle \ad_{f}^*u,u_i\rangle u_i
		=-Cu,\\
		\ad_{u}^*\eb&=\langle \ad_{u}^*\eb,\eb\rangle e+
		\langle \ad_{u}^*\eb,e\rangle \eb+\langle \ad_{u}^*\eb,f\rangle f+\sum_{i=1}^r
		\langle \ad_{u}^*\eb,u_i\rangle u_i\\	
		&=-\langle Y,u\rangle e-\langle X,u\rangle \eb-\langle U,u\rangle f,\\
		\ad_{u_i}^*u&=\langle \ad_{u_i}^*u,\eb\rangle e+
		\langle \ad_{u_i}^*u,e\rangle \eb+\langle \ad_{u_i}^*u,f\rangle f+\sum_{j=1}^r
		\langle \ad_{u_i}^*u,u_j\rangle u_j\\	
		&=-\langle u,Bu_i\rangle e-\langle u,Cu_i\rangle f+J_u^\h(u_i).
	\end{align*}
	So
	\[ \tr(J_{\eb}\circ J_u)=2\langle X,Bu\rangle+\langle U,Cu\rangle. \]
	\begin{align*}
		\tr(\ad_u\circ\ad_{\eb})&=\langle [u,[\eb,e]],\eb\rangle+\langle [u,[\eb,f]],f\rangle+\sum_{i=1}^r\langle [u,[\eb,u_i]],u_i\rangle\\
		&=\langle X,Bu\rangle-\langle T,Cu\rangle+\tr(B\circ\ad_u^\h),\\
		\tr(\ad_{u}^*\circ\ad_{\eb})&=\langle [\eb,e],[u,\eb]\rangle+\langle [\eb,f],[u,f]\rangle	+\sum_{i=1}^r\langle [\eb,u_i],[u,u_i]\rangle\\
		&=-\tr(B\circ\ad_u^\h).
	\end{align*}
	
	So
	\[ \ric(u,\eb)=-\langle X,Bu\rangle+\frac12\langle T,Cu\rangle
	-\frac14\langle U,Cu\rangle-\frac12\tr(\ad_{Bu}^\h)=0 \]since $BX=CT=CU=0$ and $B$ is a derivation, i.e, $\ad_{Bu}^\h=[B,\ad_u^\h]$.

	$\bullet$ $\ric(u,f)$. 
		\begin{align*}
		\tr(\ad_u\circ\ad_{f})&=\langle [u,[f,e]],\eb\rangle+\langle [u,[f,\eb]],e\rangle+\langle [u,[f,f]],f\rangle+\sum_{i=1}^r\langle [u,[f,u_i]],u_i\rangle\\
		&=\tr(\ad_u^\h\circ C),\\
		\tr(\ad_{u}^*\circ\ad_{f})&=\langle [f,e],[u,\eb]\rangle+\langle [f,\eb],[u,e]\rangle	+\sum_{i=1}^r\langle [f,u_i],[u,u_i]\rangle
		=-\tr(C\circ\ad_u^\h),\\
	 \tr(J_{u}\circ J_{f})
		&=-\langle \ad_{e}^*u,\ad_{\eb}^*f\rangle
		-\langle \ad_{\eb}^*u,\ad_{e}^*f\rangle-\langle \ad_{f}^*u,\ad_{f}^*f\rangle
		-\sum_{i=1}^r\langle \ad_{u_i}^*u,\ad_{u_i}^*f\rangle.
	\end{align*}
	\begin{align*}
		\ad_{\eb}^*f&=\langle \ad_{\eb}^*f,\eb\rangle e+
		\langle \ad_{\eb}^*f,e\rangle \eb+\langle \ad_{\eb}^*f,f\rangle f+\sum_{i=1}^r
		\langle \ad_{\eb}^*f,u_i\rangle u_i
		= T,      \\
		\ad_{e}^*u&=0\\
		\ad_{\eb}^*u&=\langle \ad_{\eb}^*u,\eb\rangle e+
		\langle \ad_{\eb}^*u,e\rangle \eb+\langle \ad_{\eb}^*u,f\rangle f+\sum_{i=1}^r
		\langle \ad_{\eb}^*u,u_i\rangle u_i
		=-Bu\\
	\ad_{e}^*{f}&=\langle \ad_{e}^*{f},\eb\rangle e+
		\langle \ad_{e}^*{f},e\rangle \eb+\langle \ad_{e}^*{f},f\rangle f+\sum_{i=1}^r
		\langle \ad_{e}^*{f},u_i\rangle u_i
		=0,\\
		\ad_{f}^*f&=
		\langle \ad_{f}^*f,\eb\rangle e+
		\langle \ad_{f}^*f,e\rangle \eb+\sum_{i=1}^r
		\langle \ad_{f}^*f,u_i\rangle u_i
		=0,\\
		\ad_{f}^*u&=\langle \ad_{f}^*u,\eb\rangle e+
		\langle \ad_{f}^*u,e\rangle \eb+\langle \ad_{f}^*u,f\rangle f+\sum_{i=1}^r
		\langle \ad_{f}^*u,u_i\rangle u_i
		=-Cu\\
		\ad_{u}^*f&=\langle \ad_{u}^*f,\eb\rangle e+
		\langle \ad_{u}^*f,e\rangle \eb+\langle \ad_{u}^*f,f\rangle f+\sum_{i=1}^r
		\langle \ad_{u}^*f,u_i\rangle u_i	
		=-\langle T,u\rangle e,\\
		\ad_{u_i}^*u&=\langle \ad_{u_i}^*u,\eb\rangle e+
		\langle \ad_{u_i}^*u,e\rangle \eb+\langle \ad_{u_i}^*u,f\rangle f+\sum_{j=1}^r
		\langle \ad_{u_i}^*u,u_j\rangle u_j\\	
		&=-\langle u,Bu_i\rangle e-\langle u,Cu_i\rangle f+J_u^\h(u_i)
	\end{align*}
	So
	\[ \tr(J_{f}\circ J_u)=0 \esp \ric(u,f)=-\frac12\tr(\ad_{Cu}^\h)=0 \]since $C$ is a derivation of $\h$.

	$\bullet$ $\ric(u,u)$. 
	\begin{align*}
		\tr(\ad_u\circ\ad_{u})&=\langle [u,[u,e]],\eb\rangle+\langle [u,[u,\eb]],e\rangle+\langle [u,[u,f]],f\rangle+\sum_{i=1}^r\langle [u,[u,u_i]],u_i\rangle\\
		&=2\langle u,X\rangle^2+\tr((\ad_u^\h)^2),\\
		\tr(\ad_{u}^*\circ\ad_{u})&=\langle [u,e],[u,\eb]\rangle+\langle [u,\eb],[u,e]\rangle+\langle[u,f],[u,f]\rangle	+\sum_{i=1}^r\langle [u,u_i],[u,u_i]\rangle\\
		&=-2\langle X,u\rangle^2+\langle Cu,Cu\rangle+\tr((\ad_u^\h)^*\circ\ad_u^\h)
	\end{align*}
	\begin{align*} \tr(J_{u}\circ J_{u})
		&=-2\langle \ad_{e}^*u,\ad_{\eb}^*u\rangle
		-\langle \ad_{f}^*u,\ad_{f}^*u\rangle
		-\sum_{i=1}^r\langle \ad_{u_i}^*u,\ad_{u_i}^*u\rangle.
	\end{align*}
	\begin{align*}
		\ad_{\eb}^*u&=-Bu,\;
		\ad_{e}^*u=-Au,\;
		\ad_{f}^*u=-Cu,\\
		\ad_{u_i}^*u&=\langle \ad_{u_i}^*u,\eb\rangle e+
		\langle \ad_{u_i}^*u,e\rangle \eb+\langle \ad_{u_i}^*u,f\rangle f+\sum_{j=1}^r
		\langle \ad_{u_i}^*u,u_j\rangle u_j\\	
		&=-\langle u,Bu_i\rangle e-\langle u,Au_i\rangle \eb-\langle u,Cu_i\rangle f+J_u^\h(u_i).
	\end{align*}
	
	So
	\[ \tr(J_{u}\circ J_u)=-4\langle Bu,Au\rangle-2\langle Cu,Cu\rangle+\tr((J_u^\h)^2) \]
	and
	\[ \ric(u,u)=\langle Bu,Au\rangle+\ric^\h(u,u). \]

\end{enumerate}

\end{document}